\documentclass[12pt,reqno,oneside]{amsbook}
\usepackage{amssymb,url}
\usepackage{color}
\usepackage{graphicx}
\numberwithin{equation}{chapter}

\theoremstyle{plain}
\newtheorem{thm}{Theorem}[chapter]
\newtheorem{theorem}[thm]{Theorem}
\newtheorem{lem}[thm]{Lemma}
\newtheorem{rem}[thm]{Remark}
\newtheorem{exmpl}[thm]{Example}
\newtheorem{corollary}[thm]{Corollary}
\newtheorem{prop}[thm]{Proposition}

\newcommand\diag{\operatorname{diag}}
\newcommand\ran{\operatorname{ran}}
\newcommand\rank{\operatorname{rank}}
\newcommand\tr{\operatorname{tr}}
\newcommand\re{\operatorname{Re}}

\newcommand{\C}{\mathbb{C}}
\newcommand{\N}{\mathbb{N}}
\newcommand{\Z}{\mathbb{Z}}
\newcommand{\R}{\mathbb{R}}
\newcommand{\D}{\mathbb{D}}
\newcommand{\T}{\mathbb{T}}

\DeclareMathOperator*{\Llim}{L-lim}

\newcommand{\irB}{\mathcal{B}}
\newcommand{\irC}{\mathcal{C}}

\newcommand{\irE}{\mathcal{E}}
\newcommand{\irF}{\mathcal{F}}
\newcommand{\irG}{\mathcal{G}}
\newcommand{\irH}{\mathcal{H}}
\newcommand{\irI}{\mathcal{I}}
\newcommand{\irJ}{\mathcal{J}}
\newcommand{\irK}{\mathcal{K}}
\newcommand{\irL}{\mathcal{L}}
\newcommand{\irM}{\mathcal{M}}
\newcommand{\irN}{\mathcal{N}}

\newcommand{\irP}{\mathcal{P}}

\newcommand{\irR}{\mathcal{R}}

\newcommand{\irT}{\mathcal{T}}

\newcommand{\irX}{\mathcal{X}}
\newcommand{\irY}{\mathcal{Y}}

\newcommand{\Chi}{\mathrm{Chi}}
\newcommand{\Par}{\mathrm{par}}
\newcommand{\Des}{\mathrm{Des}}
\newcommand{\Gen}{\mathrm{Gen}}
\newcommand{\Lea}{\mathrm{Lea}}
\newcommand{\roo}{\mathrm{root}}
\newcommand{\Sl}{S_{\underline{\lambda}}}

\newcommand{\Br}{\mathrm{Br}}

\newcommand{\HRule}{\rule{\linewidth}{0.5mm}}

\usepackage[top=2.5 cm, bottom=2.5 cm, left=3 cm, right=2.5 cm]{geometry}

\begin{document}

\thispagestyle{empty}

\begin{titlepage}
\begin{center}

% Upper part of the page. The '~' is needed because \\
% only works if a paragraph has started.
%\includegraphics[width=0.15\textwidth]{./logo}~\\[1cm]

% Title
\bigskip

\HRule \\[0.4cm]

{ \huge \bfseries Asymptotic behaviour of Hilbert space operators with applications \\[0.4cm] }

\HRule \\[1.5cm]

\bigskip

\textbf{\LARGE A Dissertation Presented for the \\ Doctor of Philosophy Degree}\\[0.5cm]

\bigskip

{\Large Gy\"orgy P\'al Geh\'er}

\bigskip
\bigskip

{\Large Supervisor:}

{\Large L\'aszl\'o K\'erchy}

{\Large Professor}

\vfill

{\Large Doctoral School in Mathematics and Computer Science} 

\smallskip

{\Large University of Szeged, Bolyai Institute} 

\smallskip

{\Large Department of Analysis}

\smallskip
\smallskip

{\large Szeged, 2014}

\end{center}
\end{titlepage}

\tableofcontents

\thispagestyle{plain}

\newpage

\chapter{Introduction}
The concepts of Hilbert spaces and bounded operators on them are very important in most parts of mathematical analysis and also in the mathematical foundation of quantum physics. 
One of the first important steps in this topic was to develop the theory of normal operators, to which many famous mathematicians contributed including for instance D. Hilbert and J. von Neumann. 
The main result concerning this class of operators is the well-known spectral theorem which gives us a very useful model of normal operators.
This theorem was proven by von Neumann in the 20th century. 
This is in fact the right generalization of the finite dimensional spectral theorem which was proven first by A. L. Cauchy in the 19th century. 
In view of this theory we have a very powerful tool if we would like to deal with normal operators.

On the other hand, the theory of non-normal operators is not very well-developed. 
However, many results have been obtained in the second half of the 20th century.
One of the main methods of examining non-normal operators, acting on a complex Hilbert space, is the theory of contractions.
We say that a bounded, linear operator $T\in\irB(\irH)$ is a contraction if $\|T\|\leq 1$ is satisfied, where $\irB(\irH)$ denotes the set of bounded operators acting on the complex Hilbert space $\irH$.
This area of operator theory was developed by B. Sz.-Nagy and C. Foias from the dilatation theorem of Sz.-Nagy what we state now.

\begin{theorem}[Sz.-Nagy's dilation theorem, \cite{SzN_dil}]\label{SzNdil_thm}
Let us consider a contraction $T\in\irB(\irH)$. There exists a larger Hilbert space $\irK\supseteq\irH$ and a unitary operator acting on it such that
\[
T^n = P_\irH U^n|\irH
\]
holds for all $n\in\N$ where $P_\irH\in\irB(\irK)$ denotes the orthogonal projection onto $\irH$. In this case we call $U$ a unitary dilation of $T$.

Moreover, there is a minimal unitary dilation in the sense that
\[
\bigvee_{n\in\Z} U^n\irH = \irK
\]
is fulfilled and such a minimal dilation is unique up to isomorphism. This $U$ is called the minimal unitary dilation of $T$.
\end{theorem}

Sz.-Nagy and Foias classified the contractions according to their asymptotic behaviour.
This classification can be done in a more general setting, namely when we consider the class of power bounded operators.
An operator $T\in\irB(\irH)$ is called power bounded if $\sup\{\|T^n\|\colon n\in\N\}<\infty$ holds.
We call a vector $x\in\irH$ stable for $T$ if $\lim_{n\to\infty}\|T^n x\| = 0$, and the set of all stable vectors is denoted by $\irH_0 = \irH_0(T)$.
It can be derived quite easily that $\irH_0$ is a hyperinvariant subspace for $T$ (see \cite{Ke_Hto}) which means that $\irH_0$ is a subspace that is invariant for every operator which commutes with $T$.
Therefore we will call it the stable subspace of $T$.
The natural classification of power bounded operators given by Sz.-Nagy and Foias is the following:
\begin{itemize}
\item $T$ is said to be of class $C_{1\cdot}$ or asymptotically non-vanishing if $\irH_0(T) = \{0\}$,
\item $T$ is said to be of class $C_{0\cdot}$ or stable if $\irH_0(T) = \irH$, i.e. when $T^n \to 0$ holds in the strong operator topology (SOT),
\item $T$ is said to be of class $C_{\cdot j}$ ($j\in\{0,1\}$) whenever $T^*$ is of class $C_{j\cdot}$,
\item the class $C_{jk}$ ($j,k\in\{0,1\}$) consists of those operators that are of class $C_{j \cdot}$ and $C_{\cdot k}$, simultaneously.
\end{itemize}
We will use the notation $T\in C_{jk}(\irH)$ ($j,k\in\{0,1,\cdot\}$). 

Trivially, if $T\notin C_{0\cdot}(\irH)\cup C_{1\cdot}(\irH)\cup C_{\cdot 0}(\irH)\cup C_{\cdot 1}(\irH)$, then $T$ has a non-trivial hyperinvariant subspace, namely $\irH_0(T)$ or $\irH_0(T^*)^\perp$.
Sz.-Nagy and Foias got strong structural results in the case when $T\in C_{11}(\irH)$ (see \cite{Ke_Hto,SzNF_C11}).
However, basic questions are still open (e.g. the hyperinvariant and the invariant subspace problems) in the remaining cases, i.e. when $T\in C_{00}(\irH)\cup C_{10}(\irH)\cup C_{01}(\irH)$.

Our aim in this dissertation is to explore the asymptotic behaviour of power bounded operators. 
Then we will present some applications, namely we will investigate similarity to normal operators, the commutant of non-stable contractions and cyclic properties of a recently introduced operator class, the weighted shift operators on directed trees.

Sz.-Nagy characterized those operators that are similar to unitary operators. 
This theorem belongs to the best known results concerning the study of Hilbert space operators that are similar to normal operators. 
The result is very elegant and, in its original form, reads as follows.

\begin{theorem}[Sz.-Nagy, \cite{SzN_unit}]\label{SzN_unit_thm}
An operator $T$ is similar to a unitary operator if and only if it is invertible and both $T$ and $T^{-1}$ are power bounded.
\end{theorem}

The ''$\Longrightarrow$'' direction is a trivial assertion. 
In order to verify the other direction, Sz.-Nagy defined a positive operator using Banach limits in an ingenious way. 
In what follows we will give this definition. 
The Banach space of bounded complex sequences is denoted by $\ell^\infty(\N)$. 
We call the linear functional 
\[ 
L\colon \ell^\infty(\N) \to \C,\quad \underline{x} = \{x_n\}_{n=1}^\infty \mapsto \Llim_{n\to\infty} x_n
\] 
a Banach limit if the following four conditions are satisfied:
\begin{itemize}
\item $\|L\| = 1$,
\item we have $\Llim_{n\to\infty} x_n = \lim_{n\to\infty} x_n$ for convergent sequences,
\item $L$ is positive, i.e. if $x_n\geq 0$ for all $n\in\N$, then $\Llim_{n\to\infty} x_n \geq 0$, and
\item $L$ is shift-invariant, i.e. $\Llim_{n\to\infty} x_n = \Llim_{n\to\infty} x_{n+1}$.
\end{itemize}
Note that a Banach limit is never multiplicative (see \cite[Section III.7]{Co} for further details).
We would like to point out that Sz.-Nagy was the first who effectively applied the notion of Banach limits, before that it was considered only as a mathematical curiosity.

For a power bounded $T\in\irB(\irH)$ let us consider the sequence of self-adjoint iterates $\{T^{*n}T^{n}\}_{n=1}^\infty$ and fix a Banach limit $L$. We consider the sesqui-linear form
\[
w_{T,L}\colon\irH\times\irH\to\C, \quad w_{T,L}(x,y) := \Llim_{n\to\infty} \langle T^{*n}T^nx,y \rangle 
\]
which is bounded and positive. 
Hence there exists a unique representing positive operator $A_{T,L}\in\irB(\irH)$ such that 
\[
w_{T,L}(x,y) = \langle A_{T,L}x,y \rangle \quad (x,y\in\irH).
\] 
The operator $A_{T,L}$ is called the $L$-asymptotic limit of the power bounded operator $T$, which usually depends on the particular choice of $L$. 
It is quite straightforward to show that $\ker A_{T,L} = \irH_0(T)$ is satisfied for every Banach limit $L$.
Sz.-Nagy showed that if $T$ and $T^{-1}$ are both power bounded, then $A_{T,L}$ is invertible and there exists a unitary operator $U\in\irB(\irH)$ such that
\[
A_{T,L}^{1/2}T = U A_{T,L}^{1/2}.
\]

In the case when $T$ is a contraction, the sequence of self-adjoint iterates $\{T^{*n}T^n\}_{n=1}^\infty$ is evidently decreasing. 
Thus it has a unique limit $A_T = A \in\irB(\irH)$ in SOT, which is clearly positive. 
This positive operator is said to be the asymptotic limit of the contraction $T$ which of course coincides with the $L$-asymptotic limit of $T$ for every $L$. 

The $L$-asymptotic limit gives us information about the asymptotic behaviour of the orbits $\{T^n x\}_{n=1}^\infty$, namely 
\begin{equation}\label{A_TL_orbit_eq}
\Llim_{n\to\infty} \|T^n x\|^2 = \|A_{T,L}^{1/2}x\|^2 = \|A_{T,L}^{1/2} Tx\|^2 \quad (x\in\irH).
\end{equation}

There is a certain reformulation of Theorem \ref{SzN_unit_thm} which reads as follows (see \cite{Kubrusly}).

\begin{theorem}[Sz.-Nagy]\label{Sz-N_ref_thm}
Consider an arbitrary operator $T\in\irB(\irH)$ and fix a Banach limit $L$. 
The following six conditions are equivalent:
\begin{itemize}
\item[\textup{(i)}] $T$ is similar to a unitary operator,
\item[\textup{(ii)}] $T$ is onto and similar to an isometry,
\item[\textup{(iii)}] $T$ is power bounded and there exists a constant $c > 0$ for which the inequalities $\|T^n x \| \geq c\|x\|$ and $\|T^{*n} x \| \geq c\|x\|$ hold for every $n\in\N$ and $x\in\irH$,
\item[\textup{(iv)}] $T$ is onto, power bounded and there exists a constant $c > 0$ for which the inequality $\|T^n x \| \geq c\|x\|$ holds for every $n\in\N$ and $x\in\irH$,
\item[\textup{(v)}] $T$ is power bounded and the $L$-asymptotic limits $A_{T,L}$ and $A_{T^*,L}$ are invertible,
\item[\textup{(vi)}] $T$ has bounded inverse and both $T^{-1}$ and $T$ are power bounded.
\end{itemize} 
Moreover, if we have an arbitrary power bounded operator $T\in\irB(\irH)$, then the next three conditions are also equivalent:
\begin{itemize}
\item[\textup{(i')}] $T$ is similar to an isometry,
\item[\textup{(ii')}] there exists a constant $c > 0$ for which the inequality $\|T^n x \| \geq c\|x\|$ holds for every $n\in\N$ and $x\in\irH$.
\item[\textup{(iii')}] the $L$-asymptotic limit $A_{T,L}$ is invertible,
\end{itemize} 
\end{theorem}

Sz.-Nagy's method naturally leads us to a little bit more general definition, the so-called isometric and unitary asymptote of a power bounded operator. We consider the operator $X^+_{T,L}\in\irB(\irH,\irH^+_T)$ where $\irH^+_T = (\ran A_{T,L})^- = (\ker A_{T,L})^\perp = \irH_0^\perp$ and $X^+_{T,L} x = A_{T,L}^{1/2} x$ holds for every $x\in\irH$. 
Since $\|X^+_{T,L} T x\| = \|X^+_{T,L} x\|$ is satisfied $(x\in\irH)$, there exists a unique isometry $V_{T,L} \in\irB(\irH^+_T)$ such that $X_{T,L}^+ T = V_{T,L} X_{T,L}^+$ holds. 
The operator $V_{T,L}$ (or sometimes the pair $(V_{T,L}, X^+_{T,L})$) is called the isometric asymptote of $T$. 
Let $W_{T,L}\in\irB(\irH_{T,L})$ be the minimal unitary dilation of $V_{T,L}$ and $X_{T,L}\in\irB(\irH,\irH_{T,L}), X_{T,L} x = X^+_{T,L} x$ for every $x\in\irH$. 
Obviously we have $X_{T,L} T = W_{T,L} X_{T,L}$. 
The operator $W_{T,L}$ (or sometimes the pair $(W_{T,L}, X_{T,L})$) is said to be the unitary asymptote of $T$. 
We note that a more general definition can be given, however, we will only need these canonical realizations (see \cite{BK}).
These asymptotes play an important role in the hyperinvariant subspace problem, similarity problems and operator models (see e.g. \cite{BK,Ca,Du,Ke_gen_Toep,Ke_isom_as,DuKu,Ku_A_proj,NFBK}).

When $T\notin C_{1\cdot}(\irH)\cup C_{0\cdot}(\irH)$, we have the following result which was first proven by Sz.-Nagy and Foias for contractions and by L. K\'erchy for power bounded operators.

\begin{lem}[K\'erchy \cite{Ke_isom_as}]\label{Ker_dec_lem}
Consider a power bounded operator $T\notin C_{1\cdot}(\irH)\cup C_{0\cdot}(\irH)$ and the orthogonal decomposition $\irH = \irH_0\oplus\irH_0^\perp$. 
The block-matrix form of $T$ in this decomposition is the following:
\begin{equation}\label{Ker_dec_eq}
T =
\left(\begin{matrix}
T_0 & R\\
0 & T_1
\end{matrix}\right) \in \irB(\irH_0\oplus\irH_0^\perp),
\end{equation}
where the elements $T_0$ and $T_1$ are of class $C_{0\cdot}$ and $C_{1\cdot}$, respectively.
\end{lem}

This is a very important structural result which will be applied several times throughout this dissertation.

The outline of the dissertation is as follows: in Chapter \ref{A-s_chap} we characterize those positive operators $A\in\irB(\irH)$ that arise from a contraction asymptotically, i.e. there is such a contraction $T\in\irB(\irH)$ for which $A_T = A$. 
Then, in Chapter \ref{matrix_A_C-s_chap}, we give a description of those positive semi-definite matrices $A\in\irB(\C^d)$ which arise from a power bounded matrix asymptotically, i.e. there exists a power bounded matrix $T \in \irB(\C^d)$ and a Banach limit $L$ such that $A_{T,L} = A$ holds. 
In fact we show that the matrices $A_{T,L}$ coincide for every $L$, moreover we will describe this operator as the limit of the Ces\'aro-sequence of the self-adjoint iterates of $T$.
Chapter \ref{SzN_chap} is devoted to the generalization of Theorem \ref{SzN_unit_thm}, concerning power bounded operators that are similar to normal operators. 
This gives us a method in certain cases which helps us to decide whether a given operator is similar to a normal one. 
Then, as a strengthening of Sz.-Nagy's result for contractions, we will describe those positive operators $A$ that arise from a contraction $T$ that is similar to a unitary operator. 
In Chapter \ref{comm_chap} we will investigate the case when the so-called commutant mapping of a contraction, which is a transformation that maps the commutant of a contraction into the commutant of its unitary asymptote, is injective. 
Finally, we will prove some cyclicity results in Chapter \ref{tree_chap} concerning weighted shift operators on directed trees which were recently introduced by Z. J. Jab\l onski, I. B. Jung and J. Stochel in \cite{JJS}. 
In each chapter there is a separate section for the main statements of our results and motivation, and another which is devoted to the proofs.

%---------------------------------------------------------------------------------------------------------------------------------

\newpage

\chapter{Positive operators arising asymptotically from contractions} \label{A-s_chap}

\section{Statements of the main results}
This chapter contains our results published in \cite{Ge_contr}. 
Throughout the chapter, if we do not state otherwise, it will be assumed that $T\in\irB(\irH)$ is a contraction. 
The main aim is to characterize those positive operators on a Hilbert space of arbitrary dimension which arise asymptotically from contractions. 
The following information concerning the asymptotic limits were known before or are trivial (see e.g. \cite{Kubrusly,NFBK}):
\begin{itemize}
\item[\textup{(i)}] $0\leq A_T \leq I$,
\item[\textup{(ii)}] $\ker(A_T) = \irH_0(T) := \{x\in\irH\colon \lim_{n\to\infty}\|T^nx\|=0\}$ is a hyperinvariant subspace,
\item[\textup{(iii)}] $\ker(A_T-I) = \irH_1(T) := \{x\in\irH\colon \lim_{n\to\infty}\|T^nx\|=\|x\|\}$ is the largest invariant subspace where $T$ is an isometry,
\item[\textup{(iv)}] $A_T = 0$ if and only if $T \in C_{0\cdot}(\irH)$,
\item[\textup{(v)}] $0 \notin \sigma_p(A_T)$ if and only if $T\in C_{1\cdot}(\irH)$.
\end{itemize}
We note that $\irH_1(T)\cap\irH_1(T^*)$ is the subspace on which the unitary part of $T$ acts.

Moreover, we provide the following further information about $\|A_T\|$.

\begin{theorem}\label{A_TL_norm_thm}
Suppose $L$ is a Banach limit and $T$ is a power bounded operator for which $A_{T,L} \neq 0$ holds. 
Then we have
\begin{equation}\label{A_TL_norm_eq}
\|A_{T,L}\| \geq 1. 
\end{equation}
In particular, $\|A_T\| = 1$ holds whenever $T$ is a contraction.
\end{theorem}

In finite dimension the asymptotic limit is always idempotent. 
The verification of this can be done with the Jordan decomposition theorem as well, but we will use the Sz.-Nagy--Foias--Langer decomposition instead.

\begin{theorem}\label{finite_dim_A_T_thm}
Let $T\in\irB(\C^d)$ be a contraction. 
Then $A_T = A_T^2 = A_{T^*}$, i.e. $A_T$ is simply the orthogonal projection onto the subspace $\irH_0(T)^\perp$ and $\irH_0(T) = \irH_0(T^*)$ is satisfied.
\end{theorem}

We say that the positive operator $A\in\irB(\irH)$ arises asymptotically from a contraction in uniform convergence if $\lim_{n\to\infty}\|T^{*n}T^n-A\| = 0$. 
Of course in this case $A = A_T$. 
On the other hand, it is easy to see that usually for a contraction $T\in\irB(\irH)$ the equation $\lim_{n\to\infty} T^{*n}T^n = A_T$ holds only in SOT.
Whenever $\dim \irH < \infty$, $A_T$ arises from $T$ in uniform convergence, since the SOT is equivalent to the usual norm topology. 
The symbols $\sigma_e$ and $r_e$ denote the essential spectrum and the essential spectral radius (see \cite[Section~XI.2]{Co}). 
Our result concerning the separable and infinite dimensional case reads as follows. 

\begin{theorem}\label{main_contr_separable_thm}
Let $\dim\irH=\aleph_0$ and $A$ be a positive contraction acting on $\irH$. 
The following four conditions are equivalent:
\begin{itemize}
\item[\textup{(i)}] $A$ arises asymptotically from a contraction,
\item[\textup{(ii)}] $A$ arises asymptotically from a contraction in uniform convergence,
\item[\textup{(iii)}] $r_e(A)=1$ or $A$ is a projection of finite rank,
\item[\textup{(iv)}] $\dim\irH(]0,1]) = \dim\irH(]\delta,1])$ holds for every $0\leq\delta<1$ where $\irH(\omega)$ denotes the spectral subspace of $A$ associated with the Borel subset $\omega\subseteq \R$.
\end{itemize}
Moreover, if one of these conditions holds and $\dim\ker(A-I)\in\{0,\aleph_0\}$, then $T$ can be chosen to be a $C_{\cdot 0}$-contraction such that it satisfies \textup{(ii)}.
\end{theorem}

We will also give the characterization in non-separable spaces. 
But before that we need a generalization of the essential spectrum. 
If $\kappa$ is an infinite cardinal number, satisfying $\kappa\leq\dim\irH$, then the closure of the set $\irE_\kappa := \{S\in\irB(\irH)\colon \dim (\irR(S))^-<\kappa\}$, is a proper two-sided ideal, denoted by $\irC_\kappa$. 
Let $\irF_\kappa := \irB(\irH)/\irC_\kappa$ be the quotient algebra, the mapping $\pi_\kappa\colon\irB(\irH)\to\irF_\kappa$ be the quotient map and $\|\cdot\|_\kappa$ the quotient norm on $\irF_\kappa$. 
For an operator $A\in\irB(\irH)$ we use the notations $\|A\|_\kappa := \|\pi_\kappa(A)\|_\kappa$, $\sigma_\kappa(A) := \sigma(\pi_\kappa(A))$ and $r_\kappa(A) := r(\pi_\kappa(A))$. 
(For $\kappa=\aleph_0$ we get the ideal of compact operators, $\|A\|_{\aleph_0} = \|A\|_e$ is the essential norm, $\sigma_{\aleph_0}(A)=\sigma_e(A)$ and $r_{\aleph_0}(A)=r_e(A)$.)
For more details see \cite{terElst} or \cite{Luft}. 
Now we state our result in the non-separable case.

\begin{theorem}\label{main_contr_arbitrary_thm}
Let $\dim\irH>\aleph_0$ and $A\in\irB(\irH)$ be a positive contraction. 
Then the following four conditions are equivalent:
\begin{itemize}
\item[\textup{(i)}] $A$ arises asymptotically from a contraction,
\item[\textup{(ii)}] $A$ arises asymptotically from a contraction in uniform convergence,
\item[\textup{(iii)}] $A$ is a finite rank projection, or $\kappa=\dim\irH(]0,1])\geq\aleph_0$ and $r_\kappa(A)=1$ holds,
\item[\textup{(iv)}] $\dim{\irH\left(]0,1]\right)}=\dim{\irH\left(]\delta,1]\right)}$ for any $0\leq\delta<1$.
\end{itemize}
Moreover, when $\dim\ker(A-I) \in \{0,\infty\}$ and (i) holds, then we can choose a $C_{\cdot 0}$ contraction $T$ such that $A$ is the uniform asymptotic limit of $T$.
\end{theorem}

The main tool in the proof of the above theorem is the following well-known property: any vector $h\in\irH$ generates a separable reducing subspace $\vee\{T^{*j_1}T^{k_1}\dots T^{*j_l}T^{k_l}h\colon l\in\N_0, j_1,k_1\dots j_l,k_l\in\N_0\}$ for $T$. 
Therefore $\irH$ can be decomposed into the orthogonal sum of separable reducing subspaces $\irH = \sum_{\xi\in \Xi}\oplus\irH_\xi$ and so $T=\sum_{\xi\in \Xi}\oplus T_\xi$, where $T_\xi=T|\irH_\xi$. 
Hence $A_T$ is the orthogonal sum of asymptotic limits of contractions, all acting on a separable space: $A_T = \sum_{\xi\in \Xi}\oplus A_{T_\xi}$.

It is natural to ask what condition on two contractions $T_1,T_2 \in\irB(\irH)$ imply $A_{T_1} = A_{T_2}$, or reversely what connection between $T_1$ and $T_2$ follows from the equation $A_{T_1} = A_{T_2}$. 
Finally, we will investigate this problem.

\begin{theorem}\label{main_contr_coinc_thm}
Let $\irH$ be an arbitrary Hilbert space and $T,T_1,T_2 \in\irB(\irH)$ contractions. 
The following statements are satisfied:
\begin{itemize}
\item[\textup{(i)}] if $T_1, T_2$ commute, then $A_{T_1T_2}\leq A_{T_1}$ and $A_{T_1T_2}\leq A_{T_2}$,
\item[\textup{(ii)}] if $u\in H^\infty$ is a non-constant inner function and $T$ is a c.n.u. contraction, then $A_T = A_{u(T)}$,
\item[\textup{(iii)}] $A_{T_1} = A_{T_2} = A$ implies $A \leq A_{T_1T_2}$,
\item[\textup{(iv)}] if $T_1$ and $T_2$ commute and $A_{T_1} = A_{T_2}$, then we have $A_{T_1T_2} = A_{T_1} = A_{T_2}$.
\end{itemize}
\end{theorem}

We will conclude this chapter by providing some examples.

\section{Proofs}

We begin with the proof of Theorem \ref{A_TL_norm_thm}.

\begin{proof}[Proof of Theorem \ref{A_TL_norm_thm}]
Assume that $0 < \|A_{T,L}\|$ happens. 
Set a vector $v\in\irH, \|v\| = 1$ and an arbitrarily small number $\varepsilon > 0$ such that 
\[ 
\Llim_{n\to\infty} \|T^n v\|^2 = \big\|A_{T,L}^{1/2} v\big\|^2 > \big(\big\|A_{T,L}^{1/2}\big\|-\varepsilon\big)^2 
\]
is satisfied. 
By \eqref{A_TL_orbit_eq} we get
\[ 
\Bigg\|A_{T,L}^{1/2} \frac{T^k v}{\|T^k v\|}\Bigg\|^2 = \frac{\|A_{T,L}^{1/2} v\|^2}{\|T^k v\|^2} > \frac{\big(\big\|A_{T,L}^{1/2}\big\|-\varepsilon\big)^2}{\|T^k v\|^2} \quad (k\in\N). 
\]
Since $\liminf_{k\to\infty} \|T^k v\|^2 \leq \Llim_{k\to\infty} \|T^k v\|^2 \leq \|A_{T,L}^{1/2}\|^2$, for every $\eta > 0$ there exists a $k_0\in\N$ for which $\|T^{k_0} v\|^2 \leq (\|A_{T,L}^{1/2}\|+\eta)^2$ holds. 
This suggests that
\[ 
\Bigg\|A_{T,L}^{1/2} \frac{T^{k_0} v}{\|T^{k_0} v\|}\Bigg\|^2 > \frac{(\|A_{T,L}^{1/2}\|-\varepsilon)^2}{\|T^{k_0} v\|^2} \geq \frac{(\|A_{T,L}^{1/2}\|-\varepsilon)^2}{\|A_{T,L}^{1/2}+\eta\|^2}. 
\]
Since this holds for every $\eta > 0$, we infer that
\[ 
\|A_{T,L}\| = \big\|A_{T,L}^{1/2}\big\|^2 \geq \frac{\big(\big\|A_{T,L}^{1/2}\big\|-\varepsilon\big)^2}{\big\|A_{T,L}^{1/2}\big\|^2}. 
\]
In the beginning we could choose an arbitrarily small $\varepsilon > 0$, hence we obtain \eqref{A_TL_norm_eq}.
\end{proof}

We proceed with the verification of Theorem \ref{finite_dim_A_T_thm}.

\begin{proof}[Proof of Theorem \ref{finite_dim_A_T_thm}]
Let us consider the Sz.-Nagy--Foias--Langer decomposition $T = U\oplus C \in\irB(\irH = \irN\oplus\irN^\perp)$ where $U\in\irB(\irN)$ is unitary and $C\in\irB(\irN^\perp)$ is c.n.u. 
If $\irN = \irH$, then obviously $A_T = A_{T^*} = I$. 
Therefore without loss of generality we may assume that this is not the case. 
Suppose that $r(C) = 1$. 
Then there exists a vector $0\neq v\in \ker(C-\lambda I)$ with a $\lambda\in\T$. 
We can find a $0\neq w\in(\ker(C-\lambda I))^\perp$, because $C$ is not unitary. 
Using the contractivity of $C$, we obtain
\[ 
\|v\|^2 + 2\re(\lambda\overline{\mu}\langle v,Cw\rangle) + |\mu|^2\|Cw\|^2 = \|\lambda v+\mu Cw\|^2 = \|Cv+\mu Cw\|^2 \leq \|v+\mu w\|^2 
\]
\[ 
= \|v\|^2 + 2\re(\overline{\mu}\langle v,w\rangle) + |\mu|^2\|w\|^2 = \|v\|^2 + |\mu|^2\|w\|^2 \quad (\mu\in\C), 
\]
or equivalently
\[ 
0 \leq |\mu|^2 (\|w\|^2-\|Cw\|^2) - 2\re(\lambda\overline{\mu}\langle v,Cw\rangle) \quad (\mu\in\C). 
\]
In particular we derive the following inequality:
\[ 
0 \leq t^2 (\|w\|^2 - \|Cw\|^2) - 2t|\langle v,Cw\rangle| \quad (t\in\R), 
\]
which implies $\langle v,Cw\rangle = 0$. 
Therefore any eigenspace of $C$ associated with a complex number of unit modulus has to be a reducing subspace, which is impossible since $C$ was c.n.u. 
Thus $r(C)<1$ follows, and e.g. by Gelfand's spectral radius formula (\cite[Proposition VII.3.8]{Co}) we infer $\|C^n\|\to 0$.
Hence we get $\irN^\perp = \irH_0(T) = \irH_0(T^*)$.
This implies that $A_T = I \oplus 0 = A_{T^*} \in\irB(\irN\oplus\irN^\perp)$ is the orthogonal projection onto $\irN = \irH_0^\perp$.
\end{proof}

Our next goal is to prove Theorem \ref{main_contr_separable_thm}, but before that we need three auxiliary results. 
The next lemma gives us some necessary conditions on the asymptotic limit of $T$.

\begin{lem} \label{nec_contr_lem}
If the positive operator $0\leq A\leq I$ is the asymptotic limit of a contraction $T$, then one of the following three possibilities occurs:
\begin{itemize}
\item[\textup{(i)}] $A=0$,
\item[\textup{(ii)}] $A$ is a non-zero finite rank projection, $\irH_0(T)^\perp = \irH_1(T)$ and $\dim{\irH_1(T)}\in\N$,
\item[\textup{(iii)}] $\|A\| = r_e(A) = 1$.
\end{itemize}
\end{lem}

\begin{proof}
As we have already seen $\|A_T\|\in\{0,1\}$.
Assume that (i) and (iii) do not hold.
Then 1 is an isolated point of $\sigma(A)$ and $\dim{\ker(A-I)}=\dim{\irH_1(T)}\in\N$. 
Since $T$ is an isometry on the finite dimensional invariant subspace $\irH_1(T)$, the contraction $T$ must be unitary on $\irH_1(T)$. 
Therefore $\irH_1(T)$ is reducing for $T$. 
Now consider the decomposition $T=T'\oplus U$ where $U=T|\irH_1(T)$. 
Evidently $A=A_{T'}\oplus I_{\irH_1(T)}$ and $\sigma(A_{T'})\subset [0,1[$ holds which implies $\|A_{T'}\|<1$, and thus $A_{T'}=0$. 
Therefore $\irH_0(T)=\irH_1(T)^{\perp}$, so $A$ is a projection onto $\irH_1$.
\end{proof}

Now, we would like to point out some remarks. 
It is easy to see that in a triangular decomposition
\[ 
T = \left(\begin{matrix}
T_1 & T_{12}\\
0 & T_2
\end{matrix}\right) \in \irB(\irH'\oplus\irH'').
\]
$A_T = A_{T_1}\oplus A_{T_2}$ does not hold in general. 
A counterexample can be given by the contractive bilateral weighted shift operator defined by
\begin{equation}\label{eq_shift_1_multiplic}
Te_k = \left\{\begin{matrix}
e_{k+1} & \text{ for } k>0\\
\frac{1}{2}e_{k+1} & \text{ for } k\leq 0\\
\end{matrix}\right. ,
\end{equation}
where $\{e_k\}_{k\in\Z}$ is an orthonormal basis in $\irH$. 
Indeed, an easy calculation shows that
\[ 
A_Te_k = \left\{\begin{matrix}
e_k & \text{ for } k>0\\
(\frac{1}{2})^{-2k+2}e_k & \text{ for } k\leq 0\\
\end{matrix}\right. . 
\]
On the other hand, $\irH_1(T)=\vee_{k>0}\{e_k\}$, and the matrix of $T$ is
\[ 
T = \left(\begin{matrix}
T_1 & T_{12}\\
0 & T_2
\end{matrix}\right)
\]
in the decomposition $\irH_1(T)\oplus\irH_1(T)^\perp$. 
Here $T_2\in C_{0\cdot}(\irH_1(T)^\perp)$, so $A_{T_1}\oplus A_{T_2} = I\oplus 0$ is a projection, but $A_T$ is not.

In fact $A_T$ is a projection if and only if $\irH = \irH_0(T)\oplus\irH_1(T)$ happens (\cite[Section~5.3]{Kubrusly}).
In this case $\irH_0(T)$ is a reducing subspace of $T$.

If $T$ has only trivial reducing subspaces, $\dim\irH\geq 2$ and $A_T$ is a projection, then either $T$ is stable or $T$ is a simple unilateral shift operator (i.e. all weights are 1 in \eqref{eq_shift_1_multiplic}). 
This follows directly from the previous remark, Theorem \ref{Sz-N_ref_thm} and the Neumann--Wold decomposition of isometries (\cite[p. 204]{Ke_Hto}).

Obviously for each power bounded operator $T\in\irB(\irH)$ and unitary operator $U\in\irB(\irH)$ we have
\begin{equation}\label{unitary_eqv_eq}
A_{UTU^*,L} = U A_{T,L} U^*.
\end{equation}

We proceed with the proof of the following lemma which provides a sufficient condition. 
For any positive operator $A\in\irB(\irH)$ we shall write $\underline{r}(A)$ for the minimal element of $\sigma(A)\subset[0,\infty[$.

\begin{lem}\label{block_diag_same_dim_lemma}
Suppose that the block-diagonal positive contraction $A = \sum_{j=0}^\infty\oplus A_j\in\irB(\irH = \sum_{j=0}^\infty\oplus\irX_j)$ has the following properties:
\begin{itemize}
\item[\textup{(i)}] $\dim{\irX_j}=\dim{\irX_0}>0$ for every $j$,
\item[\textup{(ii)}] $r(A_j) \leq \underline{r}(A_{j+1})$ for every $j$,
\item[\textup{(iii)}] $A_1$ is invertible, and
\item[\textup{(iv)}] $r(A_j)\nearrow 1$.
\end{itemize} 
Then $A$ arises asymptotically from a $C_{\cdot 0}$-contraction in uniform convergence. 
\end{lem}

\begin{proof} 
Let us consider the following operator-weighted unilateral shift operator $S\in\irB(\irH)$ which is given by the equation
\[ 
S (x_0\oplus x_1 \oplus x_2 \oplus\dots) = 0\oplus U_0 x_0\oplus U_1 x_1 \oplus \dots \quad (x_j\in\irX_j), 
\]
where $U_j\colon \irX_j\to\irX_{j+1}$ are unitary transformations (i.e. bijective isometries) ($j\in\N_0$). 
Let $T$ be defined by
\[ 
T|\irX_j = A_{j+1}^{-1/2}SA_{j}^{1/2} \quad (j\in\N_0).
\]
Since 
\[ 
\|A_{j+1}^{-1/2}SA_{j}^{1/2}x_{j}\| \leq \sqrt{\frac{1}{\underline{r}(A_{j+1})}} \|A_{j}^{1/2}x_{j}\| \leq \sqrt{\frac{r(A_{j})}{\underline{r}(A_{j+1})}} \|x_{j}\| \leq \|x_{j}\|, 
\]
we obtain that $T$ is a contraction of class $C_{\cdot 0}$. 
An easy calculation shows that
\[ 
T^{*n}T^n = \sum_{j=0}^\infty\oplus A_j^{1/2}S^{*n}A_{j+n}^{-1}S^nA_j^{1/2}. 
\]
By the spectral mapping theorem, we get
\[ 
\|A_j-A_j^{1/2}S^{*n}A_{j+n}^{-1}S^nA_j^{1/2}\| \leq \|A_j^{1/2}\|\cdot\|I_{\irX_j}-S^{*n}A_{j+n}^{-1}S^n\|\cdot\|A_j^{1/2}\| 
\] 
\[ 
\leq r(A_{j+n}^{-1}-I_{\irX_{j+n}}) \leq \frac{1}{\underline{r}(A_{j+n})}-1 \leq \frac{1}{\underline{r}(A_{n})}-1. 
\]
This yields
\[ 
\|T^{*n}T^n-A\| = \sup\left\{\|A_j-A_j^{1/2}S^{*n}A_{j+n}^{-1}S^nA_j^{1/2}\|\colon j\in\N_0\right\} 
\] 
\[ 
\leq \frac{1}{\underline{r}(A_{n})}-1 \longrightarrow 0 \quad (n\to\infty),
\]
which ends the proof.
\end{proof}

The last lemma before verifying Theorem \ref{main_contr_separable_thm} reads as follows.

\begin{lem}\label{diag_clust_to_1_lemma}
Let $A$ be a positive diagonal contraction on a separable infinite dimensional Hilbert space $\irH$. 
Suppose that the eigenvalues of $A$ can be arranged into an increasing sequence $\{\lambda_j\}_{j=1}^\infty\subseteq ]0,1[$, each listed according to its multiplicity such that $\lambda_j\nearrow 1$. 
Then $A$ is the uniform asymptotic limit of a $C_{\cdot 0}$-contraction.
\end{lem}

\begin{proof} First we form a matrix $(\alpha_{l,m})_{l,m\in\N}$ from the eigenvalues in the following way: $\alpha_{1,1}=\lambda_1$; $\alpha_{2,1}=\lambda_2$ and $\alpha_{1,2}=\lambda_3$; $\alpha_{3,1}=\lambda_4$, $\alpha_{2,2}=\lambda_5$ and $\alpha_{1,3}=\lambda_6$; \dots and so on.
\[ 
\left(
\begin{matrix}
\alpha_{1,1} & \alpha_{1,2} & \alpha_{1,3} & \alpha_{1,4} & \dots\\
\alpha_{2,1} & \alpha_{2,2} & \alpha_{2,3} \\
\alpha_{3,1} & \alpha_{3,2} \\
\alpha_{4,1} & & & \\
\vdots & & & & \ddots \\
\end{matrix}\right) = 
\left(\begin{matrix}
\lambda_1 & \lambda_3 & \lambda_6 & \lambda_{10} & \dots\\
\lambda_2 & \lambda_5 & \lambda_9 \\
\lambda_4 & \lambda_8 \\
\lambda_7 & & & \\
\vdots & & & & \ddots \\
\end{matrix}\right) 
\]
We can choose an orthonormal basis $\{e_{l,m}\colon l,m\in\N\}$ in $\irH$ such that $e_{l,m}$ is an eigenvector corresponding to the eigenvalue $\alpha_{l,m}$ of $A$. 
Now we form the subspaces:
\[ 
\irX_m := \vee\{e_{l,m}\colon l\in\N\} \quad (m\in\N), 
\]
which are clearly reducing for $A$. 
For any $m\in\N$, we set $A_m:=A|\irX_m$. 
Let us consider also the operator $S$ defined by $Se_{l,m}=e_{l,m+1}$ ($l,m\in\N$). 
We note that $S$ is the orthogonal sum of $\aleph_0$ many unilateral shift operators.
Now the operator $T\in\irB(\irH)$ is given by the following equality:
\[ 
T|\irX_m = A_{m+1}^{-1/2}SA_m^{1/2} \quad (m\in\N). 
\]
Since 
\[ 
T e_{l,m} = \sqrt{\frac{\alpha_{l,m}}{\alpha_{l,m+1}}}e_{l,m+1} \quad (l,m\in\N), 
\]
$T$ is a $C_{\cdot 0}$-contraction. 
Furthermore, for every $l,m,n\in\N$, we have $\lambda_n\leq \alpha_{l,m+n}$, and so
\[ 
\|T^{*n}T^{n} e_{l,m} - A e_{l,m} \| = \frac{\alpha_{l,m}}{\alpha_{l,m+n}}-\alpha_{l,m} \leq \frac{1}{\lambda_{n}}-1 \to 0 \quad (n\to\infty). 
\]
Since $e_{l,m}$ is an eigenvector for both $A$ and $T^{*n}T^{n}$, the sequence $T^{*n}T^{n}$ uniformly converges to $A$ on $\irH$. 
So $A$ arises asymptotically from a $C_{\cdot 0}$-contraction in uniform convergence.
\end{proof}

Now we are in a position to prove the characterization when $\dim\irH = \aleph_0$. 
This states that a positive contraction, which acts on a separable space, is an asymptotic limit of a contraction if and only if one of the conditions (i)--(iii) of Lemma \ref{nec_contr_lem} holds. 
In what follows, $E$ stands for the spectral measure of the positive operator $A$ and $\irH(\omega)=E(\omega)\irH$ for any Borel subset $\omega\subset\R$. 
Let us consider the orthogonal decomposition $\irH=\irH_{d}\oplus\irH_c$, reducing for $A$, where $A|\irH_{d}$ is diagonal and $A|\irH_{c}$ has no eigenvalue. 
Let us denote the spectral measure of $A|\irH_{d}$ and $A|\irH_{c}$ by $E_d$ and $E_c$, respectively. 
For any Borel set $\omega\subset\R$ we shall write $\irH_c(\omega)=E_c(\omega)\irH_c$ and $\irH_{d}(\omega)=E_{d}(\omega)\irH_d$.

\begin{proof}[Proof of Theorem \ref{main_contr_separable_thm}] 
The implication (i)$\Longrightarrow$(iii) follows from Lemma \ref{nec_contr_lem}, and (ii)$\Longrightarrow$(i) is trivial. 
First we prove the implication (iii)$\Longrightarrow$ (ii) in order to complete the implication circle (i)$\Longrightarrow$(iii)$\Longrightarrow$(ii)$\Longrightarrow$(i), and second we show the equivalence (iii)$\iff$(iv).

\smallskip

(iii)$\Longrightarrow$ (ii):
We suppose that $r_e(A)=1$. 
(If $A$ is a finite rank projection, then $T=A$ can be chosen.) 
If $\ker(A)\neq\{0\}$, then $A$ has the form $A=0\oplus A_1$ in the decomposition $\irH = \ker(A)\oplus\ker(A)^\perp$, where $r_e(A_1)=1$. 
If $A_1$ arises asymptotically from the contraction $T_1$ in uniform convergence, then $A$ arises asymptotically from $0\oplus T_1$ in uniform convergence. 
Hence we may assume that $\ker(A)=\{0\}$. 
Obviously, one of the next three cases occurs.

\smallskip

\textit{Case 1. There exists a strictly increasing sequence $0=a_0<a_1<a_2<\dots$ such that $a_n\nearrow 1$ and $\dim\irH([a_n,a_{n+1}[)=\aleph_0$ for every $n\in\N_0$.}
If 1 is not an eigenvalue of $A$, then Lemma \ref{block_diag_same_dim_lemma} can be applied. 
So we may suppose that $\dim\ker(A-I)\geq 1$. 
In this case we have the orthogonal decomposition: $A = A_0\oplus A_1$, where $A_0 = A|\ker(A-I)^\perp$ and $A_1 = A|\ker(A-I)$. 
Again using Lemma \ref{block_diag_same_dim_lemma} we obtain a contraction $T_0\in\irB(\ker(A-I)^\perp)$ such that the uniform asymptotic limit of $T_0$ is $A_0$. 
Choosing any isometry $T_1\in\irB(\ker(A-I))$, $A$ arises asymptotically from $T:=T_0\oplus T_1$ in uniform convergence.

\smallskip

\textit{Case 2. $\ker(A-I) = \{0\}$ and there is no strictly increasing sequence $0=a_0<a_1<a_2<\dots$ such that $a_n\nearrow 1$ and $\dim\irH([a_n,a_{n+1}[)=\aleph_0$ for every $n\in\N_0$.} 
If $\dim\irH([0,\beta[)<\aleph_0$ for each $0<\beta<1$, then $A$ is diagonal, all eigenvalues are in ]0,1[ and have finite multiplicities. Therefore Lemma \ref{diag_clust_to_1_lemma} can be applied. 
If this is not the case, then there is a $0<b<1$ which satisfies the following conditions: $\dim\irH([0,b[)=\aleph_0$ and $\dim\irH([b,\beta[)<\aleph_0$ for all $b<\beta<1$. 
We take the decomposition $\irH = \irH([0,b[)\oplus\irH([b,1[)$, where $\dim\irH([b,1[)=\aleph_0$ obviously holds, since $1\in\sigma_e(A)$. 
In order to handle this case, we have to modify the argument applied in Lemma \ref{diag_clust_to_1_lemma}.

Let us arrange the eigenvalues of $A$ in $[b,1[$ in an increasing sequence $\{\lambda_j\}_{j=1}^\infty$, each listed according to its multiplicity. 
We form the same matrix $(\alpha_{l,m})_{l,m\in\N}$ as in Lemma \ref{diag_clust_to_1_lemma}, and take an orthonormal basis $\{e_{l,m}\colon l,m\in\N\}$ in $\irH([b,1[)$ such that each $e_{l,m}$ is an eigenvector corresponding to the eigenvalue $\alpha_{l,m}$ of $A$. 
Let $\irX_0:=\irH([0,b[)$ and $\irX_m:=\vee\{e_{l,m}\colon l\in\N\}$ ($m\in\N$). 
Take an arbitrary orthonormal basis $\{e_{l,0}\}_{l=1}^\infty$ in the subspace $\irX_0$. 
We define the operator $T$ by the following equation:
\[ 
T|\irX_m = A_{m+1}^{-1/2}SA_{m}^{1/2}|\irX_m \quad (m\in\N_0),
\]
where $A_m:=A|\irX_m$ and $S\in\irB(\irH)$, $Se_{l,m} = e_{l,m+1}$ ($l\in\N, m\in\N_0$).

For a vector $x_0\in\irX_0$ we have
\[ 
\|Tx_0\| = \|A_1^{-1/2}SA_0^{1/2}x_0\| \leq \sqrt{\frac{1}{b}}\|A_0^{1/2}x_0\|\leq \|x_0\|, 
\]
therefore $T$ is a contraction on $\irX_0$. 
But it is also a contraction on $\irX_0^\perp$ (see the proof of Lemma~\ref{diag_clust_to_1_lemma}), and since 
$T\irX_0\perp T(\irX_0^\perp)$, it is a contraction on the whole $\irH$.

We have to show yet that $T^{*n}T^n$ converges uniformly to $A$ on $\irX_0$. 
For $x_0\in\irX_0, \|x_0\|=1$ we get
\[ 
\|T^{*n}T^{n} x_0 - A x_0\| = \|A_0^{1/2}S^{*n}(A_n^{-1}-I_{\irX_n})S^{n}A_0^{1/2} x_0\| 
\]
\[ 
\leq \|A_n^{-1} - I_{\irX_n}\| < \frac{1}{\lambda_{n}}-1\to 0. 
\]
So $A$ arises asymptotically from $T$ in uniform convergence.

\smallskip

\textit{Case 3. $\dim{\ker(A-I)}>0$.} 
If $\dim{\ker(A-I)}<\aleph_0$, then we take the orthogonal decomposition $\irH = \ker(A-I)^\perp\oplus\ker(A-I)$. 
Trivially $1\in\sigma_e(A|\ker(A-I)^\perp)$. 
By Cases 1 and 2, we can find a contraction $T_0\in\irB(\ker(A-I)^\perp)$ such that the uniform asymptotic limit of $T=T_0\oplus I_{\ker(A-I)}$ is $A$.

If $\dim{\ker(A-I)}=\aleph_0$ and $A\neq I$, then we consider an orthogonal decomposition $\ker(A-I) = \sum_{j=1}^\infty\oplus \irX_j$, where $\dim\irX_j=\dim\ker(A-I)^\perp$, and apply Lemma \ref{block_diag_same_dim_lemma}. 
If $A=I$, then just take an isometry for $T$.

\smallskip

(iii)$\iff$(iv): 
If $A$ is a projection, then $\dim(]\delta,1])$ is the rank of $A$ for every $0\leq\delta< 1$. 
If $1\in\sigma_e(A)$, then $\dim\irH(]\delta,1])=\aleph_0$ holds for all $0\leq\delta< 1$. 
Conversely, if the quantity $\dim\irH(]\delta,1])=\dim\irH(]0,1])$ is finite ($0\leq\delta< 1$), then obviously $A$ is a projection of finite rank. 
If this dimension is $\aleph_0$, then clearly $1\in\sigma_e(A)$.

Finally, from the previous discussions we can see that if the equivalent conditions (i)--(iv) hold, then the contraction $T$, inducing $A$, can be chosen from the class $C_{\cdot 0}$ provided $\dim\ker(A-I)\notin\N$.
\end{proof}

Now we turn to the case when $\dim\irH>\aleph_0$. 
As was mentioned earlier, for every contraction $T\in\irB(\irH)$ the space $\irH$ can be decomposed into the orthogonal sum of separable $T$-reducing subspaces.

\begin{proof}[Proof of Theorem \ref{main_contr_arbitrary_thm}] We may suppose that $A$ is not a projection of finite rank. 
Since $T=\sum_{\xi\in \Xi}\oplus T_\xi$, where every $T_\xi$ acts on a separable space, the (i)$\Longrightarrow$(iv) part can be proven by Theorem \ref{main_contr_separable_thm} straightforwardly. 
The implication (ii)$\Longrightarrow$(i) is obvious. First we will prove the direction (iv)$\Longrightarrow$(ii) and then the equivalence (iii)$\iff$(iv).

\smallskip

(iv)$\Longrightarrow$(ii): 
Set $\alpha=\dim\irH(]0,1])$, which is necessarily infinite. 
If $\alpha=\aleph_0$, then applying Theorem \ref{main_contr_separable_thm} we can get $A$ as the uniform asymptotic limit of a contraction (on $\ker A$ we take the zero operator). 
Therefore we may suppose that $\alpha>\aleph_0$. 
We may also assume that $A$ is injective. 
Now we take an arbitrary strictly increasing sequence $0 = a_0 < a_1 < a_2 < \dots$ such that $\lim_{j\to\infty}a_j=1$, and let $\alpha_j = \dim\irH(]a_j,a_{j+1}])$ for every $j\in\N_0$. 
Obviously $\beta := \sum_{j=0}^\infty \alpha_j =\dim\irH(]0,1[)\leq\dim\irH(]0,1])=\alpha$. 
Clearly one of the following four cases occurs.

\smallskip

\textit{Case 1: $\alpha_j=\alpha$ for infinitely many indices $j$.}
Then without loss of generality, we may suppose that this holds for every index $j$. 
By Lemma \ref{block_diag_same_dim_lemma} we can choose a contraction $T_0\in \irB(\irH(]0,1[))$ such that $\|T_0^{*n}T_0^{n}-A|\irH(]0,1[)\| \to 0$. 
Set the operator
\[
T := T_0\oplus V\in \irB\Big(\irH\big(]0,1[\big)\oplus\irH\big(\{1\}\big)\Big),
\] 
where $V\in\irB(\irH\big(\{1\}\big))$ is an arbitrary isometry. 
Trivially $T$ is a contraction with the uniform asymptotic limit $A$.

\smallskip

\textit{Case 2: $\dim\irH(\{1\}) = \alpha$.} 
Let us decompose $\irH(\{1\})$ into the orthogonal sum $\irH(\{1\}) = \big(\sum_{k=1}^\infty\oplus \irX_k\big) \oplus \irX$, where $\dim\irX_k=\beta$ for every $k\in\N$ and $\dim \irX = \alpha$. 
Setting $\irX_0 := \irH([0,1[)$, we may apply Lemma \ref{block_diag_same_dim_lemma} for the restriction of $A$ to $\sum_{k=0}^\infty\oplus \irX_k$. 
Taking any isometry on $\irX$, we obtain that (ii) holds.

\smallskip

\textit{Case 3: $\dim\irH(\{1\})<\alpha$ and $\alpha_j<\alpha$ for every $j$.} 
Then clearly $\dim\irH(]\delta,1[) = \dim\irH(]0,1[) = \alpha$ for any $\delta\in[0,1[$. 
Joining subintervals together, we may assume that $\aleph_0\leq \alpha_j<\alpha_{j+1}$ holds for every $j\in\N_0$ and $\sup_{j\geq 0}\alpha_j = \alpha$. 
Let $\irX_j := \irH(]a_{j},a_{j+1}])$ for every $j\in\N_0$. 
Obviously we can decompose every subspace $\irX_j$ into an orthogonal sum $\irX_j = \sum_{k=0}^j \oplus \irX_{j,k}$ such that $\dim\irX_{j,k} = \alpha_{k}$ for every $0 \leq k \leq j$. 
Then by Lemma \ref{block_diag_same_dim_lemma} we obtain a contraction $T_k \in \irB(\sum_{j=k}^\infty \oplus \irX_{j,k})$ such that the asymptotic limit of $T_k$ is $A\big|\sum_{j=k}^\infty \oplus \irX_{j,k}$ in uniform convergence. 
In fact, from the proof of Lemma \ref{block_diag_same_dim_lemma}, one can see that 
\[ 
\Bigg\|T_k^{*n}T_k^n - A\bigg|\sum_{j=k}^\infty\oplus \irX_{j,k}\Bigg\| \leq \frac{1}{a_{n+k}}-1 \leq \frac{1}{a_{n}}-1\to 0. 
\] 
Therefore, if we choose an isometry $V\in\irB(\irH(\{1\}))$, we get that (ii) is satisfied with the contraction $T := \big(\sum_{k=0}^\infty\oplus T_k\big)\oplus V \in\irB(\irH)$.

\smallskip

\textit{Case 4: $\dim\irH(\{1\})<\alpha$ and $\alpha_j = \alpha$ holds for finitely many $j$ (but at least for one).} 
We may assume $\alpha_0 = \alpha$, $\aleph_0\leq\alpha_j<\alpha_{j+1}$ for every $j\in\N$ and $\sup_{j\geq 1}\alpha_j = \alpha$. 
Take an orthogonal decomposition $\irH(]0,a_1[) = \sum_{k=1}^\infty \oplus \irL_k$, where $\dim\irL_k = \alpha_k$. 
Set also $\irX_j := \irH(]a_{j},a_{j+1}])$ for every $j\in\N$ and take a decomposition $\irX_j = \sum_{k=1}^j \oplus \irX_{j,k}$ such that $\dim\irX_{j,k} = \alpha_{k}$ for every $1 \leq k \leq j$. 
Thus by Lemma \ref{block_diag_same_dim_lemma} we obtain a contraction $T_k \in \irB(\irL_k\oplus\sum_{j=k}^\infty \oplus \irX_{j,k})$ such that the asymptotic limit of $T_k$ is the restriction of $A $ to the subspace $\irL_k\oplus\sum_{j=k}^\infty \oplus \irX_{j,k}$ in uniform convergence. 
As in Case 3, we get (ii).

\smallskip

(iii)$\Longrightarrow$(iv):
Since $1 = r_\kappa(A) \leq \|A\|_\kappa \leq \|A\|\leq 1$, we have $\|A\|_\kappa = 1$. 
An application of \cite[Lemma 5]{terElst} gives us $\dim\irH(]\delta,1]) = \kappa$ $(0\leq\delta<1)$.

\smallskip

(iv)$\Longrightarrow$(iii):
We may assume that $\dim\irH(]0,1])\geq\aleph_0$. 
Again applying \cite[Lemma 5]{terElst}, we get $\|A^n\|_\kappa=1$ for all $n\in\N$. 
This means that $r_\kappa(A)=1$.

\smallskip

Finally, we notice that if $\dim\ker(A-I)\notin\N$, then we can choose a $C_{\cdot 0}$-contraction.
\end{proof}

The following corollary is a straightforward consequence. 
In particular, this corollary gives that whenever $A$ is an asymptotic limit of a contraction, then $A^q$ is also an asymptotic limit of a contraction for every $0 < q$.

\begin{corollary}
Suppose that the function $g\colon[0,1]\to[0,1]$ is continuous, increasing, $g(0)=0$, $g(1)=1$ and $0<g(t)<1$ for $0<t<1$. 
If $A$ arises asymptotically from a contraction, then so does $g(A)$.
\end{corollary}

Next we verify Theorem \ref{main_contr_coinc_thm}, but before that we need the following lemma. 
For every $a\in\D$ the function $b_a\colon \D\to\D,\; b_a(z) = \frac{z-a}{1-\overline{a}z}$ is the so called M\"obius transformation which is a special type of inner functions. 
Moreover, it is a Riemann mapping from $\D$ onto itself. 
We use the notation $T_a:=b_a(T)$ where we use the well-known Sz.-Nagy--Foias functional calculus, which is a contractive, weak-* continuous algebra homomorphism $\Phi_T\colon H^\infty \to \irB(\irH)$ such that $\Phi_T(1) = I$ and $\Phi_T(\chi) = T$ ($\chi(z) = z$) is satisfied.
For further details see \cite{NFBK}.
It is easy to see that $b_{-a}(T_a) = b_{-a}(b_{a}(T)) = (b_{-a}\circ b_{a})(T) = T$.

\begin{lem} \label{cnu_Mobius}
For every c.n.u. contraction $T$ we have $A_T=A_{T_a}$
\end{lem}

\begin{proof}
Consider the isometric asymptotes $(X^+_T,V_T)$ and $(X^+_{T_a},V_{T_a})$ of $T$ and $T_a$, respectively.
Obviously $(X^+_T,b_a(V_T))$ is a contractive intertwining pair for $T_a$, hence we have a unique contractive transformation $Z\in\irI(V_{T_a}, b_a(V_T))$ and $X^+_T = ZX^+_{T_a}$. 
Since
\[ 
\big<A_Tx,x\big> = \big<X^+_Tx,X^+_Tx\big> = \|X^+_Tx\|^2 = \|ZX^+_{T_a}x\|^2 \leq \|X^+_{T_a}x\|^2 = \big<A_{T_a}x,x\big> \quad (x\in\irH),
\]
$A_{T} \leq A_{T_a}$ follows. 
Therefore
\[ 
A_T \leq A_{T_a} \leq A_{(T_a)_{-a}} = A_T.
\]
\end{proof}

Now we are ready to prove Theorem \ref{main_contr_coinc_thm}.

\begin{proof}[Proof of Theorem \ref{main_contr_coinc_thm}]
(i): For an arbitrary vector $x\in\irH$ and $i=1,2$, we have
\[ 
\big<A_{T_1T_2}x,x\big> = \lim_{n\to\infty}\big<(T_1T_2)^{*n}(T_1T_2)^{n}x,x\big> = \lim_{n\to\infty}\|(T_1T_2)^{n}x\|^2 
\]
\[
\leq \lim_{n\to\infty}\|T_i^{n}x\|^2 = \lim_{n\to\infty}\big<T_i^{*n}T_i^{n}x,x\big> = \big<A_{T_i}x,x\big>, 
\]
where we used the commuting property in the step $(T_1T_2)^n = T_1^nT_2^n = T_2^nT_1^n$.

(ii): Set $a := u(0)$ and $v = b_a\circ u$. 
Then obviously $v = \chi w$, where $\chi(z)=z$ and $w$ is an inner function. 
From (i) and Lemma \ref{cnu_Mobius} we get $A_{u(T)}=A_{v(T)}\leq A_T$. 
We consider isometric asymptotes $(X^+_T,V_T)$ and $(X^+_{u(T)},V_{u(T)})$ of $T$ and $u(T)$, respectively. 
The pair $(X^+_T,u(V_T))$ is a contractive intertwining pair of $u(T)$. 
Using the universal property of the isometric asymptotes, we get a unique contractive transformation $Z\in\irI(V_{u(T)},u(V_T))$ and $X^+_T = ZX^+_{u(T)}$. 
The last equality implies $A_T \leq A_{u(T)}$, and so $A_T = A_{u(T)}$.

(iii): Consider the isometric asymptotes
\[
(X^+, V_1), (X^+, V_2) \text{ and } (X^+_{T_1T_2},W)
\]
of $T_1$, $T_2$ and $T_1T_2$, respectively, where $X^+ = X^+_{T_1} = X^+_{T_2}$. 
Obviously the pair $(X^+,V_1V_2)$ is a contractive intertwining pair for $T_1T_2$. 
Hence we get, from the universality property of the isometric asymptote, that there exists a unique contractive $Z\in\irI(W,V_1V_2)$ and $X^+ = ZX^+_{T_1T_2}$. 
Therefore $A \leq A_{T_1T_2}$.

(iv): This is an immediate consequence of (i) and (iii).
\end{proof}

For an alternative proof of (ii) above see \cite[Lemma III.1]{CassierFack}. 
It can be also derived from \cite[Theorem 2.3]{KerchyDouglasAlg}.

Concluding this chapter we provide two examples. 
First we give two contractions $T_1,T_2\in C_{1\cdot}(\irH)$ such that $A_{T_1} = A_{T_2}$ and $A_{T_1T_2} \neq A_{T_1}$. 
This shows that (iii) of Theorem \ref{main_contr_coinc_thm} cannot be strengthened to equality even in the $C_{1\cdot}$ case. 
By (iv) of Theorem \ref{main_contr_coinc_thm} these contractions do not commute.

\begin{exmpl}
\textup{Take an orthonormal basis $\{e_{i,j}\colon i,j\in\N\}$ in $\irH$. 
The operators $T_1,T_2 \in\irB(\irH)$ are defined in the following way:}
\[ 
T_1 e_{i,j} := \left\{ \begin{matrix}
e_{i,j+1} & \text{if } j=1 \\
\frac{\sqrt{j^2-1}}{j} e_{i,j+1} & \text{if } j>1
\end{matrix}
\right., \quad T_2 e_{i,j} := \left\{ \begin{matrix}
e_{1,2} & \text{if } i=j=1 \\
e_{i+1,j-1} & \text{if } j=2 \\
\frac{\sqrt{3}}{2} e_{i-1,j+2} & \text{if } i>1, j=1 \\
\frac{\sqrt{j^2-1}}{j} e_{i,j+1} & \text{if } j>2
\end{matrix}
\right.. 
\]
\textup{$T_1$ and $T_2$ are orthogonal sums of infinitely many contractive, unilateral weighted shift operators, with different shifting schemes.
Straightforward calculations yield that}
\[ 
A_{T_1} e_{i,j} = A_{T_2} e_{i,j} = \left\{ \begin{matrix}
\frac{1}{2} e_{i,j} & \text{if } j = 1 \\
\frac{j-1}{j} e_{i,j} & \text{if } j > 1 
\end{matrix}
\right., 
\]
\textup{for every $i,j\in\N$, since $\Big(\prod_{l=j}^\infty \frac{\sqrt{l^2-1}}{l}\Big)^2 = \frac{j-1}{j}$ for $j>1$. On the other hand}
\[ 
T_2T_1 e_{i,j} = \left\{ \begin{matrix}
e_{i+1,1} & \text{if } j=1 \\
\frac{\sqrt{j^2-1}}{j}\frac{\sqrt{(j+1)^2-1}}{j+1}e_{i,j+2} & \text{if } j>1
\end{matrix}
\right., 
\]
\textup{hence}
\[ 
A_{T_2T_1} e_{i,j} = \left\{ \begin{matrix}
e_{i,1} & \text{if } j=1 \\
\frac{j-1}{j}e_{i,j} & \text{if } j>1
\end{matrix}
\right.,
\]
\textup{since} 
\[ 
\bigg(\prod_{m=0}^\infty \frac{\sqrt{(j+2m)^2-1}}{j+2m}\frac{\sqrt{(j+1+2m)^2-1}}{j+1+2m}\bigg)^2 = \bigg(\prod_{l=j}^\infty \frac{\sqrt{l^2-1}}{l}\bigg)^2 = \frac{j-1}{j}.
\] 
\textup{Therefore $A_{T_1}\leq A_{T_1T_2}$ and $A_{T_1}\neq A_{T_1T_2}$.}
\end{exmpl}

Finally, we give two contractions $T_1,T_2\in C_{0\cdot}(\irH)$ such that $T_1T_2\in\irC_{1\cdot}(\irH)$.

\begin{exmpl}
\textup{Take the same orthonormal basis in $\irH$ as in the previous example. 
The $C_{0\cdot}$-contractions $T_1,T_2\in\irB(\irH)$ are defined by}
\[ 
T_1 e_{i,j} := \frac{\sqrt{(i+1)^2-1}}{i+1} e_{i,j+1}, \qquad 
T_2 e_{i,j} := \left\{ \begin{matrix}
0 & \text{if } j=1 \\
e_{i-1,j+1} & \text{if } j>1
\end{matrix} \right.. 
\]
\textup{By a straightforward calculation we can check that}
\[ 
T_2T_1 e_{i,j} = \frac{\sqrt{(i+1)^2-1}}{i+1} e_{i+1,j}, 
\]
\textup{and so}
\[ 
A_{T_2T_1} e_{i,j} = \frac{i}{i+1} e_{i,j}. 
\]
\textup{Therefore we have $T_1T_2\in C_{1\cdot}(\irH)$.}
\end{exmpl}

%-------------------------------------------------------------------------------------------------------------------------------------------

\newpage

\chapter{Ces\`aro asymptotic limits of power bounded matrices} \label{matrix_A_C-s_chap}

\section{Statements of the main results}

This chapter is devoted to the characterization of all possible $L$-asymptotic limits of power bounded matrices which was done in \cite{Ge_matrix}. 
Throughout the chapter $T\in\irB(\C^d)$ ($2\leq d <\infty$) will always denote a power bounded matrix if we do not say otherwise. 
The first important step towards the desired characterization is to show that Banach limits can be replaced by usual limits if we consider Ces\'aro means of the self-adjoint iterates of $T$.

\begin{theorem}\label{finite_Ces_thm}
Let $T\in \irB(\C^d)$ be power bounded, then 
\[ 
A_{T,L} = A_{T,C} := \lim_{n\to\infty}\frac{1}{n}\sum_{j=1}^n T^{*j}T^j 
\]
holds for all Banach limits $L$.
\end{theorem}

We note that the above limit holds in norm, since we are in finite dimensions.
Concerning arbitrary dimensions if for an operator $T\in\irB(\irH)$ the limit $A_{T,C} = \lim_{n\to\infty}\frac{1}{n}\sum_{j=1}^n T^{*j}T^j$ exists in SOT, it will be called the Ces\`aro asymptotic limit of $T$. 
In order to prove Theorem \ref{finite_Ces_thm} we will derive some properties of the Jordan decomposition of $T$. 
Then we will give the characterization of $C_{11}$ matrices. 
We would like to point out that any matrix is of class $C_{1\cdot}$ if and only if it is of class $C_{11}$ which happens exactly when it is similar to a unitary matrix (this will follow from Proposition \ref{J_dec_pwb_mx_prop}). 
The characterization reads as follows in that case.

\begin{theorem}\label{C_11_char_matrix_thm}
The following statements are equivalent for a positive definite $A\in\irB(\C^d)$:
\begin{itemize}
\item[\textup{(i)}] $A$ is the Ces\`aro asymptotic limit of a power bounded matrix $T \in C_{11}(\C^d)$,
\item[\textup{(ii)}] if the eigenvalues of $A$ are $t_1,\dots,t_d > 0$, each of them is counted according to their multiplicities, then
\begin{equation}\label{C11_eq}
\frac{1}{t_1}+\dots+\frac{1}{t_d} = d
\end{equation}
holds,
\item[\textup{(iii)}] there is an invertible $S\in \irB(\C^d)$ with unit column vectors such that 
\[ 
A = S^{*-1}S^{-1} = (SS^*)^{-1}. 
\]
\end{itemize}
\end{theorem}

The proof of the general case uses the $C_{11}$ case, Lemma \ref{Ker_dec_lem} and a block-diagonalization of a special type of block matrices. 
We shall call a power bounded matrix $T\in\irB(\C^d)$ $l$-stable ($0\leq l\leq d$) if $\dim\irH_0 = l$. 
We will see later from Proposition \ref{J_dec_pwb_mx_prop} that $T$ and $T^*$ are simultaneously $l$-stable, which is not true in infinite dimensional spaces (simply consider a non-unitary isometry).
In the forthcoming theorem the symbol $I_l$ stands for the identity matrix on $\C^l$ and $0_{k}\in\irB(\C^k)$ is the zero matrix.

\begin{theorem}\label{non-C_11_char_matrix_thm}
The following three conditions are equivalent for a non-invertible positive semidefinite $A\in\irB(\C^d)$ and a number $1\leq l< d$:
\begin{itemize}
\item[\textup{(i)}] there exists a power bounded, $l$-stable $T\in\irB(\C^d)$ such that $A_{T,C} = A$,
\item[\textup{(ii)}] let $k = d-l$, if $t_1, \dots t_k$ denote the non-zero eigenvalues of $A$ counted with their multiplicities, then 
\[
\frac{1}{t_1}+\dots+\frac{1}{t_k}\leq k,
\]
\item[\textup{(iii)}] there exists such an invertible $S\in \irB(\C^d)$ that has unit column vectors and 
\[ 
A = S^{*-1}(I_l\oplus 0_{k})S^{-1}. 
\] 
\end{itemize}
\end{theorem}

One could ask whether there is any connection between the Ces\`aro asymptotic limit of a matrix and the Ces\`aro-asymptotic limit of its adjoint. 
If $T\in\irB(\C^d)$ is contractive, then $A_{T^*} = A_T$ is valid (see Theorem \ref{finite_dim_A_T_thm}). 
In the case of power bounded matrices usually the subspaces $\irH_0(T)$ and $\irH_0(T^*)$ are different and hence $A_{T^*}$ and $A_T$ are different, too. 
However, we provide the next connection for $C_{11}$ class $2\times 2$ power bounded matrices.

\begin{theorem}\label{2x2_regularity_matrix_thm}
For each $T\in C_{11}(\C^2)$ the harmonic mean of the Ces\`aro asymptotic limits $A_{T,C}$ and $A_{T^*,C}$ are exactly the identity $I$, i.e. 
\begin{equation}\label{2_dim_nice_eq}
A_{T,C}^{-1} + A_{T^*,C}^{-1} = 2I_2. 
\end{equation}
\end{theorem}

The proof relies on the fact that the inverse of a $2\times 2$ matrix can be expressed with its elements rather easily. 
The same method cannot be applied in higher dimensions. 
Moreover, \eqref{2_dim_nice_eq} is not even true in three-dimensions any more. 
This will be justified by a concrete $3\times 3$ example.

\section{Proofs}

We begin with the description of the Jordan decomposition of power bounded matrices. 
We recall that a Jordan matrix is a matrix of the following form:
\[
J = 
\left( \begin{matrix}
J_{\lambda_1} & 0 & \dots & 0 \\
0 & J_{\lambda_2} & \dots & 0 \\
\vdots & \vdots & \ddots & \vdots \\
0 & 0 & \dots & J_{\lambda_k}
\end{matrix} \right)
\]
where $J$ has the following matrix entries:
\[
J_\lambda = 
\left( \begin{matrix}
\lambda & 1 & \dots & 0 \\
0 & \lambda & \dots & 0 \\
\vdots & \vdots & \ddots & \vdots \\
0 & 0 & \dots & \lambda
\end{matrix} \right)
\]

By the Jordan decomposition of a matrix $T\in\irB(\C^d)$ we mean the product $T = SJS^{-1}$ where $S\in \irB(\C^d)$ is invertible and $J\in\irB(\C^d)$ is a Jordan matrix (see \cite{Pr}). 
Recall that for any matrix $B\in\irB(\C^d)$ the equation $\lim_{n\to\infty}\|B^n\| = 0$ holds if and only if $r(B)<1$ (see e.g. \cite[Proposition 0.4]{Kubrusly}). 
The symbol $\diag(\dots)$ expresses a diagonal matrix (written in the standard orthonormal base).

\begin{prop}\label{J_dec_pwb_mx_prop}
Suppose $T\in\irB(\C^d)$ is power bounded and let us consider its Jordan decomposition: $SJS^{-1}$. 
Then we have
\[ 
J = U\oplus B 
\] 
with a unitary $U = \diag(\lambda_1,\dots\lambda_k)\in\irB(\C^{k})$ ($k\in\N_0$) and a $B\in\irB(\C^{d-k})$ for which $r(B) < 1$.

Conversely, if $J$ has the previous form, then necessarily $T$ is power bounded.
\end{prop}

\begin{proof} For the first assertion, let us consider the block-decomposition of $J$: 
\[ 
J = J_{1}\oplus \dots \oplus J_k 
\] 
where $J_j$ is a $\lambda_j$-Jordan block. 
It is clear that power boundedness is preserved by similarity. 
Since
\begin{equation}\label{Jordan-block_power_eq}
J_j^n = \left( \begin{matrix}
\lambda_j & 1 & 0 & \dots & 0 \\
0 & \lambda_j & 1 & \dots & 0 \\
0 & 0 & \lambda_j & \ddots & 0 \\
\vdots & \vdots & \ddots & \ddots & \vdots \\
0 & 0 & 0 & \dots & \lambda_j
\end{matrix} \right)^n 
=
\left( \begin{matrix}
\lambda_j^n & n\lambda_j^{n-1} & \binom{n}{2}\lambda_j^{n-2} & \ddots & \ddots \\
0 & \lambda_j^n & n\lambda_j^{n+1} & \ddots & \ddots \\
0 & 0 & \lambda_j^n & \ddots & \ddots \\
\vdots & \vdots & \ddots & \ddots & \ddots \\
0 & 0 & 0 & \dots & \lambda_j^n
\end{matrix} \right)
\end{equation}
holds, we obtain that if $\lambda$ is an eigenvalue of $T$, then either $|\lambda| < 1$, or $|\lambda| = 1$ and the size of any corresponding Jordan-block is exactly $1\times 1$. 

On the other hand, if $J = U\oplus B$, then it is obviously power bounded, and hence $T$ is power bounded as well.
\end{proof}

Next we prove that the Ces\`aro asymptotic limit of $T$ always exists. 
Moreover, it can be expressed as the Ces\`aro asymptotic limit of another matrix $T'$ which is obtained by replacing the above matrix $B$ with $0$.

\begin{theorem}\label{C^d_Cer_mean_thm}
Let $T\in\irB(\C^d)$ be power bounded and by using the notations of Proposition \ref{J_dec_pwb_mx_prop} let us consider the matrices $J' = U\oplus 0$ and $T' = SJ'S^{-1}$. 
Then the Cesa\`aro asymptotic limit of $T$ and $T'$ exist and they coincide:
\begin{equation}\label{TT'_C_eq}
A_{T,C} = A_{T',C}.
\end{equation} 
Conversely, if the sequence $\{\frac{1}{n}\sum_{j=1}^n T^{*j}T^j\}_{n=1}^\infty$ converges for a matrix $T$, then it is necessarily power bounded.
\end{theorem}

\begin{proof}
For the first part, let us consider the following:
\begin{equation}\label{A_TC_and_AT'C_eq}
\begin{gathered}
 \frac{1}{n}\sum_{j=1}^n T^{*j}T^j - \frac{1}{n}\sum_{j=1}^n T'^{*j}T'^j \\
 = \frac{1}{n}S^{*-1}\Big(\sum_{j=1}^n J^{*j}S^*SJ^j - \sum_{j=1}^n J'^{*j}S^*SJ'^j\Big)S^{-1} \\
 = S^{*-1}\Big(\frac{1}{n}\sum_{j=1}^n (0\oplus B)^{*j}S^*S(0\oplus B)^j + \frac{1}{n}\sum_{j=1}^n (U\oplus 0)^{*j}S^*S(0\oplus B)^j \\
 + \frac{1}{n}\sum_{j=1}^n (B\oplus 0)^{*j}S^*S(U\oplus 0)^j \Big)S^{-1}.
\end{gathered}
\end{equation}
Since $\lim_{j\to\infty} (0\oplus B)^{j} = 0$, we obtain that \eqref{TT'_C_eq} holds whenever at least one of the limits exists.

Now we consider the partial sums: $\frac{1}{n}\sum_{j=1}^n T'^{*j}T'^j$. 
It is easy to see that multiplying from the right by a diagonal matrix acts as multiplication of the columns by the corresponding diagonal elements. 
Similarly, for the multiplication from the left action this holds with the rows. 
We have the next equality:
\begin{equation}\label{limit_eq}
\begin{gathered}
J'^{*j}S^*SJ'^j = \\
\left(\begin{matrix}
\|S e_1\|^2 & \lambda_2^n\overline{\lambda_1}^n \langle S e_1, S e_2\rangle & \dots & \lambda_m^n\overline{\lambda_1}^n \langle S e_1, S e_m\rangle & 0 & \dots & 0\\
\lambda_1^n\overline{\lambda_2}^n \langle S e_2, S e_1\rangle & \|S e_2\|^2 & \dots & \lambda_m^n\overline{\lambda_2}^n \langle S e_2, S e_m\rangle & 0 & \dots & 0 \\
\vdots & \vdots & \ddots & \vdots & \vdots & & \vdots \\
\lambda_1^n\overline{\lambda_m}^n \langle S e_m, S e_1\rangle & \lambda_2^n\overline{\lambda_m}^n \langle S e_m, S e_2\rangle & \dots & \|S e_m\|^2 & 0 & \dots & 0 \\
0 & 0 & \dots & 0 & 0 & \dots & 0\\
\vdots & \vdots & & \vdots & \vdots & \ddots & \vdots \\
0 & 0 & \dots & 0 & 0 & \dots & 0\\
\end{matrix}\right).
\end{gathered}
\end{equation}
Since 
\[ 
\lim_{n\to\infty} \Bigg|\frac{1}{n}\sum_{j=1}^n \lambda^j\Bigg| = \lim_{n\to\infty} \frac{|\lambda^n-1|}{n|\lambda-1|} = 0 
\]
if $|\lambda|\leq 1, \lambda \neq 1$ (for $\lambda = 1$ the above limit is 1) and multiplying by a fix matrix does not have an effect on the fact of convergence, we can easily infer that the right-hand side of \eqref{TT'_C_eq} exists and hence \eqref{TT'_C_eq} is verified.

For the reverse implication, let us assume that $\{\frac{1}{n}\sum_{j=1}^n T^{*j}T^j\}_{n=1}^\infty$ converges and as in the statement of Theorem \ref{finite_Ces_thm} let us denote its limit by $A_{T,C}$. 
According to the following equation
\[ 
\|A_{T,C}^{1/2}h\| = \langle A_{T,C} h,h\rangle^{1/2} = \lim_{n\to\infty} \Big(\frac{1}{n}\sum_{j=1}^n \|T^j h\|^2\Big)^{1/2},
\] 
and because of uniform convergence, there exists a large enough $\widetilde{M}>0$ such that
\[ 
\frac{1}{n}\sum_{j=1}^n \|T^j h\|^2 = \frac{1}{n}\sum_{j=1}^n \|SJ^jS^{-1} h\|^2 \leq \widetilde{M} 
\]
holds for each unit vector $h\in\C^d$. 
Since $S$ is bounded from below, the above inequality holds exactly when 
\[ 
\frac{1}{n}\sum_{j=1}^n \|J^j h\|^2 \leq M 
\]
holds for every unit vector $h\in\C^d$ with a large enough bound $M>0$. 
On the one hand, this implies that $r(J) = r(T) \leq 1$. 
On the other hand, if there is an at least $2\times 2$ $\lambda$-Jordan block in $J$ where $|\lambda| = 1$, then this above inequality obviously cannot hold for any unit vector (see \eqref{Jordan-block_power_eq}). 
This ensures the power boundedness of $T$.
\end{proof}

Let us point out that if a sequence of matrices $\{S_n\}_{n=1}^\infty$ is entry-wise $L$-convergent, then 
\[ 
\Llim_{n\to\infty} XS_n = X\cdot\Llim_{n\to\infty} S_n 
\] 
and 
\[ 
\Llim_{n\to\infty} S_nX = (\Llim_{n\to\infty} S_n)\cdot X 
\] 
hold. 
This can be easily verified from the linearity of $L$. 

Now we are ready to prove Theorem \ref{finite_Ces_thm}. 

\begin{proof}[Proof of Theorem \ref{finite_Ces_thm}]
The equality $\Llim_{n\to\infty} \lambda^n = \lambda\Llim_{n\to\infty} \lambda^n$ holds for every $|\lambda|\leq 1$ which gives us $\Llim_{n\to\infty} \lambda^n = 0$ for all $\lambda\neq 1, |\lambda|\leq 1$ (the Banach limit is trivially 1 if $\lambda = 1$). 
Now, if we take a look at equation \eqref{limit_eq}, we can see that $A_{T',L} = A_{T',C}$ holds for every Banach limit. 
Since 
\[ 
(0\oplus B)^{*j}S^*S(0\oplus B)^j + (0\oplus U)^{*j}S^*S(0\oplus B)^j + (B\oplus 0)^{*j}S^*S(0\oplus U)^j \to 0, 
\]
the equation-chain
\[ 
A_{T,L} = A_{T',L} = A_{T',C} = A_{T,C} 
\] 
is yielded.
\end{proof}

A natural question arises here. 
When does the sequence $\{T^{*n}T^n\}_{n=1}^\infty$ converge? The forthcoming theorem is dealing with this question where we will use a theorem of G. Corach and A. Maestripieri.

\begin{theorem}\label{AT_char_thm}
The following statements are equivalent for a power bounded $T\in\irB(\C^d)$:
\begin{itemize}
\item[\textup{(i)}] the sequence $\{T^{*n}T^n\}_{n=1}^\infty\subseteq\irB(\C^d)$ is convergent,
\item[\textup{(ii)}] the eigenspaces of $T$ corresponding to eigenvalues with modulus 1 are mutually orthogonal to each other.
\end{itemize}
Moreover, the following three sets coincide:
\begin{equation}\label{sets_eq}
\begin{gathered}
\left\{ A\in\irB(\C^d) \colon \exists\; T\in\irB(\C^d) \text{ power bounded such that } A = \lim_{n\to\infty} T^{*n}T^n \right\}, \\
\left\{ P^*P\in\irB(\C^d) \colon P\in\irB(\C^d), P^2 = P \right\}, \\
\left\{ A\in\irB(\C^d) \text{ positive semidefinite} \colon \sigma(A)\subseteq \{0\}\cup[1,\infty[, \dim\ker(A)\geq \rank E\big(]1,\infty[\big) \right\},
\end{gathered}
\end{equation}
where $E$ denotes the spectral measure of $A$.
\end{theorem}

\begin{proof}
The (ii)$\Longrightarrow$(i) implication is quite straightforward from \eqref{limit_eq}.

For the reverse direction let us consider \eqref{limit_eq} again. 
This tells us that if $\lambda_l \neq \lambda_k$, then $\langle S e_l, S e_k\rangle$ has to be 0 which means exactly the orthogonality.

In order to prove the further statement, we just have to take such a $T$ that satisfies (ii). 
Since the limit of $\{T^{*n}T^n\}_{n=1}^\infty$ exists if and only if the sequence $\{T'^{*n}T'^{n}\}_{n=1}^\infty$ converges (by a similar reasoning to the beginning of the proof of Theorem \ref{C^d_Cer_mean_thm}), we consider 
\[ 
T'^{*n}T'^n = S^{*-1}(J'^{*n}S^*SJ'^n)S^{-1}. 
\] 
Trivially, we can take an orthonormal basis in every eigenspace as column vectors of $S$ and if we do so, then the column vectors of $S$ corresponding to modulus 1 eigenvalues will form an orthonormal sequence, and the other column vectors (i.e. the zero eigenvectors) will form an orthonormal sequence as well. 
But the whole system may fail to be an orthonormal basis. 
This implies that
\[ 
\lim_{n\to\infty} T'^{*n}T'^n = S^{*-1}(I\oplus 0)S^{-1} = S^{*-1}(I\oplus 0)S^*S(I\oplus 0)S^{-1}. 
\]
By \cite[Theorem 6.1]{CM} we get \eqref{sets_eq}.
\end{proof}

Next we prove Theorem \ref{C_11_char_matrix_thm}.

\begin{proof}[Proof of Theorem \ref{C_11_char_matrix_thm}] 
(i)$\Longrightarrow$(iii): 
Let us suppose that $T = SUS^{-1}$ holds with an invertible $S\in\irB(\C^d)$ and a unitary $U = \diag(\lambda_1,\dots\lambda_d)$. 
Of course, it can be supposed without loss of generality that $S$ has unit column vectors (this is just a right choice of eigenvectors).
Moreover, if an eigenvalue $\lambda$ has multiplicity more than one, then the corresponding unit eigenvectors (as column vectors in $S$) can be chosen to form an orthonormal base in that eigenspace. 
Trivially, this does not change $T$. 
Now considering \eqref{limit_eq}, we get 
\[ 
\frac{1}{n} \sum_{j=1}^n U^{*j}S^*SU^j \to I 
\] 
and therefore 
\[ 
A_{T,C} = S^{*-1}S^{-1}. 
\]

\smallskip

(iii)$\Longrightarrow$(i):
Take an $S$ such that it has unit column vectors. 
If we put $\lambda_j$-s to be pairwise different in $U$ and $T = SUS^{-1}$, then we obviously get $A_{T,C} = S^{*-1}S^{-1}$ from equation \eqref{limit_eq}.

\smallskip

(iii)$\Longrightarrow$(ii):
By the spectral mapping theorem we have 
\[
d = \tr(S^*S) = \tr(SS^*) = \sum_{j=1}^d \frac{1}{t_j}.
\]

\smallskip

(ii)$\Longrightarrow$(iii):
It would be enough to find such a unitary matrix $U\in\irB(\C^d)$ which satisfies 
\[ 
\big\|\diag(\sqrt{1/t_1},\dots,\sqrt{1/t_d})\cdot U e_j\big\| = 1. 
\] 
Indeed, if we chose 
\[
S := \diag(\sqrt{1/t_1},\dots,\sqrt{1/t_d})\cdot U,
\] 
$SS^*$ would become $\diag(1/t_1,\dots,1/t_d)$ and \eqref{unitary_eqv_eq} would give what we wanted. 

The idea is that we put such complex numbers in the entries of $U$ that have modulus $1/\sqrt{d}$, because then the column vectors of $S$ will be unit ones and we only have to be careful with the orthogonality of the column vectors of $U$. 
In fact, the right choice is to consider a constant multiple of a Vandermonde matrix:
\[ 
U := \Big(\varepsilon^{(j-1)(k-1)}/\sqrt{d}\Big)_{j,k=1}^d 
\] 
where $\varepsilon = e^{2i\pi/d}$. 
We show that its column vectors are orthogonal to each other which will complete the proof. 
For this we consider $j_1\neq j_2, j_1,j_2\in\{1,2,\dots d\}$ and the inner product
\[ 
\langle Ue_{j_1}, Ue_{j_1} \rangle = \sum_{k=1}^d \frac{\varepsilon^{(j_1-1)(k-1)}}{\sqrt{d}} \frac{\overline{\varepsilon^{(j_2-1)(k-1)}}}{\sqrt{d}} = \frac{1}{d} \sum_{k=1}^d \varepsilon^{(j_1-j_2)(k-1)} = 0. 
\]
This shows that $U$ is indeed unitary.
\end{proof}

In order to handle the non-$C_{11}$ case we will use a reversal of Lemma \ref{Ker_dec_lem}. 
This reversal statement reads as follows.

\begin{lem}\label{structure_lem}
Suppose that the matrix $T\in\irB(\C^d)$ has the block-matrix upper triangular representation as in Lemma \ref{Ker_dec_lem}, with respect to an orthogonal decomposition $\C^d = \irH'\oplus\irH''$ such that $T_0\in C_{0\cdot}(\irH')$ ($R\in\irB(\irH'',\irH')$ is arbitrary) and $T_1\in C_{1\cdot}(\irH'')$. 
Then necessarily $T$ is power bounded and its stable subspace is precisely $\irH'$.
\end{lem}

\begin{proof}
By an easy calculation we get
\[ 
T^n = \left(\begin{matrix}
T_0^n & R_n \\
0 & T_1^n
\end{matrix}\right) 
\]
where 
\[ 
R_n = \sum_{j=1}^{n} T_0^{n-j}RT_1^{j}. 
\] 
Let us assume that $\|T_0^n\|,\|T_1^n\|,\|R\| < M$ is satisfied for each $n\in\N$. 
In order to see the power boundedness of $T$, one just has to use that $\|T_0^n\|\leq r^n$ holds for large $n$-s with a number $r<1$, since we are in a finite dimensional space. 
It is quite straightforward that $\irH' \subseteq \irH_0(T)$ and since for any vector $x\in\irH''$ the sequence $\{T_1^n x\}_{n=1}^\infty$ does not converge to 0, then neither does $\{T^n x\}_{n=1}^\infty$. 
Consequently we obtain that $\irH' = \irH_0(T)$.
\end{proof}

The above proof works for those kind of infinite dimensional operators as well for which $r(T_0) < 1$ is satisfied. 
However we note that in general it is not valid.

Now, we are in a position to show the characterization in the non-$C_{11}$ case.

\begin{proof}[Proof of Theorem \ref{non-C_11_char_matrix_thm}] The equivalence of (i) and (iii) can be handled very similarly as in Theorem \ref{C_11_char_matrix_thm}.

\smallskip

(i)$\Longrightarrow$(ii): 
Let us write
\[
T^* = \left(\begin{matrix}
0 & \tilde{R} \\
0 & E^*
\end{matrix}\right) = \left(\begin{matrix}
0 & RE^* \\
0 & E^*
\end{matrix}\right) \quad (E\in C_{11}(\C^k)) 
\]
for the adjoint of a power bounded matrix (up to unitarily equivalence) which is general enough for our purposes (see Theorem \ref{C^d_Cer_mean_thm} and Lemma \ref{structure_lem}). 
Since $E^*$ is invertible it is equivalent to investigate the two forms above. 
From the equation
\[ 
T^{*n}T^n = \left(\begin{matrix}
0 & RE^{*n} \\
0 & E^{*n}
\end{matrix}\right)
\left(\begin{matrix}
0 & 0 \\
E^nR^* & E^n
\end{matrix}\right) =
\left(\begin{matrix}
RE^{*n}E^nR^* & RE^{*n}E^n \\
E^{*n}E^nR^* & E^{*n}E^n
\end{matrix}\right),
\]
we get
\[ 
A_{T,C} = \left(\begin{matrix}
RA_{E,C}R^* & RA_{E,C} \\
A_{E,C}R^* & A_{E,C}
\end{matrix}\right). 
\]

Calculating the null-space of $A_{T,C}$ suggests that for the block-diagonalization we have to take the following invertible matrix:
\[ 
X = 
\left(\begin{matrix}
I_l & R \\
-R^* & I_k
\end{matrix}\right)
\left(\begin{matrix}
(I_l + RR^*)^{-1/2} & 0 \\
0 & (I_k + R^*R)^{-1/2}
\end{matrix}\right). 
\]
The inverse of $X$ is precisely
\[ 
X^{-1} = \left(\begin{matrix}
(I_l + RR^*)^{-1/2} & 0 \\
0 & (I_k + R^*R)^{-1/2}
\end{matrix}\right)
\left(\begin{matrix}
I_l & -R \\
R^* & I_k
\end{matrix}\right) 
\]
which implies that $X$ is unitary. 
With a straightforward calculation we derive
\[ 
X^{-1}A_{T,C}X = 0_l \oplus \big[(I_k + R^*R)^{1/2}A_{E,C}(I_k + R^*R)^{1/2}\big]. 
\]
Here $(I_k + R^*R)^{1/2} = I_k + Q$ holds with a positive semi-definite $Q$ for which $\rank(Q)\leq l$ holds and conversely, every such $I_k+Q$ can be given in the form $(I_k + R^*R)^{1/2}$. 
Therefore the set of all Ces\`aro-asymptotic limits of $l$-stable power bounded matrices (again, up to unitarily equivalence) are given by
\[ 
\begin{gathered}
\Big\{
0_l \oplus \left[(I_k + Q)A_{E,C}(I_k + Q)\right]\colon Q\in\irB(\C^k), \\
\text{ positive semidefinite } \rank(Q) < l, E\in C_{11}(\C^k) \Big\} 
\end{gathered}
\]
and every positive operator of the above form is the Ces\`aro-asymptotic limit of an $l$-stable power bounded matrix.

Finally, let us write
\[ 
(I_k+Q)^{-2} = U\cdot\diag(q_1,\dots q_k)\cdot U^* 
\]
and 
\[ 
U^*A_{E,C}^{-1}U = (\alpha_{i,j})_{i,j=1}^k, 
\]
where $U$ is unitary and $q_j\leq 1$ for each $1\leq j\leq k$. 
Therefore
\[ 
\frac{1}{t_1} + \dots + \frac{1}{t_k} = \tr\big((I_k+Q)^{-1}A_{E,C}^{-1}(I_k+Q)^{-1}\big) = \tr\big(A_{E,C}^{-1}(I_k+Q)^{-2}\big) 
\]
\[ 
= \tr\big(U^*A_{E,C}^{-1}U\cdot\diag(q_1,\dots q_k)\big) = \sum_{j=1}^k q_j \alpha_{j,j} \leq \tr(A_{E,C}^{-1}) = k
\]
is fulfilled for each $l$-stable power bounded matrix $T$. 

\smallskip

(ii)$\Longrightarrow$(i): 
Set some positive numbers $t_1,\dots, t_k$ such that 
\[ 
\frac{1}{t_1} + \dots + \frac{1}{t_k} \leq k 
\]
is valid. 
If we take such a $0<c<1$ for which
\[ 
\frac{1}{c\cdot t_1} + \frac{1}{t_2} + \dots + \frac{1}{t_k} = k, 
\] 
then obviously $A = \diag(c\cdot t_1, t_2\dots t_k)$ arises as the Ces\`aro asymptotic limit of a power bounded operator from $C_{11}(\C^k)$.
Therefore if we take the rank-one $Q = \diag(1/\sqrt{c}-1,0,\dots 0)$ positive semi-definite matrix, we get $(I_k+Q)A(I_k+Q) = \diag(t_1, t_2\dots t_k)$. 
This ends the proof.
\end{proof}

Next we prove our formula concerning 2$\times$2 matrices.

\begin{proof}[Proof of Theorem \ref{2x2_regularity_matrix_thm}]
Let $T = S\diag(\lambda_1,\lambda_2)S^{-1}$ with $|\lambda_1| = |\lambda_2| = 1$ and $S\in \irB(\C^2)$ be invertible for which the column vectors are unit ones. 
We have learned from the proof of Theorem \ref{C_11_char_matrix_thm} that in this case $A_{T,C}^{-1} = SS^*$ happens. 
The matrix $S$ can be written in the following form:
\[ 
S = \left(\begin{matrix}
\mu_{1,1}s & \mu_{1,2}t \\
\mu_{2,1}\sqrt{1-s^2} & \mu_{2,2}\sqrt{1-t^2}
\end{matrix}\right) 
\]
where $|\mu_{j,l}| = 1$ ($j,l\in\{1,2\}$) and $s,t\in[0,1]$ provided that the above matrix is invertible.

By taking $\diag(1/\mu_{1,1},1/\mu_{2,1})\cdot S$ instead of $S$ and using \eqref{unitary_eqv_eq}, we can see that without loss of generality $\mu_{1,1} = \mu_{2,1} = 1$ can be assumed, so we have
\[ 
S = \left(\begin{matrix}
s & \mu_{1,2}t \\
\sqrt{1-s^2} & \mu_{2,2}\sqrt{1-t^2}
\end{matrix}\right) 
\]
and thus
\[ 
A_{T,C}^{-1} = SS^* = \left(\begin{matrix}
s^2+t^2 & s\sqrt{1-s^2} + t\sqrt{1-t^2}\overline{\mu_{22}}\mu_{12} \\
s\sqrt{1-s^2} + t\sqrt{1-t^2}\mu_{22}\overline{\mu_{12}} & 2-s^2-t^2
\end{matrix}\right). 
\]

Finally, since $T^* = S^{*-1}\diag(\overline{\lambda_1},\overline{\lambda_2})S^*$ and we have 
\[ 
S^{*-1} = \frac{1}{s\sqrt{1-t^2}\overline{\mu_{22}} - t\sqrt{1-s^2}\overline{\mu_{12}}} \left(\begin{matrix}
\sqrt{1-t^2}\overline{\mu_{22}} & -\sqrt{1-s^2} \\
-t\overline{\mu_{12}} & s
\end{matrix}\right),
\]
we immediately obtain
\[ 
A_{T^*,C}^{-1} = 
\left(\begin{matrix}
\sqrt{1-t^2}\overline{\mu_{22}} & -\sqrt{1-s^2} \\
-t\overline{\mu_{12}} & s
\end{matrix}\right)
\left(\begin{matrix}
\sqrt{1-t^2}\mu_{22} & -t\mu_{12} \\
-\sqrt{1-s^2} & s
\end{matrix}\right) 
\]
\[ 
= \left(\begin{matrix}
2-s^2-t^2 & -s\sqrt{1-s^2} - t\sqrt{1-t^2}\overline{\mu_{22}}\mu_{12} \\
-s\sqrt{1-s^2} - t\sqrt{1-t^2}\mu_{22}\overline{\mu_{12}} & s^2+t^2
\end{matrix}\right) 
\]
which implies \eqref{2_dim_nice_eq}.
\end{proof}

In the proof we used that the inverse of a $2\times 2$ matrix can be expressed quite nicely. 
As we mentioned earlier, the above theorem cannot be generalized in higher dimensions. 
In fact when $\dim\irH = \aleph_0$, then we can get counterexamples quite easily. 
Let us consider a weighted bilateral shift operator which has weights 1 everywhere except for one weight which is 1/2. 
This trivially defines a contraction.
Moreover, $A_T$ and $A_{T^*}$ are invertible which implies that $T$ is similar to a unitary. 
However $A_T, A_{T^*} \leq I$, $A_T \neq I$ and $A_{T^*} \neq I$ which implies $A_{T}^{-1} + A_{T^*}^{-1} \geq 2I$ but $A_{T}^{-1} + A_{T^*}^{-1} \neq 2I$. 

In the matrix case we do not have such an easy counterexample. 
However, simple computations tell us that usually \eqref{2_dim_nice_eq} does not hold even for $3\times 3$ matrices. 
For example if we consider 
\[ 
T = \left(\begin{matrix}
i & 2 & 1 \\
0 & 1 & i \\
1 & 0 & 4
\end{matrix}\right)
\left(\begin{matrix}
1 & 0 & 0 \\
0 & -1 & 0 \\
0 & 0 & i
\end{matrix}\right)
\left(\begin{matrix}
i & 2 & 1 \\
0 & 1 & i \\
1 & 0 & 4
\end{matrix}\right)^{-1} 
\]
then the eigenvalues of $A_{T,C}^{-1} + A_{T^*,C}^{-1}$ are approximately $1.27178, 2.1285$ and $2.59972$.

Finally we give some examples. 
First, we provide such a power bounded weighted shift operator $T$ for which the $L$-asymptotic limit really depends on the choice of the particular Banach limit $L$.
Moreover, the Ces\`aro asymptotic limit of $T$ exists. 
We will use a characterization from \cite{Lo} of all the possible Banach limits of a bounded real sequence. 
Namely, G. G. Lorentz proved that for every $\underline{x}\in\ell^\infty$ real sequence the equality
\[ \big\{\Llim_{n\to\infty} x_n \colon L\in\irL\big\} = [q'(\underline{x}),q(\underline{x})] \subseteq \R \]
holds where $\irL$ stands for the set of all Banach limits and
\[ 
q(\underline{x}) = \inf\Big\{\limsup_{k\to\infty} \frac{1}{p}\sum_{j=1}^{p} x_{n_j+k} \colon p\in\N, n_1,\dots n_p \in\N \Big\}, 
\]
\[ 
q'(\underline{x}) = \sup\Big\{\liminf_{k\to\infty} \frac{1}{p}\sum_{j=1}^{p} x_{n_j+k} \colon p\in\N, n_1,\dots n_p \in\N \Big\}. 
\]

\begin{exmpl}
\textup{Let us fix an orthonormal basis $\{e_j\}_{j\in\N}$ and define $T$ in the following way:}
\[ 
T e_j = \left\{ \begin{matrix}
\sqrt{2}\cdot e_{j+1} & j = 3^l \textup{ for some } l\in\N \\
\sqrt{1/2}\cdot e_{j+1} & j = 3^l+l \textup{ for some } l\in\N \\
e_{j+1} & \textup{otherwise}
\end{matrix}\right.. 
\]
\textup{Since $\|T^n\| = \sqrt{2}$ holds for each $n\in\N$, $T$ is power bounded. 
A rather easy calculation shows that the equation}
\[ T^{*n}T^n e_1 = \left\{ \begin{matrix}
2 e_1 & 3^l \leq n < 3^l+l \textup{ for some } l\in\N \\
e_1 & \textup{otherwise}
\end{matrix}\right. \]
\textup{is valid. 
Therefore by Lorentz's characterization we get}
\[ 
\big\{\langle A_{T,L} e_1, e_1\rangle \colon L\in\irL\big\} = [1,2]. 
\]

\textup{It is quite easy to see that the sequence $\{T^{*n}T^n\}_{n=1}^\infty$ consists of diagonal operators (with respect to our fixed orthonormal basis). 
A straightforward computation gives us $A_{T,C} = I$.}
\end{exmpl}

\smallskip

The last two examples concern the existence of the Ces\`aro asymptotic limit and power boundedness. 
From these examples we will see that none of the mentioned properties concerning $T$ implies the other one.

\begin{exmpl}
\textup{We define $T$ with the equation}
\[ 
T e_j = \left\{ \begin{matrix}
\sqrt{2}\cdot e_{j+1} & j = 3^l \textup{ for some } l\in\N \\
\sqrt{1/2}\cdot e_{j+1} & j = 2\cdot 3^l \textup{ for some } l\in\N \\
e_{j+1} & \textup{otherwise}
\end{matrix}\right.. 
\]
\textup{Clearly}
\[ 
T^{*n}T^n e_1 = \left\{ \begin{matrix}
2 \cdot e_1 & 3^l \leq n < 2\cdot 3^l \textup{ for some } l\in\N \\
e_1 & \textup{otherwise}
\end{matrix}\right. 
\]
\textup{holds. 
The Ces\`aro means of the sequence $\{\langle T^{*n}T^n e_1,e_1\rangle\}_{n\in\N}$ does not converge, hence the the Ces\`aro asymptotic limit of $T$ cannot exist. 
On the other hand, $\|T^n\| = \sqrt{2}$ is satisfied for every $n\in\N$, thus $T$ is power bounded.}
\end{exmpl}

\begin{exmpl}
\textup{Our shift operator is defined as follows:}
\[ 
T e_j = \left\{ \begin{matrix}
\sqrt{2}\cdot e_{j+1} & 3^l \leq j < 3^l + l \textup{ for some } l\in\N \\
e_{j+1} & \textup{otherwise}
\end{matrix}\right.. 
\]
\textup{On the one hand, it is not power bounded because $\|T^n\| = \sqrt{2^n}$. 
On the other hand, it is not hard to see that $A_{T,C}$ exists and it is precisely the identity operator.}
\end{exmpl}

%----------------------------------------------------------------------------------------------------------------------------------------------

\newpage

\chapter{Asymptotic limits of operators similar to normal operators} \label{SzN_chap}

\section{Statements of the results}

The present chapter contains our results from \cite{Ge_SzN}. 
First we will give a generalization of the necessity part in Sz.-Nagy's similarity result (Theorem \ref{SzN_unit_thm} and \ref{Sz-N_ref_thm}).
Let us consider a normal power bounded operator $N\in\irB(\irH)$. 
Since $r(N)^n = r(N^n) = \|N^n\|$ holds ($n\in\N$), we obtain that $N$ is a contraction. 
A quite straightforward application of the functional model of normal operators (see \cite[Chapter IX]{Co}) gives us that $A_N = I_{\irH_0(N)^\perp}\oplus 0_{\irH_0(N)}$ holds where $N|(\ran A_N)^-$ is the unitary part of $N$. 
In view of Theorem \ref{Sz-N_ref_thm} it is reasonable to conjecture that when a power bounded operator $T$ is similar to a normal operator then $A_{T,L}|(\ran A_{T,L})^-$ should be invertible.
Indeed this is the case as we shall see from the forthcoming theorem.
But before stating it we need some definitions. 
We recall that if the operator $A\in\irB(\irH)$ is not zero, then the reduced minimum modulus of $A$ is the quantity $\gamma(A) := \inf\{\|Ax\|\colon x\in(\ker A)^\perp, \|x\|=1\}$. 
In particular if $A$ is a positive operator, then $\gamma(A)>0$ holds exactly when $A|(\ran A)^-$ is invertible. 
Our first result concerns two similar power bounded operators and it reads as follows. 

\begin{theorem}\label{similar_thm}
Let us consider two power bounded operators $T,S\notin C_{0\cdot}(\irH)$ which are similar to each other. Then $\gamma(A_{T,L})>0$ holds for some (and then for all) Banach limits $L$ if and only if $\gamma(A_{S,L})>0$ is valid. 

Moreover, $\gamma(A_{T,L})>0$ holds if and only if the powers of $T$ are bounded from below uniformly on $\irH_0(T)^\perp$, i.e. there exists a constant $c > 0$ such that 
\[ 
c\|x\| \leq \|T^n x\| \quad (x\in\irH_0(T)^\perp, n\in\N).
\]

In particular, if $T$ is similar to a normal operator, then $\gamma(A_{T,L})>0$ and $\gamma(A_{T^*,L})>0$ are satisfied.
\end{theorem}

The proof of Theorem \ref{similar_thm} will use Lemma \ref{Ker_dec_lem}. 
This theorem can be considered as a generalization of the necessity part in Sz.-Nagy's theorem and it can help in deciding whether an operator is similar to e.g. a normal operator in some cases. 
However, the last point of Theorem \ref{similar_thm} is clearly not reversible which can be seen by easy counterexamples. 
For instance set an arbitrary operator $B\in\irB(\irH)$ with $r(B)<1$ and consider $T := B\oplus I \in\irB(\irH\oplus\irH)$. 
On the one hand, $T$ is usually not similar to any normal operators, but on the other hand, $A_T = 0\oplus I$ which implies $\gamma(A_T)=1$.

As we have seen, Sz.-Nagy's theorem says that $A_{T,L}$ is invertible when $T$ is similar to a unitary. 
If $T$ is a contraction, then we can state more than simply the invertibility of $A_T$. 
Namely we will prove the following statement where (i) is only an implications statement, but (ii) is an ''if and only if'' one.
We recall that the minimal element of the spectrum of a self-adjoint operator $A$ is denoted by $\underline{r}(A)$.

\begin{theorem}\label{similar_to_unitary_thm}\noindent
\begin{itemize}
\item[\textup{(i)}] Let $\dim\irH = \infty$ and $T\in\irB(\irH)$ be a contraction which is similar to a unitary operator. Then $\dim\ker(A_T-\underline{r}(A_T) I) \in \{0, \infty\}$. 
Consequently, the condition $\underline{r}(A_T)\in \sigma_e(A_T)$ holds.
\item[\textup{(ii)}] Let $\dim\irH = \aleph_0$, $A\in\irB(\irH)$ be a positive, invertible contraction and suppose that the conditions $1\in\sigma_e(A)$ and $\dim\ker(A-\underline{r}(A)I) \in \{0,\aleph_0\}$ are fulfilled. 
Then there exists a contraction $T\in\irB(\irH)$ which is similar to a unitary operator and the asymptotic limit of $T$ is exactly $A$.
\end{itemize}
\end{theorem}

This theorem can be considered as a strengthening of the necessity part in Sz.-Nagy's similarity theorem. 
In particular, in the separable case it provides a characterization of asymptotic limits for those contractions which are similar to unitary operators. 
We note that it is an open problem to describe the $L$ asymptotic limits of those operators that are similar to unitary operators and act on infinite dimensional spaces.

\section{Proofs}

In order to prove Theorem \ref{similar_thm} we need two lemmas. 
The first one reads as follows.

\begin{lem}\label{Sz.-Nagy_gen_lem}
Consider a power bounded operator $T\notin C_{0\cdot}(\irH)$. 
Then the following conditions are equivalent:
\begin{itemize}
\item[\textup{(i)}] the inequality $\gamma(A_{T,L}) > 0$ is satisfied for some and then for all Banach limits $L$,
\item[\textup{(ii)}] the compression $T_1 := P_1 T|\irH_0^\perp$ is similar to an isometry, where $P_1$ denotes the orthogonal projection onto the subspace $\irH_0^\perp$,
\item[\textup{(iii)}] the powers of $T$ are bounded from below uniformly on $\irH_0^\perp$, i.e. there exists a constant $c > 0$ such that $c\|x\| \leq \|T^n x\|$ is satisfied on $\irH_0^\perp$ for all $n\in\N$.
\end{itemize}
\end{lem}

\begin{proof} (i)$\Longrightarrow$(ii). 
We will use the decomposition $A_{T,L} = 0\oplus A_1 \in\irB(\irH_0\oplus\irH_0^\perp)$, where $A_1$ is obviously invertible. 
Consider the isometric asymptote $(X^+_{T,L},V_{T,L})$. 
If we restrict the equation
\begin{equation}\label{isom_as_eq}
V_{T,L} X^+_{T,L} = X^+_{T,L} T
\end{equation}
to the subspace $\irH_0^\perp$, we get the following:
\[ 
V_{T,L} A_1^{1/2} = V_{T,L}X_{T,L}^+|\irH_0^\perp = X_{T,L}^+ T|\irH_0^\perp = A_1^{1/2}T_1, 
\]
which verifies that the operator $T_1$ is indeed similar to the isometry $V_{T,L}$.

(ii)$\Longrightarrow$(iii). 
By Lemma \ref{Ker_dec_lem}, we have
\[ 
T^n = \left(\begin{matrix}
T_0^n & * \\
0 & T_1^n
\end{matrix}\right). 
\]
Therefore, by Sz.-Nagy's theorem, there exists a constant $c > 0$ for which
\[ 
\|T^n x\| \geq \|T_1^n x\| \geq c \|x\| \quad (n\in\N, x\in\irH_0^\perp).
\]

(iii)$\Longrightarrow$(i). 
Let $x\in\irH_0^\perp$ be arbitrary, then we have
\[ 
\|A_{T,L}^{1/2} x\|^2 = \Llim_{n\to\infty} \|T^n x\|^2 \geq c^2 \|x\|^2, 
\]
which means exactly that $\gamma(A_{T,L}^{1/2}) \geq c$ and hence $\gamma(A_{T,L}) \geq c^2 > 0$ is satisfied.
\end{proof}

We proceed with the following technical lemma.

\begin{lem}\label{inv_uper_block-triang_lem}
Consider an orthogonal decomposition $\irH = \irK\oplus\irL$, and an invertible operator $X\in\irB(\irH)$. 
Suppose that the block-matrix of $X$ is
\[ 
X = \left(\begin{matrix}
X_{11} & X_{12}\\
0 & X_{22}
\end{matrix}\right) \in \irB(\irK\oplus\irL), 
\]
and the element $X_{11}\in\irB(\irK)$ is surjective. 
Then the elements $X_{11}\in\irB(\irK)$ and $X_{22}\in\irB(\irL)$ are invertible and the block-matrix form of $X^{-1}$ is the following:
\[ 
X^{-1} = \left(\begin{matrix}
X_{11}^{-1} & -X_{11}^{-1}X_{12}X_{22}^{-1}\\
0 & X_{22}^{-1}
\end{matrix}\right) \in \irB(\irK\oplus\irL). 
\]
\end{lem}

\begin{proof}
Let $X^{-1} = \left(\begin{matrix}
Y_{11} & Y_{12}\\
Y_{21} & Y_{22}
\end{matrix}\right) \in \irB(\irK\oplus\irL)$. 
Since $X$ is invertible, $X_{11}$ has to be injective, thus bijective. 
The (2,1)-element of the block-matrix decomposition of $X^{-1}X = I$ is $Y_{21}X_{11} = 0 \in\irB(\irK,\irL)$ which gives us $Y_{12} = 0$. 
The (2,2)-elements of $XX^{-1} = I$ and $X^{-1}X = I$ are $X_{22}Y_{22} = I_\irL$ and $Y_{22}X_{22} = I_\irL$, respectively, which imply the invertibility of $X_{22} \in \irB(\irL)$. 
Finally, an easy calculation verifies the block-matrix form of $X^{-1}$.
\end{proof}

Now we are in a position to present our generalization of Sz.-Nagy's theorem.

\begin{proof}[Proof of Theorem \ref{similar_thm}] We begin with the first part. Let $X\in\irB(\irH)$ be an invertible operator for which $S = XTX^{-1}$ holds. It is easy to see that $\irH_0(S) = X\irH_0(T)$, which gives us $\dim\irH_0(T) = \dim\irH_0(S)$ and $\dim\irH_0(T)^\perp = \dim\irH_0(S)^\perp$. Therefore we can choose a unitary operator $U\in\irB(\irH)$ such that the equation
\begin{equation}\label{sim_pwb_stabel_spc_eq}
\irH_0(T) = U \irH_0(S) = U X \irH_0(T) = \irH_0(USU^*)
\end{equation}
is valid. By \eqref{unitary_eqv_eq}, it is enough to prove the inequality 
\[\gamma(A_{USU^*,L}) > 0.\]

Now, consider the block-matrix decompositions (\ref{Ker_dec_eq}) and
\[ UX = \left(\begin{matrix}
Y_{11} & Y_{12}\\
0 & Y_{22}
\end{matrix}\right) \in \irB(\irH_0(T)\oplus(\irH_0(T))^\perp). \]
The latter one is indeed upper block-triangular and moreover, the element $Y_{11}$ is surjective, because of (\ref{sim_pwb_stabel_spc_eq}). Therefore by Lemma \ref{inv_uper_block-triang_lem} we obtain the equation
\[ (UX)^{-1} = \left(\begin{matrix}
Y_{11}^{-1} & -Y_{11}^{-1}Y_{12}Y_{22}^{-1}\\
0 & Y_{22}^{-1}
\end{matrix}\right) \in \irB(\irH_0(T)\oplus(\irH_0(T))^\perp). \]
An easy calculation gives the following:
\[ P_1 USU^*|(\irH_0(T))^\perp = P_1(UX)T(UX)^{-1}|(\irH_0(T))^\perp = Y_{22}T_1Y_{22}^{-1}, \]
where $P_1$ denotes the orthogonal projection onto the subspace $(\irH_0(T))^\perp$. Now, if the inequality $\gamma(A_{T,L}) > 0$ holds, then by Lemma \ref{Sz.-Nagy_gen_lem} the operator $T_1$ is similar to an isometry. But this gives that the compression $P_1 USU^*|(\irH_0(T))^\perp$ is also similar to an isometry, and hence by Lemma \ref{Sz.-Nagy_gen_lem} and (\ref{sim_pwb_stabel_spc_eq}) we get that $\gamma(A_{S,L}) > 0$ holds.

The second part was proven in Lemma \ref{Sz.-Nagy_gen_lem}. 

The third part is an easy consequence of the fact that the asymp\-totic limit of a power bounded normal operator $N$ is always idempotent.
\end{proof}

Next we prove a consequence of Theorem \ref{similar_thm}. We recall definitions of some special classes of operators to which the similarity will be investigated in the forthcoming corollary. The operator $T\in\irB(\irH)$ is said to be 
\begin{itemize}
\item \emph{of class $Q$} if $\|Tx\|^2 \leq \frac{1}{2}(\|T^2x\|^2+\|x\|^2)$ holds for every $x\in\irH$,
\item \emph{log-hyponormal} if $\log(T^*T) \geq \log(TT^*)$ is satisfied.
\end{itemize}
An operator $T$ is called \emph{paranormal} if $\|T x\|^2 \leq \|T^2 x\|\|x\|$ is valid for all $x\in\irH$. It is quite easy to verify from the arithmetic-geometric mean inequality that every paranormal operator is of class $Q$ as well. 

We say that the operator $T$ has the \emph{Putnam--Fuglede property} (or \emph{PF property} for short) if for any operator $X\in\irB(\irH,\irK)$ and isometry $V\in\irB(\irK)$ for which $TX = XV^*$ holds, the equation $T^*X = XV$ is satisfied as well.

\begin{corollary}\label{exotic_cor}
For a power bounded operator $T\in\irB(\irH)$ and a Banach limit $L$ the following implications are valid:
\begin{itemize}
\item[\textup{(i)}] if $T\notin C_{\cdot 0}(\irH)$ is similar to a power bounded operator that has the PF property, then the condition $\gamma(A_{T^*,L}) > 0$ is fulfilled.
\item[\textup{(ii)}] if $T\notin C_{\cdot 0}(\irH)$ is similar to an operator that is either log-hyponormal or of class $Q$ or paranormal, then the inequality $\gamma(A_{T^*,L}) > 0$ is satisfied.
\end{itemize}
\end{corollary}

\begin{proof}
Theorem 3.2 of \cite{Pa_PF} tells us that the PF property for a power bounded operator $T$ is equivalent to the condition that $T$ is the orthogonal sum of a unitary and a power bounded operator of class $C_{\cdot 0}$. Therefore (i) is an easy consequence of Pagacz's result and Theorem \ref{similar_thm}.

If $T$ is log-hyponormal, then Mecheri's result (see \cite{Me}) implies that $T$ has the PF property, and thus $\gamma(A_{T*,L}) > 0$ holds. 

Finally, let us assume that $T$ is a power bounded operator which also belongs to the class $Q$. We prove that then it is a contraction as well. If $\|Tx\|^2-\|x\|^2 > a > 0$ held for a vector $x\in\irH$, then we would obtain $\|T^2x\|^2-\|Tx\|^2 \geq \|Tx\|^2-\|x\|^2 > a$. By induction we could prove that $\|T^{n+1}x\|^2-\|T^nx\|^2 > a$ would hold for every $n\in\N$. Therefore the inequality $\|T^{n+1}x\|^2-\|x\|^2 > n\cdot a$ would be true, which would imply that $T$ could not be power bounded. Consequently, $T$ has to be a contraction. P. Pagacz showed that a contraction which belongs to the class $Q$, shares the PF property (see \cite{Pa_Wold} and \cite{Ok} for the paranormal case). This gives us that $\gamma(A_{T*,L}) > 0$ is valid, which completes our proof.
\end{proof}

Next an application of the above results will be presented.
It is easy to see that if the weighted bilateral shift operator $T$ with weights $\{w_k\}_{k\in\Z} \subseteq \C$ is power bounded, then the $L$-asymptotic limit satisfies the equation $A_{T,L} e_k = \big(\Llim_{n\to\infty} \prod_{j=0}^{n} |w_{k+j}|^2\big)\cdot e_k$ for every $k\in\Z$.
A weighted bilateral shift operator $T$ is power bounded if and only if the inequality 
\[ 
\sup \Big\{ \prod_{j=0}^{n} |w_{k+j}| \colon k\in\Z, n\in\N\cup\{0\} \Big\} < \infty 
\]
is fulfilled (see \cite[Proposition 2]{Shields}). 
By Sz.-Nagy's theorem $T$ is similar to a unitary operator exactly when it is power bounded and, in addition, the following holds:
\[ 
\inf \Big\{ \prod_{j=0}^{n} |w_{k+j}| \colon k\in\Z, n\in\N\cup\{0\} \Big\} > 0. 
\]

\begin{corollary}\label{shift_cor}
Let $I$ be an arbitrary set of indices. 
Consider the orthogonal sum $W = \oplus_{i\in I} W_i \in \irB(\oplus_{i\in I}\irH_i)$ which is power bounded, and each summand $W_i$ is a weighted bilateral shift operator that is similar to a unitary operator. 
If $W$ is similar to a normal operator, then necessarily it is similar to a unitary operator.
\end{corollary}

\begin{proof} Let us denote the $L$-asymptotic limit of $W_i$ by $A_i$. 
Since the subspaces $\irH_i$ are invariant for the operators $W^{*n}W^n$ ($i\in I, n\in\N$), we obtain the equation $A_{W,L} = \oplus_{i\in I} A_i$. 
Since each summand $W_i$ is similar to a unitary operator, the operator $A_i$ is invertible and $A_i^{1/2}W_i = S_iA_i^{1/2}$ holds for every $i\in I$ where $S_i$ denotes a simple (i.e. unweighted) bilateral shift operator. 
From the power boundedness of $W$, $\sup\{\|A_i\|\colon i\in I\}<\infty$ follows.

On the one hand, if $\sup\{ \|A_i^{-1}\| \colon i\in I\} < \infty$ is satisfied, then
\[ 
W = (\oplus_{i\in I} A_i)^{-1/2}(\oplus_{i\in I} S_i) (\oplus_{i\in I} A_i)^{1/2}
\]
gives that $W$ is similar to a unitary operator. 
On the other hand, if the inequality fails, then $A_{W,L} = \oplus_{i\in I} A_i$ is not invertible, but injective. 
By Theorem \ref{similar_thm}, we obtain that in this case $W$ cannot be similar to any normal operator.
Our proof is complete.
\end{proof}

We continue with the verification of (i) in Theorem \ref{similar_to_unitary_thm}. Point (ii) will be proven later.

\begin{proof}[Proof of (i) in Theorem \ref{similar_to_unitary_thm}]
Since $A_T$ is invertible, the inequality $\underline{r} := \underline{r}(A_T) > 0$ is satisfied. 
It is trivial that if $\underline{r} = 1$, then $A_T = I$, and in this case the statement of the theorem is obviously true. 
Therefore we may suppose that $\underline{r} < 1$.

We will use the notation $\irM = \ker(A_T-\underline{r} I)$. 
Assume that the condition $0 < \dim\irM < \infty$ holds. 
If we set an arbitrary vector $h\in\irM$, then we have
\begin{equation}\label{T_inv_inv_eq}
\underline{r}^{1/2} \|h\| = \|A_T^{1/2} h\| = \|A_T^{1/2} T^{-1} h\| \geq \underline{r}^{1/2} \|T^{-1} h\|
\end{equation}
which implies that the inequality $\|T^{-1} h\| \leq \|h\|$ is fulfilled for any $h\in\irM$. 
But $T$ is a contraction, therefore $\|T^{-1} h\| = \|h\|$ for every $h\in\irM$. 
Because of the latter equation and \eqref{T_inv_inv_eq} we deduce $\|A_T^{1/2} T^{-1} h\| = \underline{r}^{1/2} \|T^{-1} h\| \; (h\in\irM)$ which implies that the finite-dimensional subspace $\irM$ is invariant for the operator $T^{-1}$. 
Since $T$ is bijective, we get that $T^{-1}\irM = \irM$ is satisfied and the restriction $T|\irM$ is unitary. 
Since $T$ is a contraction and $\dim\irM < \infty$, these imply that $\irM$ is reducing for $T$. 
On the other hand, $\underline{r} I|\irM = A_T|\irM = I|\irM$ follows which is a contradiction.
\end{proof}

Before proving (ii) of Theorem \ref{similar_to_unitary_thm}, we need the following lemma. 
We note that the method which will be used here is similar to the one which was used in Section 2. 
There operator-weighted unilateral shift operators were used and here we use operator-weighted bilateral shift operators. 
This will result in some further complications.

\begin{lem}\label{rev_lem}
Suppose we have a positive, invertible contraction $A\in\irB(\irH)$, an orthogonal decomposition $\irH = \oplus_{k=-\infty}^\infty \irY_k$ where the subspaces $\irY_k$ are reducing for $A$, and a unitary operator $U\in\irB(\irH)$ such that the following conditions hold:
\begin{itemize}
\item[\textup{(i)}] the equation $U\irY_k = \irY_{k+1}$ is satisfied for all $k\in\Z$,
\item[\textup{(ii)}] we have $\lim_{k\to\infty} \underline{r}(A|\irY_k) = 1$, and
\item[\textup{(iii)}] the inequality $\|A^{1/2} y_k\| \leq \|A^{1/2} U y_k\|$ is fulfilled for every $k\in\Z$ and $y_k\in\irY_k$.
\end{itemize}
Then $T := A^{-1/2}UA^{1/2} \in \irB(\irH)$ is a contraction for which $A_T = A$ holds.
\end{lem}

\begin{proof}
By (iii) we obtain 
\[ 
\|T^* y_k\| = \|A^{1/2}U^*A^{-1/2} y_k\| \leq \|A^{1/2}UU^*A^{-1/2} y_k\| = \|y_k\|, 
\]
which gives that $T$ is indeed a contraction.

Consider an arbitrary vector $y_k\in\irY_k$ ($k\in\Z$).
The following inequality holds for any $\varepsilon > 0$ choosing $n$ large enough:
\[ 
\| T^{*n}T^n y_k - A y_k \| = \| A^{1/2}U^{-n}(A^{-1}-I)U^nA^{1/2} y_k \| 
\]
\[ 
\leq \| A^{1/2}\| \cdot \|(A^{-1}-I)|\irY_{k+n}\| \cdot \|A^{1/2} y_k\| \leq \varepsilon \cdot \|y_k\|. 
\]
This shows that $T^{*n}T^n y_k \to A y_k$ holds for every vector $y_k\in\irY_k$ and number $k\in\Z$. 
But $\irY_k$ is reducing for the operator $T^{*n}T^n$ ($k\in\Z, n\in\N$) which implies $A_T = A$.
\end{proof}

Finally, we present our proof concerning (ii) in Theorem \ref{similar_to_unitary_thm}.

\begin{proof}[Proof of (ii) in Theorem] \ref{similar_to_unitary_thm}.

\smallskip

PART I.

\smallskip
 
We will assume in this part that $\ker (A-I) = \{0\}$ holds. 
First, we choose an arbitrary two-sided sequence $\{a_k\}_{k\in\Z} \subseteq ]\underline{r},1[$ for which the conditions $a_k < a_{k+1}$ $(k\in\Z)$, $\lim_{k\to -\infty} a_k = \underline{r}$ and $\lim_{k\to \infty} a_k = 1$ are valid. 
We will use the notations $\irX_k := \irH([a_k,a_{k+1}[)$ $(k\in\Z)$ and $\irM = \ker(A-\underline{r}I)$ where $\irH(\omega)$ denotes the spectral subspace of $A$ associated with the Borel subset $\omega\subset\R$. 
There are five different cases which will be considered separately.

\smallskip

Case I.1. 
\emph{When $\#\{ k < 0 \colon \dim \irX_k = \aleph_0 \} = \#\{ k \geq 0 \colon \dim \irX_k = \aleph_0 \} = \aleph_0$ holds.} 
Choosing a subsequence if necessary we may suppose that $\dim \irX_k = \aleph_0$ is valid for every $k\in\Z$. 
On the one hand, if additionally $\irM = \{0\}$ is satisfied, then using Lemma \ref{rev_lem} we get what we wanted. 
On the other hand, if $\dim\irM = \aleph_0$, then we choose an orthonormal base $\{e_k\}_{k=-\infty}^{-1}$ in $\irM$, set
\[
\irY_k = 
\left\{ \begin{matrix}
\irX_k\oplus(\C\cdot e_k) & \text{if } k < 0 \\
\irX_k & \text{elsewhere}
\end{matrix}\right.,
\]
and choose an arbitrary unitary operator $U$ such that $U\irY_k = \irY_{k+1}$ ($k\in\Z$) and $U e_{k-1} = e_k$ ($k<0$) hold. Then again, we use Lemma \ref{rev_lem}.

\smallskip

Case I.2. 
\emph{When $\#\{ k < 0 \colon \dim \irX_k = \aleph_0 \} = \aleph_0 > \#\{ k \geq 0 \colon \dim \irX_k = \aleph_0 \}$ is valid.} 
We can assume without loss of generality that the condition $\dim \irX_k = \aleph_0$ is satisfied exactly when $k < 0$. 
We choose an orthonormal base $\{ e_{k,l} \colon k\geq 0, l\geq 0 \}$ in $\irH([a_0,1[)$ such that every vector $e_{k,l}$ is an eigenvector of $A$ associated with the eigenvalue $\alpha_{k,l}$, and the equation $\alpha_{k,l}\leq \alpha_{k+1,l}$ holds for every $k\geq 0, l\geq 0$.
This can be obviously done since $1\in\sigma_e(A)$. 
Set $\irY_k := \vee\{e_{k,l}\colon l\geq 0\}$ if $k\geq 0$ and for $k < 0$ in the same way as in the previous case (bearing in mind whether $\irM = \{0\}$ or not). 
Now, if we choose such a unitary operator $U\in\irB(\irH)$ for which point (i) of Lemma \ref{rev_lem} and the equation $U e_{k-1} = e_k$ ($k<0$) hold, we can easily complete this case.

\smallskip

Case I.3. 
\emph{When $\#\{ k < 0 \colon \dim \irX_k = \aleph_0 \} < \aleph_0 = \#\{ k \geq 0 \colon \dim \irX_k = \aleph_0 \}$ is true.} 
We may suppose that the equation $\dim \irX_k = \aleph_0$ holds if and only if $k \geq 0$. 
The restriction $A|\irH([\underline{r},a_0[)$ is trivially diagonal and trivially acts on an infinite-dimensional subspace, since $\underline{r}\in\sigma_e(A)$. 
Therefore we can choose an orthonormal base $\{e_{k,l} \colon k,l < 0\}$ which consists of $A$-eigenvectors: $A e_{k,l} = \alpha_{k,l} e_{k,l}$, such that the inequality $\alpha_{k-1,l}\leq \alpha_{k,l}$ is fulfilled for all $k, l < 0$. 
If we have
\[ 
\irY_k = \left\{ 
\begin{matrix}
\vee\{e_{k,l} \colon l < 0\} & \text{if } k < 0\\
\irX_k & \text{elsewhere}
\end{matrix}\right. 
\]
and a unitary operator $U\in\irB(\irH)$ such that the equations $U\irY_k = \irY_{k+1}$ ($k\in\Z$) and $U e_{k-1,l} = e_{k,l}$ ($k<0$) are satisfied, then again by Lemma \ref{rev_lem} we easily complete this case.

\smallskip

Case I.4. 
\emph{When $0 < \#\{ k \in \Z \colon \dim \irX_k = \aleph_0 \} < \aleph_0$ is fulfilled.} 
Without loss of generality we may suppose that the equation $\dim \irX_k = \aleph_0$ is true exactly when $k = 0$. 
Obviously, both of the restrictions $A|\irH([\underline{r},a_0[)$ and $A|\irH([a_1,1[)$ are diagonal operators and are acting on infinite-dimensional subspaces. 
We choose orthonormal bases $\{e_{k,l} \colon k<0, l\in\Z\}$ and $\{e_{k,l} \colon k>0, l\in\Z\}$ in the subspaces $\irH([\underline{r},a_0[)$ and $\irH([a_1,1[)$, respectively, such that $A e_{k,l} = \alpha_{k,l}e_{k,l}$ holds for some positive numbers $\{\alpha_{k,l} \colon k,l\in\Z, k\neq 0\}$, and $\alpha_{k,l} \leq \alpha_{k',l}$ is valid for arbitrary integers $k,k',l$ for which $k\neq 0, k'\neq 0, k < k'$. 
Now, if 
\[ 
\irY_k = \left\{ \begin{matrix}
\irX_0 & \text{if } k = 0\\
\vee\{e_{k,l} \colon l\in\Z \} & \text{if } k \neq 0
\end{matrix}\right. 
\]
and $U\in\irB(\irH)$ is a unitary operator such that $U\irY_k = \irY_{k+1}$ ($k\in\Z$) and $U e_{k-1,l} = e_{k,l}$ ($k<0, k > 1$) are satisfied, then again, we can use Lemma \ref{rev_lem} in order to complete this case.

\smallskip

Case I.5. 
\emph{When $\dim \irX_k < \aleph_0$ is fulfilled for all $k\in\Z$.} 
We can use a very similar argument as in the previous case, therefore we omit the details.

\smallskip

PART II.

\smallskip
 
In this part we assume that $\dim\ker (A-I) > 0$. 
It is easy to see that exactly one of the equations $\ker (A-I) = \irH$ and $\dim \ker (A-I)^\perp = \aleph_0$ holds, because $\underline{r}\in\sigma_e(A)$ is valid. 
It is also quite obvious that if $T'\in\irB(\irH')$ and $T''\in\irB(\irH'')$ are two contractions such that both of them are similar to a unitary operator, then $T'\oplus T''\in\irB(\irH'\oplus\irH'')$ is also a contraction that is similar to a unitary operator and we have $A_{T'\oplus T"} = A_{T'}\oplus A_{T"}$.

\smallskip

Case II.1. 
\emph{When either $A = I$ or $1\in\sigma_e(A|\ker (A-I)^\perp)$.} 
If $A = I$, then we simply set $T = I$. 
If the relation $1\in\sigma_e(A|\ker (A-I)^\perp)$ holds, then by PART I we can choose a contraction $T'\in\irB(\ker (A-I)^\perp)$ that is similar to a unitary operator and for which $A_{T'} = A|\ker (A-I)^\perp$ holds. 
If $T = T'\oplus I \in\irB(\ker (A-I)^\perp \oplus \ker (A-I))$, we get $A_T = A$.

\smallskip

Case II.2. 
\emph{When $\dim \ker (A-I)^\perp = \aleph_0$ and $1\notin\sigma_e(A|\ker (A-I)^\perp)$.} 
Obviously, the conditions $\dim\ker (A-I) = \aleph_0$ and $r(A|\ker (A-I)^\perp) < 1$ hold. 
Consider a sequence $\{a_k\}_{k=-\infty}^0 \subseteq [\underline{r},r(A|\ker (A-I)^\perp)]$ such that $a_{k-1} < a_k$ ($k \leq 0$), $a_0 = r(A|\ker (A-I)^\perp)$ and $\lim_{k\to -\infty} a_k = \underline{r}$ are satisfied. 
Set $\irX_k := \irH([a_k,a_{k+1}[)$ ($k < 0$). 
If the equation $\dim \irX_k = \aleph_0$ holds for infinitely many $k < 0$, then we may assume that it holds for every $k<0$. 
Assume first that this happens. 
Then we simply choose $\irY_k = \irX_k$ ($k\in\Z$) where $\ker (A-I) = \oplus_{k=0}^\infty \irX_k, \dim\irX_k = \aleph_0$ $(k\geq 0)$, and we use Lemma \ref{rev_lem}. 
In fact, this is the case when the scalar valued spectral measure of $A|\ker (A-I)^\perp$ is continuous (i.e. there are no atoms).

Second, we suppose that the equation $\dim \irX_k = \aleph_0$ holds only for finitely many numbers $k < 0$. 
We consider the decomposition $\ker (A-I)^\perp = \irK_a \oplus \irK_c$ where both of the subspaces $\irK_a$ and $\irK_c$ are $A$-invariant and the scalar valued spectral measure of the restrictions $A|\irK_a$ and $A|\irK_c$ are purely atomic and continuous, respectively.
If $\irK_c \neq \{0\}$, then we choose a splitting $\ker (A-I) = \irL_1 \oplus \irL_2$ where $\dim \irL_1 = \dim \irL_2 = \aleph_0$ holds. 
By the previous paragraph there is a contraction $T' \in\irB(\irL_1\oplus \irK_c)$ which is similar to a unitary operator and for which $A_{T'} = A|(\irL_1\oplus \irK_c)$ holds. 
This shows that it is enough to consider the case when $A$ is diagonal.

If $A$ is diagonal, then the eigenvalues (counting with their multiplicities) $\{\alpha_{k,l}\colon k,l\in\Z\}$ can be easily ordered in a way such that $\alpha_{k,l} \leq \alpha_{k+1,l}$ ($l,k \in\Z$) and $\alpha_{k,l} = 1$ ($k,l \in\Z, k > 0$) hold. 
Now, if we choose $\irY_k := \vee\{e_{k,l}\colon l\in\Z\}$ and define the unitary operator $U\in\irB(\irH)$ by the equation $U e_{k,l} = e_{k+1,l}$ ($k,l\in\Z$), then by a straightforward application of Lemma \ref{rev_lem} our proof is completed.
\end{proof}

%---------------------------------------------------------------------------------------------------------------------------------------------

\newpage

\chapter{Injectivity of the commutant mapping} \label{comm_chap}

\section{Statements of the results}
Throughout this chapter $T$ always denotes a contraction. 
Given any $C\in\{T\}'$ there exists exactly one $D\in\{W_T\}'$ such that $X_TC = DX_T$.
This enables us to define the commutant mapping $\gamma_T$ of $T$ in the following way:
\[ 
\gamma = \gamma_T\colon \{T\}'\to\{W_T\}', \; C\mapsto D \text{ such that } X_TC = DX_T. 
\]
It can be shown that $\gamma$ is a contractive algebra-homomorphism (see \cite[Section IX.1]{NFBK} for further details). 
This commutant mapping is among the few links which relate the contraction to a well-understood operator. 
It can be exploited to get structure theorems or stability results, see e.g. \cite{Ba,Ke_PAMS,Ke_cycl,Ke_shifttype,KV}. 
Hence it is of interest to study its properties. 
Our purpose in \cite{GeKe} was to examine the injectivity of $\gamma$, and in this chapter we give the results obtained in that publication. 
If $T$ is asymptotically non-vanishing, then $X_T$ and hence $\gamma_T$ are clearly injective. 
A natural and non-trivial question is that if the commutant mapping is injective, is necessarily $T$ of class $C_{1\cdot}$? 
However, this is not true. 
A counterexample will be given which will justify it. 
But before that we will prove the next two results. 
The first one provides four necessary conditions for injectivity. 
For a contraction $T$ let $P_0$ denote the orthogonal projection onto the stable subspace. 
The compressions $P_0 T|\irH_0$ and $(I-P_0) T|\irH_0^\perp$ are denoted by $T_{00}$ and $T_{11}$, respectively.

\begin{theorem}\label{comm_nec_thm}
If the commutant mapping $\gamma$ of the contraction $T\in\irB(\irH)$ is injective, then
\begin{itemize}
\item[\textup{(i)}] $\irI(T_{11},T_{00})=\{0\}$,
\item[\textup{(ii)}] $\overline{\sigma_{ap}(T_{00}^*)} \cap \sigma_{ap}(T_{11})\neq \emptyset$,
\item[\textup{(iii)}] $\sigma_p(T)\cap \overline{\sigma_p(T^*)} \cap\D = \emptyset$, and
\item[\textup{(iv)}] there is no direct decomposition $\irH = \irM_0 \dotplus \irM_1$ such that $\irM_0, \irM_1$ are invariant subspaces 
of $T$ and $\{0\}\ne \irM_0\subset\irH_0$.
\end{itemize}
\end{theorem}

In view of condition (iii) of the previous theorem it is of interest to express the point spectra of $T$ and $T^*$ in terms of the matrix entries. 
We will give such a result as Proposition \ref{point_sp_prop}.

For an operator $A$ on a Hilbert space $\irF$, and for a complex number $\lambda$, the root subspace of $A$ corresponding to $\lambda$ is defined by $\widetilde\ker(A-\lambda I):= \vee_{j=1}^\infty \ker(A-\lambda I)^j$. 
We say that the operator $A$ has a generating root subspace system if $\irF= \vee\left\{\widetilde\ker(A-\lambda I): \lambda\in\sigma_p(A)\right\}$. 

Our second result is about sufficiency. 
Namely, it shows that under certain circumstances (iii) of Theorem \ref{comm_nec_thm} is sufficient for the injectivity of $\gamma$.

\begin{theorem}\label{comm_suff_thm}
Let us assume that the stable component $T_{00}$ of the contraction $T$ satisfies the following conditions:
\begin{itemize}
\item[\textup{(i)}] $\sigma_p(T_{00}^*)\subset \overline{\sigma_p(T_{00})}$,
\item[\textup{(ii)}] $T_{00}^*$ has a generating root subspace system.
\end{itemize}
Then $\gamma$ is injective if and only if $\sigma_p(T)\cap \overline{\sigma_p(T^*)}\cap\D=\emptyset$.
\end{theorem}

We recall that a c.n.u. contraction $T$ is called a $C_0$-contraction, if there exists a function $\vartheta\in H^\infty$ such that $\vartheta(T) = 0$.
If $\vartheta = \vartheta_i\vartheta_e$ where $\vartheta_i$ and $\vartheta_e$ are the inner and outer part of $\vartheta$, respectively, then $\vartheta_i(T) = 0$.
It can be shown that here exits an inner function $\vartheta_T\in H^\infty$ such that $\vartheta_T(T) = 0$, and whenever $\vartheta\in H^\infty, \vartheta(T) = 0$ holds, the function $\vartheta$ is a multiple of $\vartheta_T$.
This $\vartheta_T$ is called the minimal function of $T$.
The conditions of the previous theorem are satisfied in a large class of $C_0$-contractions.

\begin{corollary}\label{Blaschke_cor}
Let us assume that the stable component $T_{00}$ of the contraction $T$ is a $C_0$-contraction, and the 
minimal function $\vartheta$ of $T_{00}$ is a Blaschke product.
Then $\gamma$ is injective if and only if $\sigma_p(T)\cap \overline{\sigma_p(T^*)}\cap\D=\emptyset$.
\end{corollary}

After that we will be able to present our example for a non-$C_{1\cdot}$ contraction which has injective commutant mapping. 
This will be followed by investigating the injectivity of commutant mappings for quasisimilar contractions and orthogonal sum of contractions.
Two operators $A,B\in\irB(\irH)$ are quasisimilar if the quasiaffinities (i.e. operators with trivial kernel and dense range) $X\in\irI(A,B)$ and $Y\in\irI(A,B)$ exist.
We say that $T\in\irB(\irH)$ is in stable relation with $T'\in\irB(\irH')$, if $C\in\irI(T,T')$ and $\ran C\subset \irH'_0(T')$ imply $C=0$, and if $C'\in\irI(T',T)$ and $\ran C'\subset\irH_0(T)$ imply $C'=0$. 
Our result concerning quasisimilar contractions and orthogonal sum of contractions reads as follows.

\begin{theorem}\label{comm_quasi_oplus_thm}
Let $T\in\irB(\irH)$ and $T'\in\irB(\irH')$ be contractions.
\begin{itemize}
\item[\textup{(i)}] If $T$ and $T'$ are quasisimilar, then $\gamma_T$ is injective if and only if $\gamma_{T'}$ is injective.
\item[\textup{(ii)}] The commutant mapping $\widetilde\gamma := \gamma_{T\oplus T'}$ is injective if and only if $\gamma_T$ and $\gamma_{T'}$ are injective, and $T$ is in stable relation with $T'$.
\item[\textup{(iii)}] Assuming $T = T'$, the commutant mapping $\widetilde\gamma = \gamma_{T\oplus T}$ is injective if and only if $\gamma_T$ is injective.
\end{itemize}
\end{theorem}

In particular if the contraction $T\in\irB(\irH)$ is quasisimilar to a normal operator $N\in\irB(\irG)$, then $\gamma_T$ is injective if and only if $T$ is of class $C_{11}$. 
We would like to point out that (ii)-(iii) of the above theorem can be extended quite easily for orthogonal sum of countably many contractions.
We also note that the injectivity of $\gamma_T$ and $\gamma_{T'}$ does not imply the injectivity of $\widetilde\gamma$.
This will be verified by a concrete example.

Finally, using certain properties of shift operators, we will provide an example for which (i)-(iv) of Theorem \ref{comm_nec_thm} are satisfied but still $\gamma_T$ is not injective. 
This shows that these four conditions together are not enough to provide a characterization of injectivity of $\gamma_T$.

\section{Proofs}
Suppose that $\gamma$ is injective and consider a $C\in\{T\}'$ such that $\ran C\subset\irH_0$.
Then $0X=0=XC=\gamma(C)X$ implies $\gamma(C)=0$, whence $C=0$ follows.
Second, if $\gamma$ is not injective, then $\gamma(C)=0$ for some non-zero $C\in\{T\}'$, and $XC=\gamma(C)X=0X=0$ yields $\ran C\subset\irH_0$.
Therefore $\gamma$ is injective if and only if the only operator in $\{T\}'$, 
whose range is included in the stable subspace $\irH_0$, is the zero operator.
In particular, if $T\in C_{0\cdot}$ then $\gamma\equiv 0$ is highly non-injective,
while $T\in C_{1\cdot}$ evidently yields that $\gamma$ is injective.
Therefore we are interested in the mixed case $0\ne\irH_0\ne\irH$.

Throughout this chapter we will use Lemma \ref{Ker_dec_lem} many times. 
But here we will use the following more convenient notation:
\[ 
T = \left(\begin{matrix}
T_{00} & T_{01} \\
0 & T_{11}
\end{matrix}\right) \in\irB(\irH_0\oplus\irH_0^\perp).
\]
Now we prove the following lemma.

\begin{lem}\label{basic_lem}
The commutant mapping $\gamma$ is injective if and only if the conditions
\begin{equation}\label{basic_inj_eq}
T_{00}C_{01}-C_{01}T_{11}=C_{00}T_{01} \quad \hbox{ and } \quad C_{00}\in\{T_{00}\}'
\end{equation}
imply $C_{00}=0$ and $C_{01}=0$.
\end{lem}

\begin{proof}
The range of any operator $C\in\irB(\irH)$ is included in $\irH_0$ precisely when its matrix is of the form
\[ 
C = \left(\begin{matrix}
C_{00} & C_{01} \\
0 & 0
\end{matrix}\right)\in\irB(\irH_0\oplus\irH_0^\perp).
\]
Furthermore, the equation $CT=TC$ reduces to the equations $C_{00}T_{00}= T_{00}C_{00}$ and $C_{00}T_{01}+C_{01}T_{11}= T_{00}C_{01}$.
Thus, by the observations of the first paragraph, the proof is complete.
\end{proof}

Now we are in a position to give the verification of Theorem \ref{comm_nec_thm}.

\begin{proof}[Proof of Theorem \ref{comm_nec_thm}]
Taking any $C_{01}\in\irI(T_{11},T_{00})$, the condition \eqref{basic_inj_eq} is fulfilled with $C_{00}=0$.
Hence (i) must hold, if $\gamma$ is injective. 
If (ii) fails, then the mapping $\irT\in\irB(\irB(\irH_1,\irH_0))$, defined by
$\irT\colon Y\mapsto T_{00}Y-YT_{11}$, is surjective (see \cite{DaRo}). 
Thus \eqref{basic_inj_eq} can be satisfied with $C_{00}=I$ and $C_{01}= \irT^{-1}(T_{01})$, and so
$\gamma$ is not injective.
If (iii) fails, then there exist $\lambda\in\D$ and unit vectors $u,v\in\irH$ such that $Tu=\lambda u$ and $T^*v= \overline{\lambda} v$.
Then $\ran(u\otimes v)= \C u\subset\irH_0$, and
\[
(u\otimes v)T= u\otimes T^* v= u \otimes(\overline{\lambda} v)= (\lambda u)\otimes v= T(u\otimes v),
\]
hence $\gamma$ cannot be injective.
Finally, if (iv) fails, then the projection $P$ onto $\irM_0$ parallel to $\irM_1$, commutes with $T$ and is transformed to zero by $\gamma$.
\end{proof}

As we mentioned before by (iii) of Theorem \ref{comm_nec_thm} it is of interest to express the point spectra of $T$ and $T^*$ in terms of the matrix entries. 
The next proposition is about that problem.

\begin{prop}\label{point_sp_prop}
For any $\lambda\in\D$, we have:
\begin{itemize}
\item[\textup{(i)}] $\lambda\in\sigma_p(T)$ if and only if $\lambda\in\sigma_p(T_{00})$,
\item[\textup{(ii)}] $\lambda\in\sigma_p(T^*)$ if and only if $\lambda\in\sigma_p(T_{11}^*)$ or
\[ 
T_{01}^*\left(\ker(T_{00}^*-\lambda I)\setminus\{0\}\right)\cap (T_{11}^*-\lambda I)\irH_0^\perp\ne\emptyset. 
\]
\end{itemize}
\end{prop}

\begin{proof}
(i): Let $h=h_0\oplus h_1\in \irH_0\oplus\irH_0^\perp$ be a non-zero vector. 
The equation $Th=\lambda h$ holds exactly when $T_{00}h_0+T_{01}h_1= \lambda h_0$ and $T_{11}h_1= \lambda h_1$. 
Since $\sigma_p(T_{11})\cap\D=\emptyset$, it follows that $h_1=0$, and so the latter equations are equivalent to $T_{00}h_0=\lambda h_0$ with a non-zero $h_0$.

(ii): the equation $T^*h=\lambda h$ is equivalent to $T_{00}^*h_0= \lambda h_0$ and $T_{01}^*h_0+ T_{11}^* h_1= \lambda h_1$. 
If $\lambda\in \sigma_p(T_{11}^*)$, then these equations hold with $h_0=0$ and with a non-zero $h_1$. 
If $\lambda\not\in \sigma_p(T_{11}^*)$, then the previous two equations hold precisely when $T_{00}^*h_0=\lambda h_0$ is fulfilled with a non-zero $h_0$ and $T_{01}^*h_0= (T_{11}^*-\lambda I)(-h_1)$.
\end{proof}

Now we turn to the sufficiency result of this chapter, but before we prove Theorem \ref{comm_suff_thm}, we need some auxiliary results. 
Given two operators $A\in\irB(\irF)$ and $B\in\irB(\irG)$, let us consider the transformation $\irT_{A,B}\in\irB(\irB(\irG,\irF))$ defined by $\irT_{A,B}\colon Y\mapsto AY-YB$. 
The following lemma is crucial in our proof.

\begin{lem}\label{FG_lem}
If
\[ 
\irF = \vee\left\{\widetilde\ker(A^*-\lambda I): \lambda\in\sigma_p(A^*)\setminus\sigma_p(B^*)\right\} 
\]
or
\[ 
\irG = \vee\left\{\widetilde\ker(B-\lambda I): \lambda \in\sigma_p(B)\setminus\sigma_p(A)\right\}, 
\]
then the mapping $\irT_{A,B}$ is injective.
\end{lem}

\begin{proof}
Let us assume that the second condition holds, and $\irT_{A,B}(Y)=0$ is true for some $Y\in\irB(\irG,\irF)$.
Then for every $\lambda\in\sigma_p(B)\setminus\sigma_p(A),\; j\in\N$ and $g\in\ker(B-\lambda I)^j$ we have $(A-\lambda I)^j Yg=Y(B-\lambda I)^j g=0$, whence $Yg=0$ follows,
since $A-\lambda I$ is injective.
Our assumption yields now that $Y=0$.
The other case can be treated similarly turning to the adjoints.
\end{proof}

Now we give a technical sufficient condition for the injectivity of the commutant mapping.

\begin{prop}\label{suff_before_prop}
Let us assume that the contraction $T\in\irB(\irH)$ satisfies the conditions:
\begin{itemize}
\item[\textup{(i)}] $\irH_0=\vee\left\{\widetilde\ker(T_{00}^*-\lambda I): \lambda\in\sigma_p(T_{00}^*)\setminus\sigma_p(T_{11}^*)\right\}$,
\item[\textup{(ii)}] $T_{01}^*\left(\ker(T_{00}^*-\lambda I)\setminus\{0\}\right)\cap (T_{11}^*-\lambda I)\irH_1=\emptyset\;$ {\it for all} $\;\lambda\in\sigma_p(T_{00}^*)$.
\end{itemize}
Then the commutant mapping $\gamma$ of $T$ is injective.
\end{prop}

\begin{proof}
Assuming that $\gamma$ is not injective, by Lemma \ref{basic_lem} we can find transformations $C_{00}\in\{T_{00}\}'$ and
$C_{01}\in\irB(\irH_1,\irH_0)$ satisfying the equation
\begin{equation}\label{**_eq}
T_{00}C_{01}-C_{01}T_{11}= C_{00}T_{01}
\end{equation}
such that $C_{00}\ne0$ or $C_{01}\ne0$.
Our condition (i) implies by Lemma \ref{FG_lem} that $C_{00}\ne0$.
Another application of (i) yields that we can give $\lambda\in\sigma_p(T_{00}^*)\setminus \sigma_p(T_{11}^*)$ and $r\in\N$ so that
$$C_{00}^*\left(\ker(T_{00}^*-\lambda I)^r\right)\ne\{0\}\quad \hbox{ and } \quad C_{00}^*\left(\ker(T_{00}^*-\lambda I)^{r-1}\right)=\{0\}.$$
Assuming $r=1$, there exists $v\in\ker(T_{00}^*-\lambda I)$ such that $C_{00}^*v\ne0$.
The equations $(T_{00}^*-\lambda I)C_{00}^* v= C_{00}^*(T_{00}^*-\lambda I)v=0$ imply that $C_{00}^*v\in \ker(T_{00}^*-\lambda I)$.
Applying \eqref{**_eq} we obtain
\begin{equation}\label{***_eq}
(T_{11}^*-\lambda I)(-C_{01}^* v)= C_{01}^*(T_{00}^*-\lambda I)v - (T_{11}^*-\lambda I)C_{01}^*v= T_{01}^* C_{00}^*v,
\end{equation}
which contradicts the condition (ii).

Let us assume now that $r>1$.
For every $v\in \ker(T_{00}^*-\lambda I)$ we have \eqref{***_eq} with $T_{01}^*C_{00}^*v=0$.
Since $T_{11}^*-\lambda I$ is injective, we get $C_{01}^*\ker(T^*_{00}-\lambda I)= \{0\}$.
Suppose that 
$$C_{01}^*\left(\ker(T_{00}^*-\lambda I)^{s-1}\right)=\{0\}$$
 holds for an integer $1<s<r$.
Taking any $v\in\ker(T_{00}^*-\lambda I)^s$, we obtain \eqref{***_eq} with $T_{01}^*C_{00}^*v=0$,
whence $C_{01}^*\left(\ker(T_{00}^*-\lambda I)^s\right)= \{0\}$ follows.
By induction we arrive at the relation 
$$C_{01}^*\left(\ker(T_{00}^*-\lambda I)^{r-1}\right)= \{0\}.$$
Setting any $v\in\ker(T_{00}^*-\lambda I)^r$ such that $C_{00}^*v\ne0$, we deduce \eqref{***_eq},
where $0\ne C_{00}^*v\in \ker(T_{00}^*-\lambda I)$, contradicting to (ii).
\end{proof}

We proceed with proving Theorem \ref{comm_suff_thm}

\begin{proof}[Proof of Theorem \ref{comm_suff_thm}]
The necessity of the condition was stated in Theorem \ref{comm_nec_thm}.
Let us assume now that $\sigma_p(T)\cap \overline{\sigma_p(T^*)}\cap\D=\emptyset$.
By Proposition \ref{point_sp_prop} we have $\overline{\sigma_p(T_{00})} \cap \sigma_p(T^*_{11})=\emptyset$, and that
$T_{01}^*\left(\ker(T_{00}^*-\lambda I)\setminus\{0\}\right)\cap (T_{11}^*-\lambda I)\irH_1=\emptyset$
holds for every $\lambda\in \overline{\sigma_p(T_{00})}$.
Since $\sigma_p(T_{00}^*)\subset \overline{\sigma_p(T_{00})}$,
we conclude that the conditions of Proposition \ref{suff_before_prop} are satisfied, and so $\gamma$ is injective.
\end{proof}

We continue with the verification of Corollary \ref{Blaschke_cor}.

\begin{proof}[Proof of Corollary \ref{Blaschke_cor}]
We have to verify the conditions of Theorem \ref{comm_suff_thm}.
Since $\sigma_p(T_{00})\cap\D = \{z\in\D: \vartheta(z)=0\}$ and the minimal function of $T_{00}^*$ is $\widetilde\vartheta(z)= \overline{\vartheta(\overline z)}$, we obtain
that $\sigma_p(T_{00}^*) = \overline{\sigma_p(T_{00})}$ (see \cite[Proposition III.4.7 and Theorem III.5.1]{NFBK}).
The completeness of the root subspace system of $T_{00}^*$ readily follows from \cite[Theorem III.6.3]{NFBK}.
\end{proof}

In what follows we provide an example for a non-$C_{1\cdot}$ contraction which has injective commutant mapping.
This answers a question mentioned before. 
In order to present the desired example, we need the following technical lemma.

\begin{lem}\label{techn_lem}
Let $\irK_i\; (i=1,2,3,4)$ be non-zero Hilbert spaces, 
and $A\in\irB(\irK_1,\irK_2), B\in\irB(\irK_1,\irK_3), Y\in\irB(\irK_2,\irK_4), Z\in\irB(\irK_3,\irK_4)$.
Fixing $A$ and $B$, the equation $YA=ZB$ has only the trivial solution $Y=0, Z=0$ if and only if $A$ and $B$ have dense ranges and
$\ran A^*\cap \ran B^* =\{0\}$.
\end{lem}

\begin{proof}
If $A$ does not have dense range, then taking non-zero vectors $v\in\ker A^*$ and $k_4\in\irK_4$, 
the transformations $Y= k_4\otimes v$ and $Z=0$ provide a non-trivial solution of the equation $YA=ZB$.
The case when $B$ does not have dense range can be discussed similarly.
If $A^*k_2= B^*k_3=u\ne0$ holds with some $k_i\in\irK_i\; (i=2,3)$, 
then setting a non-zero $k_4\in\irK_4$, the transformations $Y=k_4\otimes k_2$ and $Z=k_4\otimes k_3$ provide a non-trivial solution.

Assuming that $A^*, B^*$ are injective and $\ran A^* \cap \ran B^*= \{0\}$, let us consider the adjoint equation $A^*Y^*=B^*Z^*$.
If $Y\ne 0$, then $Y^*k_4\ne 0$ for some $k_4\in \irK_4$, and so the non-zero vector $u= A^*Y^*k_4= B^*Z^* k_4$ belongs to
$\ran A^*\cap \ran B^*$.
Hence $Y=0$, and a similar reasoning shows that $Z=0$.
\end{proof}

Next we present our example.

\begin{exmpl}
\textup{Let $\irH_0$ be a 1-dimensional Hilbert space and $T_{00} = \lambda I\in\irB(\irH_0)$ with some $\lambda\in\D\setminus\{0\}$.
Let us fix a unit vector $e\in\irH_0$.
Let $\irH_1$ be an $\aleph_0$-dimensional Hilbert space, and let $T_{11}\in\irB(\irH_1)$ be a $C_{1\cdot}$-contraction satisfying the conditions
$1\not\in \sigma_p(T_{11}^*T_{11}),\; \overline\lambda\not\in \sigma_p(T_{11}^*)$ and $\ran(T_{11}^*-\overline\lambda I)\ne\irH_1$.
We construct a contraction $T$ on the Hilbert space $\irH= \irH_0\oplus\irH_1$ such that its matrix in this decomposition is of the form}
\begin{equation}
T = 
\left(\begin{matrix}
\lambda I & T_{01}\\
0 & T_{11}
\end{matrix}\right) \in\irB(\irH_0\oplus\irH_0^\perp)
\end{equation}
\textup{It remains to identify $T_{01}$.}

\textup{Since the operator $I- \overline\lambda T_{11}= (I-T_{11}^*T_{11}) + (T_{11}^*-\overline\lambda I) T_{11}$ is invertible, we can find a unit vector $u\in \irH_1$ such that
$u-T_{11}^*T_{11}u$ does not belong to the range of $T_{11}^*-\overline\lambda I$.
We recall that the defect operator of a contraction $T$ is $D_T = \sqrt{I-T^*T}$.
Clearly $D_{T_{00}}= D_{T_{00}^*}= (1-|\lambda|^2)^{1/2}I$.
Now we define $T_{01}= D_{T_{00}}(e\otimes D_{T_{11}}u)D_{T_{11}}$.
By \cite[Lemma IV.2.1]{FF} the operator $T$ is a contraction, whose stable subspace obviously coincides with the non-zero space $\irH_0$.
Let us assume that $T_{00}C_{01}-C_{01}T_{11}= C_{00}T_{01}$ holds 
for some transformations $C_{01}\in\irB(\irH_1,\irH_0)$ and $C_{00}\in \{T_{00}\}'= \irB(\irH_0)$.
It follows that $C_{01}(\lambda I-T_{11})= C_{00}T_{01}$.
Applying Lemma \ref{techn_lem} with $A=\lambda I-T_{11}$ and $B=T_{01}$, we infer that $C_{01}=0$ and $C_{00}=0$.
Thus the commutant mapping $\gamma$ of $T$ is injective by Lemma \ref{basic_lem}.}
\end{exmpl}

We note that a contraction $T_{11}$ with the requested properties can be obtained as a bilateral weighted shift operator with weight sequence
$\{w_i\}_{i=-\i}^\infty$ satisfying the conditions $w_i=\lambda$ for $i\le 0,\; w_i\in ]0,1[$ for $i>0$, and $\prod_{i=1}^\infty w_i>0$.

Next we prove Theorem \ref{comm_quasi_oplus_thm}

\begin{proof}[Proof of Theorem \ref{comm_quasi_oplus_thm}]
(i): Let $Y\in\irI(T,T')$ and $Z\in\irI(T',T)$ be quasiaffinities.
It can be easily seen that $Y\irH_0(T)\subset \irH'_0(T')$ and $Z\irH'_0(T')\subset \irH_0(T)$ hold.
If $\gamma_T$ is not injective, then there exists a non-zero $C\in\{T\}'$ such that $C\irH\subset \irH_0(T)$.
But $C'=YCZ\in\{T'\}'$ is also non-zero, and $C'\irH'= YCZ\irH'\subset YC\irH\subset Y\irH_0(T)\subset \irH'_0(T')$ is satisfied.
Thus $\gamma_{T'}$ is not injective.

(ii): We know that $\widetilde\gamma$ is injective exactly when 
$\widetilde C\in\{\widetilde T\}',\; \widetilde C\widetilde\irH\subset \widetilde\irH_0(\widetilde T) = \irH_0(T)\oplus \irH'_0(T')$
imply $\widetilde C =0$.
Taking the matrix of $\widetilde C$ in the decomposition $\widetilde\irH = \irH\oplus \irH'$, the statements of the theorem
can be easily verified.

(iii): It is clear that $T$ is in stable relation with itself if and only if $\gamma_T$ is injective.
\end{proof}

After that we give example for the case when $\gamma_T, \gamma_{T'}$ are injective but $\widetilde\gamma$ is not injective.

\medskip

\begin{exmpl}
\textup{Let us assume that $\irH=\irH' =\irH_0\oplus\irH_0^\perp$, where $\dim \irH_0=1$ and $\dim \irH_0^\perp =\aleph_0$.
Let $\lambda, \lambda'$ be complex numbers such that $0< |\lambda|< |\lambda'|<1$.
The operators $T_{00}, T_{00}'\in \irB(\irH_0)$ are defined by $T_{00}= \lambda I$ and $T_{00}'= \lambda' I$.
Choosing an orthonormal basis $\{e_n\}_{n=-\infty}^{\infty}$ in $\irH_0^\perp$, let us define the operators
$T_{11}, T_{11}'\in \irB(\irH_0^\perp)$ by the equations}
\[
T_{11}e_n =
\left\{
\begin{matrix}
\lambda e_{n+1} & \text{if } n\le 0\\
(1-n^{-2})e_{n+1} & \text{if } n> 0\\
\end{matrix}
\right.,
\qquad
T_{11}'e_n =
\left\{
\begin{matrix}
\lambda' e_{n+1} & \text{if } n\le 0\\
(1-n^{-2})e_{n+1} & \text{if } n> 0\\
\end{matrix}
\right..
\]
\textup{It is easy to check that 1 is not an eigenvalue of $T_{11}^*T_{11}$ and $T_{11}'^* T_{11}'$, and that the operators
$T_{11}^*-\overline\lambda I$ and $T_{11}'^*- \overline{\lambda'} I$ are injective, but not surjective.
As in Example 9, we can define the transformations $T_{01}, T_{01}'\in \irB(\irH_0^\perp,\irH_0)$ so that the operators}
\[
T = \left(
\begin{matrix}
T_{00} & T_{01}\\
0 & T_{11}
\end{matrix}\right)
\quad\text{and}\quad
T' = \left(
\begin{matrix}
T'_{00} & T'_{01}\\
0 & T'_{11}
\end{matrix}\right)
\]
\textup{are contractions on $\irH$, their stable subspaces coincide with $\irH_0$,
and the corresponding commutant mappings $\gamma_T$ and $\gamma_{T'}$ are injective.}

\textup{The operator $C = \left(
\begin{matrix}
C_{00} & c_{01}\\
0 & 0
\end{matrix}\right)$ intertwines $T$ with $T'$ precisely when
$C_{00}T_{00}= T_{00}'C_{00}$ and $C_{00}T_{01}+C_{01}T_{11}= T_{00}'C_{01}$.
The first equation holds only with $C_{00}=0$, and then the second equation takes the form $C_{01}(T_{11}-\lambda'I)=0$.
Applying the condition $\lambda'\in\D$ it is easy to verify that $\overline{\lambda'}\in \sigma_p(T_{11}^*)$.
Since the range of $T_{11}-\lambda'I$ is not dense in $\irH_0^\perp$, we can find a non-zero transformation $C_{01}$ satisfying $C_{01}(T_{11}-\lambda'I)=0$.
Thus the contraction $T$ is not in stable relation with $T'$, and so
the commutant mapping $\widetilde\gamma$ of the orthogonal sum $\widetilde T= T\oplus T'$ is not injective.}
\end{exmpl}

Our final aim in this chapter is to provide such a concrete contraction $T\in\irB(\irH)$ which satisfies (i)--(iv) of Theorem \ref{comm_nec_thm} but still $\gamma_T$ is not injective. 
Before that we give a reformulation of condition (iv) of Theorem \ref{comm_nec_thm}. 
We say that an operator $A\in\irB(\irF)$ is strongly irreducible if the only decomposition $\irF = \irF_1 \dotplus\irF_2$, where $\irF_1$ and $\irF_2$ are invariant subspaces of $A$, is the trivial decomposition.

\begin{prop}
Let us assume that the stable component $T_{00}$ of the contraction $T\in\irB(\irH)$ is strongly irreducible. 
Then condition (iv) of Theorem \ref{comm_nec_thm} is equivalent to $T_{01}\not\in \ran \irT_{T_{00},T_{11}}$.
\end{prop}

\begin{proof}
If (iv) fails, then the invariant subspaces $\irM_0, \irM_1$ of $T$ exist such that $\{0\}\ne \irM_0\subset\irH_0$ and $\irM_0\dotplus\irM_1=\irH$.
Let $\irN_0$ denote the stable subspace of the restriction $T|\irM_1$.
Since $\irM_0\dotplus\irN_0=\irH_0$ and $T_{00}$ is strongly irreducible, it follows that $\irM_0=\irH_0$.
Let $P\in\irB(\irH)$ be the projection onto $\irH_0$ parallel to $\irM_1$, and let us consider its matrix 
$P = \left(\begin{matrix}
I & P_{01}\\
0 & 0
\end{matrix}\right)$ 
in the decomposition $\irH=\irH_0\oplus\irH_0^\perp$.
Since $P$ commutes with $T$, we infer that $T_{01}+P_{01}T_{11}= T_{00}P_{01}$, whence
$T_{01}= T_{00}P_{01}-P_{01}T_{11}\in \ran \irT_{T_{00},T_{11}}$ follows.

Conversely, assuming $T_{01}\in\ran\irT_{T_{00},T_{11}}$ there exists a transformation 
$P_{01}\in \irB(\irH_0^\perp,\irH_0)$ such that $T_{00}P_{01}-P_{01}T_{11}= T_{01}$ is valid.
Then the operator 
$P = \left(\begin{matrix}
I & P_{01}\\
0 & 0
\end{matrix}\right)$ 
is a projection, commuting with $T$.
Hence condition (iv) fails with the subspaces $\irM_0=\irH_0$ and $\irM_1=\ker P$.
\end{proof}

We finish this chapter with providing the desired example as follows.

\begin{theorem}
There exists a contraction $T\in\irB(\irH)$ such that the conditions (i)--(iv) of Theorem \ref{comm_nec_thm} hold, but the commutant mapping of $T$ is not injective.
\end{theorem}

\begin{proof}
Let $\irH_0$ and $\irH_1$ be $\aleph_0$-dimensional Hilbert spaces, and let us form the orthogonal sum $\irH= \irH_0\oplus\irH_1$.
Let us fix an orthonormal basis $\{e_j: j\in\N\}$ in $\irH_0$, and let $\{f_{j,i}: j\in\N, i\in\Z\}$ be
an orthonormal basis in $\irH_1$.
The operator $T_{00}\in\irB(\irH_0)$ is defined by $T_{00}e_j= (j+2)^{-1}e_{j+1}$ for $j\in\N$.
It is easy to see that $T_{00}$ is a strongly irreducible $C_{00}$-contraction,
furthermore $D_{T_{00}^*}e_1=e_1$ and $D_{T_{00}^*}e_j= (1-(j+1)^{-2})^{1/2}e_j$ for $j>1$.
The operator $T_{11}\in\irB(\irH_1)$ is defined by 
\[
T_{11}f_{j,i} =
\left\{
\begin{matrix}
(j+2)^{-1}f_{j,1} & \text{if } i=0, j\in\N\\
f_{j,i+1} & \text{if } i\in\Z\setminus\{0\}, j\in\N\\
\end{matrix}
\right..
\]
It is clear that $T_{11}$ is a $C_{11}$-contraction, where 
\[
D_{T_{11}}f_{j,i} = 
\left\{
\begin{matrix}
(1-(j+2)^{-2})^{1/2}f_{j,0} & \text{if } i=0, j\in\N\\
0 & \text{if } i\in\Z\setminus\{0\}, j\in\N\\
\end{matrix}
\right..
\]
The partial isometry $V\in\irB(\irH_1,\irH_0)$ is defined by 
\[
Vf_{j,i} =
\left\{
\begin{matrix}
e_j & \text{if } i=0, j\in\N\\
0 & \text{if } i\in\Z\setminus\{0\}, j\in\N\\
\end{matrix}
\right..
\]
Setting $T_{01}= \tfrac{1}{2}V\in\irB(\irH_1,\irH_0)$, let us consider the operator
\[ 
T = \left(\begin{matrix}
T_{00} & T_{01} \\
0 & T_{11}
\end{matrix}\right) \in\irB(\irH).
\]
Applying \cite[Lemma IV.2.1]{FF}, it is easy to verify that $T$ is a contraction, whose stable subspace obviously coincides with $\irH_0$.

Since $T_{00}$ is of class $C_{00}$ and $T_{11}$ is of class $C_{11}$, we infer that $\irI(T_{00}^*,T_{11}^*)= \{0\}$, whence $\irI(T_{11},T_{00})= \{0\}$ follows. 
Hence condition (i) of Theorem \ref{comm_nec_thm} holds.
Since $\sigma_p(T_{00})=\emptyset$, it follows by Proposition \ref{point_sp_prop} that $\sigma_p(T)\cap\D=\emptyset$. 
Thus condition (iii) is also fulfilled.

Following the method of \cite{Fi} we show that $V$ is not in the range of the transformation $\irT_{T_{00},T_{11}}$.
For any $Z\in\irB(\irH_1,\irH_0)$ and $j\in\N$, we have
\[ 
\|T_{00}Z-ZT_{11}-V\| \ge \|(T_{00}Z-ZT_{11}-V)f_{j,0}\| \ge \|T_{00}Zf_{j,0}-e_j\|- \|Z\| \|T_{11}f_{j,0}\|. 
\]
Furthermore, for any $j\in\N$ we have
\[ 
\|T_{00}Zf_{j,0}-e_j\|^2 = \|T_{00}Zf_{j,0}\|^2 - 2 \re\langle Zf_{j,0}, T_{00}^*e_j\rangle + 1 \ge 1- 2\|Z\| \|T_{00}^*e_j\|. 
\]
Since $\lim_{j\to\infty}\|T_{11}f_{j,0}\|=0$ and $\lim_{j\to\infty}\|T_{00}^*e_j\|=0$, we obtain that
$\|T_{00}Z-ZT_{11}-V\|\ge 1$, and so $V\not\in \ran\irT_{T_{00},T_{11}}$.
Since $\irT_{T_{00},T_{11}}$ is not surjective, we infer by \cite{DaRo} that the condition (ii) must hold.
Applying the previous proposition we also conclude that condition (iv) is valid.

Therefore, all conditions of Theorem \ref{comm_nec_thm} hold.
It remains to show that the commutant mapping $\gamma$ is not injective.
By Lemma \ref{basic_lem} it is enough to verify that $-T_{00}T_{01} \in \ran\irT_{T_{00},T_{11}}$.
Then an appropriate $C_{01}$ can be found for $C_{00}=-T_{00}\ne 0$.

Let $\widehat{\irH_1}$ denote the linear span of the basis vectors $f_{j,i}$ in $\irH_1$.
Let us consider the linear transformation $Y_0\colon \widehat{\irH_1}\to \irH_0$ defined by
\[
Y_0f_{j,i} =
\left\{
\begin{matrix}
0 & \text{if } i\le 0\\
e_{j+1} & \text{if } i=1\\
\left(\prod_{k=j+3}^{j+i+1} k^{-1}\right)e_{j+i} & \text{if } i\ge2\\
\end{matrix}
\right..
\]
To show that $Y_0$ is bounded, we note first that $e_1$ is orthogonal to $Y_0\widehat{\irH_1}$, and
$Y_0^{-1}(\C e_2)=\C f_{1,1}$.
For any $j\ge3$, the inverse image $Y_0^{-1}(\C e_j)$ is the linear span of the vectors $f_{j-1,1}, f_{j-2,2}, \dots, f_{1,j-1}$.
For any choice of complex numbers $\{\xi_i\}_{i=1}^{j-1}$ we have
\[ 
\left\|Y_0\sum_{i=1}^{j-1}\xi_i f_{j-i,i}\right\| \le \sum_{i=1}^{j-1} |\xi_i| \|Y_0f_{j-i,i}\| \le \left(\sum_{i=1}^{j-1} |\xi_i|^2\right)^{1/2}\cdot \left(\sum_{i=1}^{j-1} \|Y_0f_{j-i,i}\|^2\right)^{1/2}, 
\]
where
\[ 
\sum_{i=1}^{j-1} \|Y_0f_{j-i,i}\|^2 = 1 + \sum_{i=2}^{j-1} \prod_{k=j-i+3}^{j+1} k^{-2}\le 1 + (j-2) (j+1)^{-2}\le 2.
\]
Since these inverse images linearly generate $\widehat{\irH_1}$, we obtain that $Y_0$ is indeed bounded.
Thus $Y_0$ can be uniquely extended to a transformation $Y\in\irB(\irH_1,\irH_0)$.
Simple computations show that $T_{00}Yf_{j,i}-YT_{11}f_{j,i}= -T_{00}Vf_{j,i}$ holds for every $j\in\N$ and $i\in\Z$.
We conclude that $T_{00}Y-YT_{11}=-T_{00}V$.
Hence $-T_{00}T_{01}= -\tfrac{1}{2}T_{00}V\in \ran \irT_{T_{00},T_{11}}$, what we wanted to prove.
\end{proof}

The complete characterization of injectivity of the commutant mapping remains open.

%---------------------------------------------------------------------------------------------------------------------------------------------

\newpage

\chapter{Cyclic properties of weighted shifts on directed trees} \label{tree_chap}

\section{Statements of the results}
The classes of the so-called weighted bilateral, unilateral or backward shift operators (\cite{Ni, Shields}) are very useful for an operator theorist. 
Besides normal operators these are the next natural classes on which conjectures could be tested. 
Recently Z. J. Jab\l onski, I. B. Jung and J. Stochel defined a natural generalization of these classes in \cite{JJS}. 
They were interested among others in hyponormality, co-hyponormality, subnormality etc., and they provided many examples for several unanswered questions.  

In this chapter we will study cyclic properties of bounded (mainly contractive) weighted shift operators on directed trees. 
Namely first we will explore their asymptotic behaviour, and as an application we will obtain some results concerning cyclicity.
These are the results of \cite{Ge_tree}.

We recall some definitions from \cite{JJS}. 
The pair $\irT=(V,E)$ is a directed graph if $V$ is an arbitrary (usually infinite) set and $E\subseteq (V\times V)\setminus \{(v,v)\colon v \in V\}$. 
We call an element of $V$ and $E$ a vertex and a (directed) edge of $\irT$, respectively. 
We say that $\irT$ is connected if for any two distinct vertices $u,v\in V$ there exists an undirected path between them, i.e. there are finitely many vertices: $u=v_0, v_1,\dots v_n=v\in V, n\in\N$ such that $(v_{j-1},v_j)$ or $(v_j,v_{j-1})\in E$ for every $1\leq j\leq n$. 
The finite sequence of distinct vertices $v_0,v_1,\dots v_n\in V, n\in\N$ is called a (directed) circuit if $(v_{j-1},v_j)\in E$ for all $1\leq j\leq n$ and $(v_n,v_0)\in E$. 
The directed graph $\irT = (V,E)$ is a directed tree if the following three conditions are satisfied:

\begin{itemize}
\item[\textup{(i)}] $\irT$ is connected,
\item[\textup{(ii)}] for each $v\in V$ there exists at most one $u\in V$ such that $(u,v)\in E$, and
\item[\textup{(iii)}] $\irT$ has no circuit.
\end{itemize}

From now on $\irT$ always denotes a directed tree. 
In the directed tree a vertex $v$ is called a child of $u\in V$ if $(u,v)\in E$. 
The set of all children of $u$ is denoted by $\Chi_\irT(u)=\Chi(u)$. 
Conversely, if for a given vertex $v$ we can find a (unique) vertex $u$ such that $(u,v)\in E$, then we shall say that $u$ is the parent of $v$. 
We denote $u$ by $\Par_\irT(v) = \Par(v)$. 
We will also use the notation $\Par^k(v) = \underbrace{\Par(\dots (\Par}_{k \text{-times}}(v))\dots)$ if it makes sense, and $\Par^0$ will be the identity map. 

If a vertex has no parent, then we call it a root of $\irT$. 
A directed tree is either rootless or has a unique root (see \cite[Proposition 2.1.1.]{JJS}). 
We will denote this unique root by $\roo_\irT = \roo$, if it exists. 
A subgraph of a directed tree which is itself a directed tree is called a subtree. 
We will use the notation $V^\circ = V\setminus \{\roo\}$. 
If a vertex has no children, then we call it a leaf, and $\irT$ is leafless if it has no leaves. 
The set of all leaves of $\irT$ will be denoted by $\Lea(\irT)$. 
Given a subset $W\subseteq V$ of vertices, we put $\Chi(W) = \cup_{v\in W}\Chi(v)$, $\Chi^0(W) = W$, $\Chi^{n+1}(W) = \Chi(\Chi^{n}(W))$ for all $n\in\N$ and $\Des_\irT(W) = \Des(W)=\bigcup_{n=0}^\infty \Chi^{n}(W)$, where $\Des(W)$ is called the descendants of the subset $W$, and if $W = \{u\}$, then we simply write $\Des(u)$.

If $n\in\N_0$, then the set $\Gen_{n,\irT}(u)=\Gen_n(u)=\bigcup_{j=0}^n\Chi^j(\Par^j(u))$ is called the $n$th generation of $u$ (i.e. we can go up at most $n$ levels and then down the same amount of levels) and $\Gen_\irT(u)=\Gen(u)=\bigcup_{n=0}^\infty\Gen_n(u)$ is the (whole) generation or the level of $u$. 
From the equation (see \cite[Proposition 2.1.6]{JJS})
\begin{equation}\label{level_eq}
V = \bigcup_{n=0}^\infty\Des(\Par^n(u))
\end{equation}
one can easily see that the different levels can be indexed by the integer numbers (or with a subset of the integers) in such a way that if a vertex $v$ is in the $k$th level, then the children of $v$ are in the $(k+1)$th level and $\Par(v)$ is in the $(k-1)$th level whenever $\Par(v)$ makes sense.

The complex Hilbert space $\ell^2(V)$ is the usual space of all square summable complex functions on $V$ with the standard innerproduct
\[ 
\langle f,g\rangle = \sum_{u\in V} f(u)\overline{g(u)}, \quad f,g\in\ell^2(V). 
\]
For $u\in V$ we define $e_u(v)=\delta_{u,v}\in\ell^2(V)$, where $\delta_{u,v}$ is the Kronecker-delta function. 
Obviously the set $\{e_u\colon u\in V\}$ is an orthonormal basis. 
We will refer to $\ell^2(W)$ as the subspace $\vee\{e_w\colon w\in W\}$ for any subset $W\subseteq V$.

Let $\underline{\lambda} = \{\lambda_v\colon v\in V^\circ\}\subseteq\C$ be a set of weights satisfying
\[
\sup\left\{\sqrt{\sum_{v\in\Chi(u)}|\lambda_v|^2}\colon u\in V\right\}<\infty.
\]
Then the weighted shift operator on the directed tree $\irT$ is the operator defined by
\[
\Sl\colon \ell^2(V)\to \ell^2(V), \quad e_u\mapsto \sum_{v\in\Chi(u)} \lambda_v e_v. 
\]
By \cite[Proposition 3.1.8.]{JJS} this defines a bounded linear operator with precise norm $\|\Sl\| = \sup\left\{\sqrt{\sum_{v\in\Chi(u)}|\lambda_v|^2}\colon u\in V\right\}$.

We will consider only bounded weighted shift operators on directed trees, especially contractions (i.e. $\|\Sl\|\leq 1$) in certain parts of the chapter. 
We recall that every $\Sl$ is unitarily equivalent to $S_{|\underline{\lambda}|}$ where $|\underline{\lambda}|:=\{|\lambda_v|\colon v\in V^\circ\}\subseteq[0,\infty[$. 
Moreover, the unitary operator $U$ with $S_{|\underline{\lambda}|} = U\Sl U^*$ can be chosen such that $e_u$ is an eigenvector of $U$ for every $u\in V$. 
It is also proposed that if a weight $\lambda_v$ is zero, then the weighted shift operator on this directed tree is a direct sum of two other weighted shift operators on directed trees (see \cite[Theorem 3.2.1 and Proposition 3.1.6]{JJS}. 
In view of these facts, this chapter will exclusively consider weighted shift operators on directed trees with positive weights.

The boundedness and the condition about weights together imply that every vertex has countably many children. 
Thus, by \eqref{level_eq}, $\ell^2(V)$ is separable.

An operator $T\in\irB(\irH)$ is cyclic if there is a vector such that 
\[ 
\irH_{T,h} = \irH_h := \vee\{T^n h\colon n\in\N_0\} = \{p(T) h \colon p\in \irP_\C\}^- = \irH, 
\]
where $\irP_\C$ denotes the set of all complex polynomials. 
Such an $h\in\irH$ vector is called a cyclic vector for $T$. 
The first theorem concerns a countable orthogonal sum of weighted backward shift operators.

\begin{theorem}\label{backw_cycl_thm}
Suppose that $\{e_{j,k}\colon j\in \irJ, k\in\N_0\}$ is an orthonormal basis in the Hilbert space $\irH$ where $\irJ \neq \emptyset$ is a countable set and $\{w_{j,k}\colon j\in \irJ, k\in\N_0\}\subset \C$ is a bounded set of weights. 
Consider the following operator:
\[ 
Be_{j,k} = \left\{ \begin{matrix}
0 & \text{ if } k=0 \\
w_{j,k-1}e_{j,k-1} & \text{ otherwise}
\end{matrix} \right.. 
\]
\begin{itemize}
\item[\textup{(i)}] If there is no zero weight, then $B$ is cyclic.
\item[\textup{(ii)}] Assume there is no zero weight and there exists a vector $g\in\cap_{n = 1}^\infty \ran(B^n)$ such that for every fixed $j\in\irJ$, $\langle g,e_{j,k}\rangle \neq 0$ is fulfilled for infinitely many $k\in\N_0$. 
Then we can find a cyclic vector from the linear manifold $\cap_{n = 1}^\infty \ran(B^n)$.
\item[\textup{(iii)}] The operator $B$ is cyclic if and only if $B$ has at most one zero weights.
\end{itemize}
\end{theorem}

The proof of Theorem \ref{backw_cycl_thm} was motivated by the solution of \cite[Problem 160]{Halmos}. 
We would like to note that the case in (iii) of Theorem \ref{backw_cycl_thm} when $\#\irJ = 1$ was done by Z. G. Hua in \cite{Hu}. 
However this article was written in Chinese, therefore the author of the dissertation was not able to read the proof. 
The above theorem can be considered as a generalization of Hua's result.

Now we give our results concerning cyclicity of weighted shift operators on directed trees. 
The following quantity:
\[ 
\Br(\irT) = \sum_{u\in V\setminus\Lea(\irT)} (\#\Chi(u)-1) 
\]
is called the branching index of $\irT$. 
By (ii) of \cite[Proposition 3.5.1]{JJS} we have 
\begin{equation} \label{co-rank_eq} \dim(\ran(\Sl)^\perp) = \left\{ \begin{matrix}
1+\Br(\irT) & \text{if } \irT \text{ has a root,} \\
\Br(\irT) & \text{if } \irT \text{ has no root.} 
\end{matrix} \right.
\end{equation}
It is easy to see that for every cyclic operator $T\in\irB(\irH)$ we have $\dim(\ran(T)^\perp)\leq 1$. 
Therefore the only interesting case concerning the cyclicity of $\Sl$ is when $\irT$ has no root and $\Br(\irT) = 1$ (if the branching index is zero, we obtain a usual shift operator). 

\begin{figure}\label{tree_1-2Lea_fig}
\centering
\includegraphics[scale=0.5]{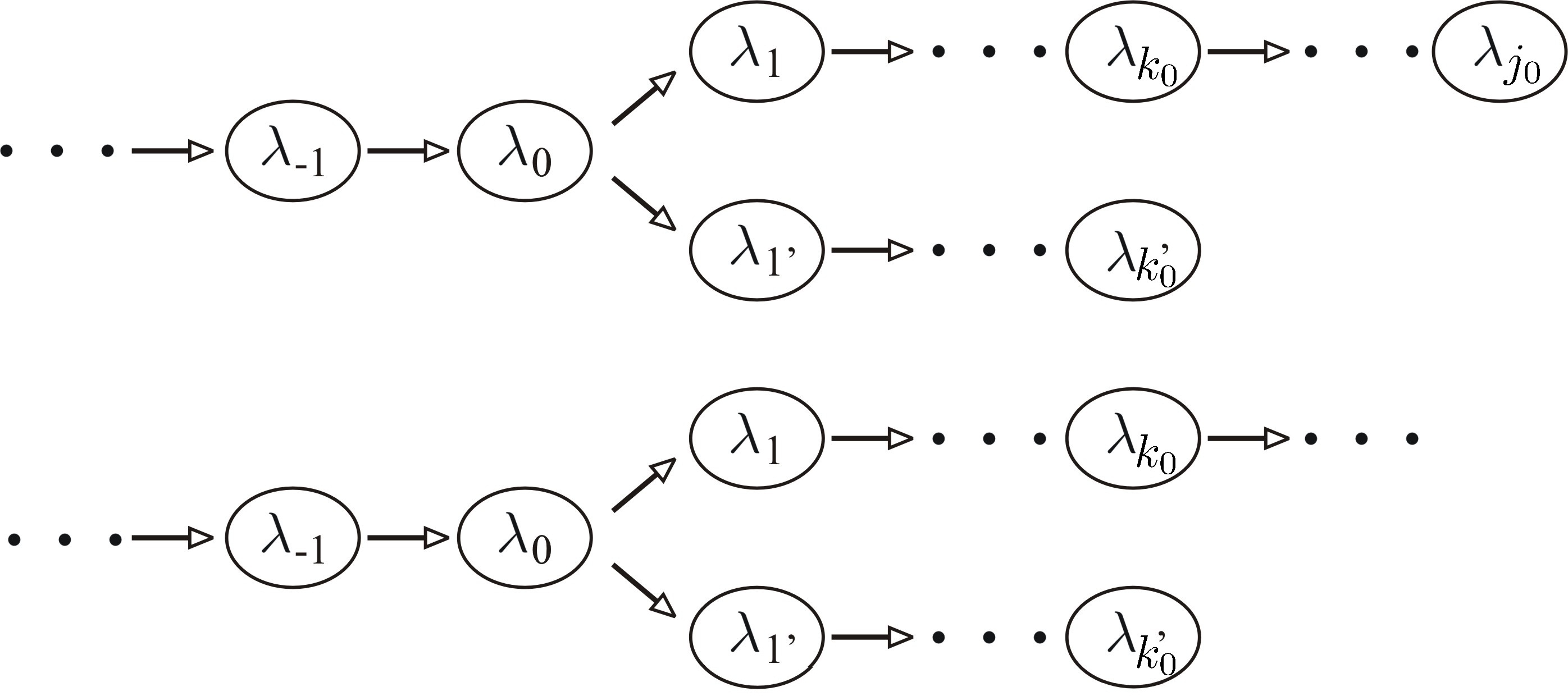}
\caption{When there is at least one leaf.}
\end{figure}

If $\#\Lea(\irT) = 2$, then $\irT$ will be represented by the graph $(V,E)$ with vertices $V = \{j\in\Z\colon j\leq j_0\}\cup\{k'\colon 1\leq k \leq k_0\}$ where we assume $1\leq k_0\leq j_0<\infty$, and edges $E = \{(j-1,j)\colon j\leq j_0\}\cup\{(0,1')\}\cup\{((j-1)',j')\colon 1<j\leq k_0\}$.
The corresponding weights will be $\underline{\lambda} = \{\lambda_v\colon v\in V\}$. 
If $\#\Lea(\irT) = 1$, then it will be represented by a very similar graph (see Figure \ref{tree_1-2Lea_fig}). 

\begin{figure}
\label{tree_noLea_fig}
\includegraphics[scale=0.5]{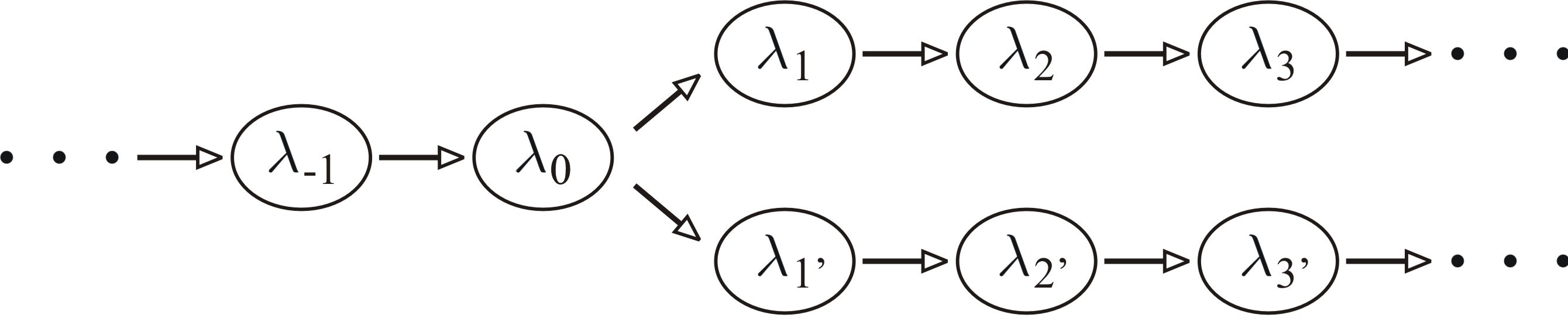}
\caption{$\widetilde{\irT}$}
\end{figure}

If there is no leaf, then usually the associated representative graph will be denoted by another symbol: $\widetilde{\irT}$ (Figure 2). We remind the reader that every weight is positive.

\begin{theorem}\label{tree_cycl_thm}
Suppose that the directed tree $\irT = (V,E)$ has no root and $\Br(\irT) = 1$.
\begin{itemize}
\item[\textup{(i)}] If $\#\Lea(\irT) = 2$, then every bounded weighted shift on $\irT$ is cyclic.
\item[\textup{(ii)}] Suppose $\irT = (V,E)$ has a unique leaf. 
A weighted shift $\Sl$ on $\irT$ is cyclic if and only if the bilateral shift operator $W$ on the subtree $\irT' := (\Z,E\cap(\Z\times\Z))$ with weights $\{\lambda_n\}_{n=-\infty}^\infty = \{\lambda_v\colon \in V\cap\Z\}$ is cyclic. 
In particular, if $\Sl\notin C_{\cdot 0}(\ell^2(V))$ is contractive, then $\Sl$ is cyclic.
\item[\textup{(iii)}] Assume that $\Sl$ is contractive and of class $C_{1\cdot}$. 
Then it has no cyclic vectors.
\item[\textup{(iv)}] There exists a cyclic weighted shift operator $\Sl$ on $\widetilde\irT$.
\end{itemize}
\end{theorem}

As far as we know complete characterization of cyclicity of bilateral weighted shift operators is open.
Our last result concerns cyclicity of the adjoint of weighted shift operators on directed trees.

\begin{theorem}\label{tree_adj_cycl_thm}\noindent
\begin{itemize}
\item[\textup{(i)}] If $\irT$ has a root and the contractive weighted shift operator $\Sl$ on $\irT$ is of class $C_{1\cdot}$, then $\Sl^*$ is cyclic.
\item[\textup{(ii)}] If $\irT$ is rootless, $\Br(\irT)<\infty$ and the weighted shift contraction $\Sl$ on $\irT$ is of class $C_{1\cdot}$, then $\Sl^*$ is cyclic.
\end{itemize}
\end{theorem}

In order to prove the above theorem we will verify that $S\oplus (S^+_{k})^*$ ($k\in\N$) is cyclic where $S$ denotes a simple bilateral shift operator and $S^+_{k}$ denotes the orthogonal sum of $k$ pieces of simple unilateral shift operators.

\section{Proofs}
First we would like to show how the asymptotic behaviour of a Hilbert space contraction can be applied in order to obtain cyclicity results. 
If $T$ is cyclic and has dense range, then $h$ is cyclic if and only if $Th$ is cyclic. 
This and a consequence are stated in the next lemma for Hilbert spaces, but we note that in Banach spaces the proof would be the same. 
This proves, in that case, that the set of cyclic vectors span the whole space. 
In fact, this is always true (see \cite{GeherL} for an elementary proof).

\begin{lem}\label{cyclic_denserenage_nilp_lem}\noindent
\begin{itemize}
\item[\textup{(i)}] If a dense range operator $T$ has a cyclic vector $f$, then $T f$ is also a cyclic vector.
\item[\textup{(ii)}] If $T$ is a cyclic operator which has dense range and $N\in\irB(\C^n)$ is a cyclic nilpotent operator ($n\in\N$), i.e. a 0-Jordan block, then $T\oplus N$ is also cyclic.
\end{itemize}
\end{lem}

\begin{proof}
(i) The set $\{p(T) f\colon p\in\irP_\C\}$ is dense. 
Since $\ran(T)$ is also dense, a dense subset has dense image under $T$, so $\{Tp(T) f = p(T) Tf \colon p\in\irP_\C\}$ is also dense.

(ii) Let us take a cyclic vector $f \in \irH$ for $T$ and a cyclic vector $e \in \C^n$ for $N$. 
We will show that $f\oplus e$ is cyclic for the orthogonal sum $T\oplus N$. 
Of course $\vee\{T^k f \oplus N^k e = T^k f \oplus 0 \colon k\geq n\} = \irH$. 
Thus $0 \oplus N^j e \in \vee\{T^k f \oplus N^k e \colon k \in \N_0\}$ for every $0 \leq j < n$, and hence $f\oplus e$ is a cyclic vector.
\end{proof}

The vector $h\in\irH$ is hypercyclic for $T$ if 
\[ 
\{T^n h\colon n\in\N_0\}^- = \irH. 
\]
Then the operator $T$ is said to be hypercyclic. 

The previous and the next lemma will be used many times throughout this chapter. 

\begin{lem} \label{is_as_forcyclem}\noindent
\begin{itemize}
\item[\textup{(i)}] if $T, Q, Y\in B(\irH)$, $Y$ has dense range, $YT = QY$ holds and $f$ is cyclic (or hypercyclic, resp.) for $T$, then $Yf$ is cyclic (or hypercyclic, resp.) for $Q$,
\item[\textup{(ii)}] if $T\in C_{1\cdot}(\irH)$ is a contraction and the isometric asymptote $V_T$ has no cyclic vectors, then neither has $T$,
\item[\textup{(iii)}] if $T\in C_{1\cdot}(\irH)$ is a contraction and $V_T^*$ has a cyclic vector $g$, then $A_T^{1/2} g$ is cyclic for $T^*$,
\item[\textup{(iv)}] if $T\in C_{\cdot 1}(\irH)$ is a contraction and $U_{T^*}^*$ has a cyclic vector $g$, then $A_{T^*}^{1/2} g$ is cyclic for $T$.
\end{itemize}
\end{lem}

\begin{proof}
(i): Of course $Y p(T) = p(Q) Y$ holds for all $p\in\irP_\C$. 
Let us assume that $f$ is cyclic for $T$, i.e. $\{p(T)f\colon p\in\irP_\C\}$ is dense in $\irH$. 
Then $\{Y p(T)f = p(Q) Y f \colon p\in\irP_\C\}$ is also dense since $Y$ has dense range, which means that $Y f$ is cyclic for $Q$. 
The hypercyclic case is very similar.

The other three points are special cases of (i).
\end{proof}

There is a consequence of the previous lemma.

\begin{corollary}
Suppose that the operator $T$ is hypercyclic. 
Then $T/\|T\|\notin C_{1\cdot}(\irH)$.
\end{corollary}

\begin{proof}
Obviously $T \neq 0$. 
Assume that $T/\|T\|\in C_{1\cdot}(\irH)$ and let us fix a hypercyclic vector $f\in\irH$ for $T$. 
Since $A^{1/2}T^k f = \|T\|^k V_T^kA^{1/2} f$, $\{\|T\|^k V_T^kA^{1/2} f\}_{k=0}^\infty$ should be dense in $\irH$. 
But $\big\|\|T\|^k V_T^kA^{1/2} f\big\| = \|T\|^k \|A^{1/2} f\|$ is bounded or bounded from below, which is a contradiction.
\end{proof}

It is easy to see that for the adjoint of a contractive weighted bilateral shift operator: $S_{\underline{w}}^* e_k = w_{k}e_{k-1}$, $(0 < )|w_k| \leq 1$, the asymptotic limit is defined by 
\[ 
A_* e_k = \Big(\prod_{j\leq k} |w_j|^2\Big) e_k. 
\]
This means that $S_{\underline{w}}^*$ is stable or $S_{\underline{w}} \in C_{\cdot 1}(\ell^2(V))$. 
If $S_{\underline{w}} \in C_{\cdot 1}(\ell^2(\Z))$, then the isometric asymptote $V_{S_{\underline{w}}^*} \in \irB(\ell^2(\Z))$ of $S_{\underline{w}}^*$ is the simple bilateral backward shift operator: $V_{S_{\underline{w}}^*}e_k = e_{k-1}$. 
Since $V_{S_{\underline{w}}^*}^*$ is cyclic, this means that all contractive $C_{\cdot 1}$ bilateral shift operators are cyclic. 

We point out that non-cyclic bilateral shift operators exist. 
The first example was given by B. Beauzamy in \cite{non-cyclic_bil_shift}. 
In \cite[Proposition 42]{Shields} sufficient conditions can be found for cyclicity and non-cyclicity. 
But there is no characterization for cyclic bilateral shift operators which is quite surprising since for other cyclic type properties such characterizations exist (see e.g. \cite{hypercyclic_bil} and \cite{supercyclic_bil}). 
Therefore it is a challenging problem to give this characterization for cyclicity.

We proceed with the proof of Theorem \ref{backw_cycl_thm}.

\begin{proof}[Proof of Theorem \ref{backw_cycl_thm}]
Throughout the proof we may always assume without loss of generality that $\|B\|\leq 1$ holds. 
This happens if and only if the weights are from $\D^-$.

\smallskip

(i): Take a vector of the following form:
\[ 
f = \sum_{l = 1}^\infty \xi_{j_l,k_l}\cdot e_{j_l,k_l} \in \irH, 
\]
such that $\xi_{j_l,k_l}>0$ for every $l\in\N$, $0 < k_{l+1}-k_l\nearrow\infty$ and for any $j\in\irJ$ infinitely many $l\in\N$ exist which satisfy $j_l=j$.

Our aim is to modify $f$ by decreasing its coordinates in a way that they remain positive and after the procedure we obtain a modification of $f$: $\tilde{f}\neq 0$ which is a cyclic vector of $B$. 
We have
\begin{equation} \label{for_cyclicity_eq} 
\begin{matrix}
\displaystyle \frac{1}{\xi_{j_m,k_m} w_{j_m,k_m-1}\dots w_{j_m,k_m-k}}B^k f \\
\\
\displaystyle = e_{j_m,k_m-k} + \sum_{l > m} \frac{\xi_{j_l,k_l}w_{j_l,k_l-1}\dots w_{j_l,k_l-k}}{\xi_{j_m,k_m} w_{j_m,k_m-1}\dots w_{j_m,k_m-k}}\cdot e_{j_l,k_l-k}
\end{matrix}
\end{equation}
for every $m\in\N, k_{m-1} < k \leq k_m$, where we set $k_0 = -1$. 
Let us consider the quantity 
\begin{equation} \label{cyclic_eq}
\begin{gathered}
\Sigma_m := \max\Bigg\{ \sum_{l > m} \bigg|\frac{1}{\xi_{j_m,k_m} w_{j_m,k_m-1}\dots w_{j_m,k_m-k}}B^k f - e_{j_m,k_m-k}\bigg|^2 \colon k_{m-1} < k \leq k_m \Bigg\}\\
= \max\Bigg\{ \sum_{l > m} \bigg|\frac{\xi_{j_l,k_l}w_{j_l,k_l-1}\dots w_{j_l,k_l-k}}{\xi_{j_m,k_m} w_{j_m,k_m-1}\dots w_{j_m,k_m-k}}\bigg|^2 \colon k_{m-1} < k \leq k_m \Bigg\}
\end{gathered}
\end{equation}
which is clearly finite. 
Let us suppose for a moment that $\Sigma_m\leq (1/2)^m$ is satisfied for all $m\in\N$. 
In this case $e_{j,k} \in \irH_{B,f}$ would hold for every $j\in\irJ, k\in\N_0$ and thus $f$ would be a cyclic vector for $B$. 
Therefore our aim during the modification process is that \eqref{cyclic_eq} will hold for the modified vector.

If $\Sigma_1 > 1/2$, then let us change every $\xi_{j_l,k_l}$ to $\frac{\xi_{j_l,k_l}}{\sqrt{2\Sigma_1}}$ for every $l > 1$, otherwise we do not do anything. 
Then with these modified coordinates $\Sigma_1 \leq 1/2$. 
If $\Sigma_2 > 1/4$, then we change every $\xi_{j_l,k_l}$ to $\frac{\xi_{j_l,k_l}}{\sqrt{4\Sigma_2}}$ for every $l > 2$, otherwise we do not modify anything. 
Then $\Sigma_1$ becomes less or equal to than before and $\Sigma_2 \leq 1/4$ \dots Suppose that we have already achieved $\Sigma_j \leq 1/2^j$ for every $1 \leq j \leq m-1$. 
In case when $\Sigma_m > 1/2^m$, we modify $\xi_{j_l,k_l}$ to $\frac{\xi_{j_l,k_l}}{\sqrt{2^m\Sigma_m}}$ for every $l > m$. 
Otherwise we do not change anything. 
Then $\Sigma_j$ becomes less or equal to than before for every $1 \leq j \leq m-1$ and $\Sigma_m \leq 1/2^m$ \dots and so on. 
We notice that every coordinate was modified only finitely many times. 
Therefore this procedure gives us a new vector $\tilde f$ with positive coordinates satisfying \eqref{cyclic_eq}. 
Therefore $\tilde f$ is cyclic for $B$.

\smallskip

(ii): For a vector $x := \sum_{l = 1}^\infty x_{j_l,k_l}\cdot e_{j_l,k_l} \in \irH$, $x \in \ran(B^n)$ holds if and only if $\sum_{l = 1}^\infty |x_{j_l,k_l}|^2/|w_{j_l,k_l}\dots w_{j_l,k_l+n-1}|^2 < \infty$. 
Thus if $x\in \ran(B^n)$ and $y := \sum_{l = 1}^\infty y_{j_l,k_l}\cdot e_{j_l,k_l} \in \irH$ with $|y_{j_l,k_l}| \leq |x_{j_l,k_l}|$ ($l\in\N$), then $y \in \ran(B^n)$ is also fulfilled. 
Assume that in the beginning we could choose a vector $g\in\cap_{n = 1}^\infty \ran(B^n)$ such that for every fixed $j\in\irJ$, $\langle g,e_{j,k}\rangle \neq 0$ is fulfilled for infinitely many $k\in\N_0$. 
Then $\hat{g} := \sum_{j\in\irJ}\sum_{k=0}^\infty \big|\langle g,e_{j,k}\rangle \big| e_{j,k}$ is also in $\cap_{n = 1}^\infty \ran(B^n)$, and by changing some coordinates to zero we can find a vector $f$ which has the form as in the beginning of the proof of point (i). 
Therefore, applying the process above we can find a cyclic vector for $B$ which is in $\cap_{n = 1}^\infty \ran(B^n)$.

\smallskip

(iii): For the sufficiency we may consider the operator $B\oplus N$ where $B\in\irB(\irH)$ is a backward shift operator with strictly positive weights and $N$ is a cyclic nilpotent operator acting on $\C^n$ (for some $n\in\N$) (i.e. a Jordan block with zero diagonals), because this is unitarily equivalent to the operator mentioned in the statement of the theorem. 
Since $B$ has a dense range and it is cyclic, (ii) of Lemma \ref{cyclic_denserenage_nilp_lem} gives us what we wanted.

The necessity is clear since the co-dimension of $\ran(B)^-$ is at most one whenever $B$ is cyclic.
\end{proof}

We proceed with the verification of several results concerning the asymptotic behaviour of a weighted shift operator on a directed tree. 
They will be crucial in the proof of Theorem \ref{tree_cycl_thm} and \ref{tree_adj_cycl_thm}. 
The first step is to calculate the powers of $\Sl$ and their adjoints which was done in \cite[Lemma 2.3.1]{BJJS1}.
Namely, for a weighted shift operator $\Sl$ on the directed tree $\irT$ and for any $n\in\N, u\in V$ the following holds:
\[ 
\Sl ^n e_u = \sum_{v\in\Chi^n(u)} \prod_{j=0}^{n-1}\lambda_{\Par^j(v)}\cdot e_v, 
\]
\[ 
\Sl ^{*n} e_u = \left\{\begin{matrix}
\prod_{j=0}^{n-1}\lambda_{\Par^j(u)}\cdot e_{\Par^n(v)}, & \textit{ if } \Par^n(u) \textit{ makes sense},\\
0, & \textit{ otherwise} 
\end{matrix}\right. 
\]

In this chapter the asymptotic limit of $\Sl$ and $\Sl^*$ will be denoted by $A$ and $A_*$, respectively, and $U$, $U_*$ will stand for the isometric asymptote of them.

We calculate the asymptotic limit $A$ of $\Sl$ as follows.

\begin{prop} \label{aslim_prop}
Let $\Sl$ be a weighted shift operator on $\irT$ which is a contraction. 
Then every $e_u$ is an eigenvector for $A$ 
\[A e_u = \alpha_u e_u \quad \forall\; u\in V,\] 
with the corresponding eigenvalues: $\alpha_u = \lim_{n\to\infty} \sum_{v\in \Chi^n(u)}\prod_{j=0}^{n-1}\lambda_{\Par^j(v)}^2$.
\end{prop}

\begin{proof}
For any $n\in\N$ and $u\in V$, since the asymptotic limit exists
\[ 
\Sl ^{*n}\Sl ^n e_u = \sum_{v\in\Chi^n(u)} \prod_{j=0}^{n-1}\lambda_{\Par^j(v)}\cdot \Sl ^{*n}e_v = \sum_{v\in\Chi^n(u)} \prod_{j=0}^{n-1}\lambda_{\Par^j(v)}^2 \cdot e_{u} \to \alpha_u e_u, 
\]
which implies $A e_u = \alpha_u e_u$.
\end{proof}

The stable subspace of $\Sl$ will be denoted by $\irH_0$. 
We proceed with obtaining some properties of $\irH_0^\perp = \ell^2(V)\ominus\irH_0$. 
Since $A$ is a diagonal operator, there exists a set $V'\subset V$ with the following properties: $\irH_0 = \ell^2(V\setminus V')$ and $\irH_0^\perp = \ell^2(V')$.

\begin{prop} \label{stable_subspc_prop}
The following implications are valid for every contractive weighted shift operator $\Sl$ on $\irT$ and for every vertex $u\in V$:
\begin{itemize}
\item[\textup{(i)}] if $e_u\in\irH_0$, then $\ell^2(\Des(u))\subseteq\irH_0$ (i.e. $u \notin V' \Longrightarrow \Des(u) \subseteq V\setminus V'$),
\item[\textup{(ii)}] $e_u\in\irH_0$ if and only if $\ell^2(\Chi(u))\subseteq\irH_0$ (i.e. $u \notin V' \iff \Chi(u) \subseteq V\setminus V'$); this is fulfilled in the special case when $u$ is a leaf,
\item[\textup{(iii)}] if $e_u\in\irH_0^\perp$, then $e_{\Par^k(u)}\in\irH_0^\perp$ for every $k\in\N_0$ (i.e. $u \in V' \Longrightarrow \Par^k(u) \in V', \; \forall \; k\in \N_0$),
\item[\textup{(iv)}] the subgraph $\irT'=(V',E'=E\cap (V'\times V'))$ is a leafless subtree,
\item[\textup{(v)}] if $\irT$ has no root, neither has $\irT'$, and
\item[\textup{(vi)}] if $\irT$ has a root, then either $\Sl\in C_{0\cdot}(\ell^2(V))$ or $\roo_\irT=\roo_{\irT'}$.
\end{itemize}
\end{prop}

\begin{proof}
The facts that $\irH_0$ is invariant for $\Sl$ and that the weights are strictly positive imply (i). 

The sufficiency in (ii) is a part of (i). 
On the other hand, suppose that $\ell^2(\Chi(u))\subseteq\irH_0$. 
Then 
\[ 
\alpha_u = \lim_{n\to\infty} \sum_{w\in \Chi^n(u)}\prod_{j=0}^{n-1}\lambda_{\Par^j(w)}^2 
\]
\[ 
= \lim_{n\to\infty} \sum_{v\in\Chi(u)} \lambda_{v}^2 \sum_{w\in \Chi^{n-1}(v)}\prod_{j=0}^{n-2}\lambda_{\Par^j(w)}^2 = \sum_{v\in\Chi(u)} \lambda_{v}^2 \alpha_{v} =0, 
\]
since $\sum_{v\in\Chi(u)} \lambda_{v}^2 \sum_{w\in \Chi^{n-1}(v)}\prod_{j=0}^{n-2}\lambda_{\Par^j(w)}^2 \leq \sum_{v\in\Chi(u)} \lambda_{v}^2 \leq 1$ for all $n\in\N$ and $\sum_{w\in \Chi^{n-1}(v)}\prod_{j=0}^{n-2}\lambda_{\Par^j(w)}^2\searrow\alpha_v$. 
This proves the necessity in (ii). 
Point (iii) follows from (ii) immediately.

We have to check three conditions for $\irT'$ to be a subtree. 
Two of them is obvious since they were also true in $\irT$. 
In order to see the connectedness of $\irT'$, two distinct $u',v'\in V'$ are taken. 
Since $V = \cup_{j=0}^\infty \Des_\irT(\Par^j_\irT(u'))$, $\Par^k_\irT(u') = \Par^l_\irT(v')$ holds with some $k,l\in\N_0$. 
Then (iii) gives $\Par^i_\irT(u') = \Par^j_\irT(v') \in V'$ for every $i\leq k$ and $j\leq l$, which provides an undirected path in $\irT'$ connecting $u'$ and $v'$.

Finally via (ii) it is trivial that $\irT'$ is leafless, and the last two points immediately follow from (iii).
\end{proof}

In view of (v)-(vi), we have $\Par_{\irT}(u')=\Par_{\irT'}(u')$ for any $u'\in V'$, so we will simply write $\Par(u')$ in this case as well.

Let us take an arbitrary leafless subtree $\irT'=(V',E')$ of $\irT$ with the properties that if $\irT$ is rootless, then $\irT'$ is also rootless, and if $\irT$ has a root, then $\irT'$ has the same root. 
It is trivial that $\ell^2(V\setminus V')$ is invariant for $\Sl$. 
In the special case when it is the stable subspace, it is also hyperinvariant. 
But is it hyperinvariant for all weighted shift operators that are defined on $\irT$? 
The answer is negative as we will see from the next example.

\begin{exmpl}\label{leafless_Br1_exmpl}
\textup{Let us consider $\widetilde\irT$ and set all of the weights to be equal to 1, then $\Sl$ is a bounded weighted shift operator on $\widetilde{\irT}$. 
The unitary operator defined by the following equations: $Ue_0 = e_0$ and $U e_k = e_{k'}$, $U e_{k'} = e_k$, $U e_{-k} = e_{-k}$ for every $k\in\N$, obviously commutes with $\Sl$. 
But it is easy to see that $\ell^2(\N)$ is not invariant for $U$, hence it is not hyperinvariant for $\Sl$.}
\end{exmpl}

Now we identify the asymptotic limit $A_*$ of the adjoint $\Sl^*$. 
The stable subspace of $\Sl^*$ will be denoted by $\irH_0^*$. 
Since the weights are in the interval $]0,1]$ any infinite product is unconditionally convergent.

\begin{prop} \label{dual_aslim_prop}
If $\Sl$ is a contractive weighted shift operator on the directed tree $\irT$, then the following two points are satisfied:
\begin{itemize}
\item[\textup{(i)}] If $\irT$ has a root, then $\Sl \in C_{\cdot 0}(\ell^2(V))$.
\item[\textup{(ii)}] If $\irT$ is rootless, then $\irH_0^{*\perp} = \vee\{h_u\colon u\in V\}$ 
where 
\[ 
h_u = \sum_{v\in\Gen(u)}\prod_{j=0}^\infty\lambda_{\Par^j(v)}\cdot e_v \in\ell^2(V). 
\] 
The equality $h_u = h_v$ holds if $v\in\Gen(u)$ and every $h_u$ is an eigenvector:
\[ 
A_* h_u = a_u h_u \quad \forall u\in V 
\]
with the corresponding eigenvalue
\[ 
a_u = \|h_u\|^2 = \sum_{v\in\Gen(u)}\prod_{j=0}^\infty\lambda_{\Par^j(v)}^2. 
\] 
So, every level has one such $h_u$. 
Moreover, if $h_u$ is not zero for a vertex $u$, then it is not zero for every $u\in V$.
\end{itemize}
\end{prop}

\begin{proof}
The first statement is clear, so we only deal with (ii).
We have 
\[
\sum_{v\in\Gen_n(u)}\prod_{j=0}^{\infty}\lambda_{\Par^j(v)}^2 \leq \sum_{v\in\Gen_n(u)}\prod_{j=0}^{n-1}\lambda_{\Par^j(v)}^2 = \|\Sl e_{\Par^n(u)}\|^2 \leq 1,
\]
for all $n\in\N$. Indeed, $0\leq \sum_{v\in\Gen(u)}\prod_{j=0}^\infty\lambda_{\Par^j(v)}^2\leq 1$ which means that $h_u$ is actually a vector of $\ell^2(V)$. For $n\in\N$ we have
\[ 
\Sl ^n\Sl ^{*n}e_u = \prod_{j=0}^{n-1}\lambda_{\Par^j(u)}\cdot \Sl ^n e_{\Par^n(u)} = \prod_{j=0}^{n-1}\lambda_{\Par^j(u)}\sum_{v\in\Gen_n(u)} \prod_{j=0}^{n-1}\lambda_{\Par^j(v)}\cdot e_v. 
\]
Since $\lim_{n\to\infty}\Sl ^n\Sl ^{*n}e_u = A_*e_u$,
\[ \langle A_*e_u,e_v\rangle = \left\{\begin{matrix}
\prod_{j=0}^{\infty}\lambda_{\Par^j(u)}\prod_{j=0}^{\infty}\lambda_{\Par^j(v)} & \text{if } v\in\Gen(u) \\
0 & \text{otherwise}
\end{matrix}\right., \]
which yields 
\[ 
A_*e_u = \prod_{j=0}^{\infty}\lambda_{\Par^j(u)}\sum_{v\in\Gen(u)} \prod_{j=0}^{\infty}\lambda_{\Par^j(v)}\cdot e_v = \prod_{j=0}^{\infty}\lambda_{\Par^j(u)}\cdot h_u.
\]
Now we get
\[ 
A_*h_u = \sum_{v\in\Gen(u)}\prod_{j=0}^\infty\lambda_{\Par^j(v)}\cdot A_*e_v = \bigg(\sum_{v\in\Gen(u)}\prod_{j=0}^\infty\lambda_{\Par^j(v)}^2\bigg) h_u
\]
since $\Gen(u)=\Gen(v)$ and thus $h_u=h_v$.

To conclude the relation $\irH_0^{*\perp}\subseteq \vee\{h_u\colon u\in V\}$, fix a vector $h\in\ell^2(\Gen(u))$, $h\perp h_u$. 
Using the notation $\eta_v=\langle h,e_v\rangle $ ($v\in\Gen(u)$), the equation
\[
A_* h = \sum_{v\in\Gen(u)} \eta_v A_* e_v = \bigg(\sum_{v\in\Gen(u)} \eta_v \cdot \prod_{j=0}^{\infty}\lambda_{\Par^j(v)}\bigg) h_u = 0
\]
follows by the orthogonality of $h$ and $h_u$. 
Therefore the equation $\irH_0^{*\perp} = \vee\{h_u\colon u\in V\}$ is trivial since $h_u=0$ if and only if $a_u=0$.

Finally let us suppose that $h_u = 0$ holds for a vertex $u\in V$. 
Then $\ell^2(\Gen(u))\subseteq\irH_0^*$, and since $\irH_0^*$ is invariant for $\Sl ^*$, $\ell^2(\Gen(\Par^k(u)))\subseteq\irH_0^*$ for every $k\in \N$. 
If we set a $\tilde{w}\in\Chi(u)$, then
\[ 
a_{\tilde{w}} = \sum_{w\in\Gen(\tilde{w})}\prod_{j=0}^\infty\lambda_{\Par^j(w)}^2 = \sum_{v\in\Gen(u)}\sum_{w\in\Chi(v)}\prod_{j=0}^\infty\lambda_{\Par^j(w)}^2 
\]
\[ 
= \sum_{v\in\Gen(u)}\prod_{j=0}^\infty \lambda_{\Par^j(v)}^2\sum_{w\in\Chi(v)}\lambda_{w}^2 = \sum_{v\in\Gen(u)}\prod_{j=0}^\infty |\langle h_u, e_v \rangle|^2 \sum_{w\in\Chi(v)}\lambda_{w}^2 = 0. 
\]
So $\ell^2(\Chi(\Gen(u)))\subseteq \irH_0^*$. 
By induction, we get $\irH_0^*=\ell^2(V)$.
\end{proof}

The next step is to compute the isometric asymptote. 
We call the vertex $u$ a branching vertex if $|\Chi(u)| > 1$ and the set of all branching vertices is denoted by $V_\prec$. 
In Proposition \ref{stable_subspc_prop} we used the notation $\irT' = (V',E')$ for the subtree such that $\ell^2(V') = \irH_0^\perp$. 

We will write $S\in\irB(\ell^2(\Z))$ and $S^+\in\irB(\ell^2(\N_0))$ for the simple bilateral and unilateral shift operators, i.e.: $Se_n = e_{n+1}$, $n\in\Z$ and $S^+e_k = e_{k+1}$, $k\in\N_0$. 
From the von Neumann--Wold decomposition it is clear that the c.n.u. isometries are exactly the orthogonal sums of some simple unilateral shift operators.

\begin{prop} \label{isom_as_prop}
For such a weighted shift operator $\Sl$ on $\irT$ that is a contraction and $\Sl \notin C_{0\cdot}(\ell^2(V))$, the isometric asymptote $U\in\irB(\ell^2(V'))$ is a weighted shift operator on the subtree $\irT' = (V',E')$: $U=S_{\underline{\beta}}$, with precise weights $\beta_{v'} = \frac{\lambda_{v'}\sqrt{\alpha_{v'}}}{\sqrt{\alpha_{\Par(v')}}}$, $v'\in (V')^\circ$. 
This isometry is unitarily equivalent to one of the followings:
\begin{itemize}
\item[\textup{(i)}] $\sum_{j=1}^{\Br(\irT')+1}\oplus S^+$, if $\irT$ has a root,
\item[\textup{(ii)}] $\sum_{j=1}^{\Br(\irT')}\oplus S^+$, if $\irT$ has no root and $U$ is a c.n.u. isometry, i.e.: \\ $\sum_{v'\in\Gen_{\irT'}(u')}\prod_{j=0}^\infty\beta_{\Par^j(v')}^2 = 0$ for some (and then for every) $u'\in V'$,
\item[\textup{(iii)}] $S\oplus\sum_{j=1}^{\Br(\irT')}\oplus S^+$, if $\irT$ has no root and $U$ is not a c.n.u. isometry.
\end{itemize}
\end{prop}

\begin{proof}
Because of the condition $\Sl\notin C_{0\cdot}(\ell^2(V))$, $V' \neq \emptyset$. 
For any $u'\in V'$ we have
\[ 
U e_{u'} = \frac{1}{\sqrt{\alpha_{u'}}}\cdot UA^{1/2}e_{u'} = \frac{1}{\sqrt{\alpha_{u'}}}\cdot A^{1/2}\Sl e_{u'} 
\]
\[ 
= \frac{1}{\sqrt{\alpha_{u'}}}\cdot \sum_{v\in\Chi_\irT(u')}\lambda_{v} \cdot A^{1/2}e_{v} = \sum_{v'\in\Chi_{\irT'}(u')} \frac{\lambda_{'v} \sqrt{\alpha_{v'}}}{\sqrt{\alpha_{u'}}}\cdot e_{v'}. 
\]
This establishes that $U$ is a weighted shift operator on $\irT'$ with weights $\beta_{v'} = \frac{\lambda_{v'}\sqrt{\alpha_{v'}}}{\sqrt{\alpha_{\Par(v')}}}$ ($v'\in (V')^\circ$).

First, suppose that $\irT$ has a root. 
Then $\irT'$ has the same root as $\irT$. 
But contractive weighted shift operators on a directed tree which has a root are of class $C_{\cdot 0}$, so in this case $U$ is a unilateral shift operator.
Since the co-rank of $U$ is $\Br(\irT')+1$, we infer that $U$ and $\sum_{j=1}^{\Br(\irT')+1}\oplus S^+$ are unitarily equivalent.

Second, assume that $\irT$ has no root and $U$ is a c.n.u. isometry. 
The isometry $U$ is c.n.u. if and only if $U\in C_{\cdot 0}(\ell^2(V'))$ which happens if and only if $\sum_{v'\in\Gen_{\irT'}(u')}\prod_{j=0}^\infty\beta_{\Par^j(v')}^2 = 0$ for some (and then for every) $u'\in V'$, by Proposition \ref{dual_aslim_prop}. 
Again, the co-rank of $U$ is $\Br(\irT')$, and therefore $U$ is unitarily equivalent to $\sum_{j=1}^{\Br(\irT')}\oplus S^+$.

Finally, let us suppose that $\irT$ has no root and $\sum_{v'\in\Gen_{\irT'}(u')}\prod_{j=0}^\infty\beta_{\Par^j(v')}^2 > 0$ for every $u'\in V'$. 
Then the unitary part of $U$ acts on 
\[
(\irH_0^*(U))^\perp = \bigvee\bigg\{(0\neq) k_{u'} = \sum_{v'\in\Gen_{\irT'}(u')}\prod_{j=0}^\infty\beta_{\Par^j(v')}\cdot e_{v'}\colon u'\in V'\bigg\}.
\]
Set $u'\in V'$, then 
\[ 
U k_{u'} 
= \sum_{v'\in\Gen_{\irT'}(u')}\prod_{j=0}^\infty\beta_{\Par^j(v')}\cdot Ue_{v'}
= \sum_{v'\in\Gen_{\irT'}(u')}\sum_{w'\in\Chi_{\irT'}(v')} \prod_{j=0}^\infty\beta_{\Par^j(v')}\cdot\beta_{w'}e_{w'} 
\]
\[
= \sum_{w'\in\Gen_{\irT'}(\tilde{w}')} \prod_{j=0}^\infty\beta_{\Par^j(w')}\cdot e_{w'} = k_{\tilde{w}'}
\]
for some $\tilde{w}'\in\Chi_{\irT'}(u')$. 
Therefore we get that $U|\irH_0^*(U)$ is a simple bilateral shift operator. 
Since the co-rank of $U$ is $\Br(\irT')$, the operator $U$ is unitarily equivalent to $S\oplus\sum_{j=1}^{\Br(\irT')}\oplus S^+$.
\end{proof}

\begin{rem} \label{unitarily_eq_rem}
\textup{(i) If the directed tree $\irT$ has a root, then any isometric weighted shift operator on $\irT$ is of class $C_{\cdot 0}$, i.e.: an orthogonal sum of $\Br(\irT)$ pieces of unilateral shift operators.}

\textup{(ii) In general if we have an isometric weighted shift operator $U$ on a directed tree, then the set-up of the tree does not tell us whether $U$ is a c.n.u. isometry or not. 
To see this take a rootless binary tree, i.e.: $\#\Chi(u)=2$ $\forall\; u\in V$. 
If we set the weights $\beta_v=\frac{1}{\sqrt{2}}$ $(v\in V^\circ)$, then $S_{\underline{\beta}}$ is clearly an isometry with $\sum_{v\in\Gen_{\irT}(u)}\prod_{j=0}^\infty\beta_{\Par^j(v)}^2 = \sum_{v\in\Gen_{\irT}(u)}\prod_{j=0}^\infty \frac{1}{2} = 0$ for every $u\in V$. 
Therefore $U$ is a unilateral shift operator.}

\textup{On the other hand, if we fix a two-sided sequence of vertices: $\{u_l\}_{l=-\infty}^\infty\subseteq V$ such that $\Par(u_l)=u_{l-1}$ for every $l\in\Z$, and consider the weights}
\[ 
\beta_v = \left\{\begin{matrix}
\frac{1}{\sqrt{2}} & \text{if } v\in V\setminus \left(\cup_{l=-\infty}^\infty \Chi(u_l)\right), \\
\exp{\frac{-1}{(|l|+1)^2}} & \text{if } v = u_l \text{ for some } l\in\Z, \\
1-\beta_{u_l}^2 & \text{if } v \in \Chi(u_{l-1})\setminus\{u_l\} \text{ for some } l\in\Z,
\end{matrix} \right.
\]
\textup{we clearly get an isometric weighted shift operator on that directed tree which has non-trivial unitary part. 
In fact,}
\[ 
\sum_{v\in\Gen_{\irT}(u_l)}\prod_{j=0}^\infty\beta_{\Par^j(v)}^2 \geq \prod_{j=0}^\infty\beta_{u_{l-j}}^2 = \exp{\left(2\sum_{j=0}^\infty \frac{-1}{(|l-j|+1)^2}\right)} > 0 \quad (l\in\Z). 
\]

\textup{(iii) Suppose that the rootless directed tree $\irT$ has a vertex $u\in V$ which has the following property:}
\[ V = \big(\Des(u)\big)\bigcup\bigg(\bigcup_{k=1}^\infty \big\{\Par^k(u)\big\}\bigg). \]
\textup{If we take an isometric weighted shift operator on $\irT$ with weights $\{\beta_{v}\colon v\in V^\circ\}$, then $S_{\underline{\beta}}$ is not a c.n.u. isometry. 
Indeed, $\beta_{\Par^k(u)}=1$ for every $k\in\N_0$, and thus}
\[
\sum_{v\in\Gen(\Par^k(u))}\prod_{j=0}^\infty\beta_{\Par^{j+k}(v)}^2 = 1 > 0.
\]
\end{rem}

The above points show that two unitarily equivalent weighted shift operators on directed trees can be defined on a very different directed tree. 
Now we turn to the calculation of the isometric asymptote $U_*\in\irB((\irH_0^*)^\perp)$ of the adjoint $\Sl^*$.

\begin{prop} \label{dual_isom_as_prop}
Suppose that the contractive weighted shift operator $\Sl$ on $\irT$ does not belong to the class $C_{\cdot 0}$. 
Then $\irT$ has no root and we have
\[ 
U_* h_{u} = \frac{\sqrt{a_u}}{\sqrt{a_{\Par(u)}}} \cdot h_{\Par(u)} \quad (u\in V)
\]
where $h_u\neq 0$ for every $u\in V$. 
As a matter of fact, $U_*$ is a simple unilateral shift operator if there is a last level (i.e. $\Chi(\Gen(u))=\emptyset$ for some $u\in V$), and it is a bilateral shift operator elsewhere.
\end{prop}

\begin{proof}
If $\Sl\notin C_{\cdot 0}(\ell^2(V))$, then $h_u\neq 0$ and $a_u\neq 0$ for every $u\in V$. 
For a $u\in V$:
\[ 
U_* \frac{1}{\sqrt{a_u}}h_{u} = \frac{1}{a_u} U_*A_*^{1/2}h_u = \frac{1}{a_u} A_*^{1/2}\Sl^* h_u 
\]
\[ 
= \frac{1}{a_u} \sum_{v\in\Gen(u)}\prod_{j=0}^\infty\lambda_{\Par^j(v)} A_*^{1/2}\Sl^* e_v = \frac{1}{a_u} \sum_{v\in\Gen(u)}\prod_{j=0}^\infty\lambda_{\Par^j(v)} \lambda_v A_*^{1/2}e_{\Par(v)} 
\]
\[ 
= \frac{1}{a_u}\sum_{v\in\Gen(u)}\prod_{j=0}^\infty\lambda_{\Par^j(v)}\lambda_v \frac{\langle e_{\Par(v)},h_{\Par(v)}\rangle }{\|h_{\Par(v)}\|^2} A_*^{1/2} h_{\Par(v)} 
\]
\[ 
= \frac{1}{a_u \sqrt{a_{\Par(u)}}}\bigg(\sum_{v\in\Gen(u)}\prod_{j=0}^\infty\lambda_{\Par^j(v)}^2 \bigg) h_{\Par(u)} = \frac{1}{\sqrt{a_{\Par(u)}}} h_{\Par(u)}.
\]
One can easily see the unitary equivalence with the simple uni- or bilateral shift operator.
\end{proof}

As a consequence we obtain a characterization for contractive weighted shifts on directed trees that are similar to isometries or co-isometries. 
The related result reads as follows.

\begin{corollary}
Consider the contraction $\Sl$ which is a weighted shift operator on a directed tree. 
Then the following statements hold:
\begin{itemize}
\item[\textup{(i)}] $\Sl$ is similar to an isometry if and only if $\inf\{\alpha_u\colon u\in V\}>0$.

\item[\textup{(ii)}] $\Sl$ is similar to a co-isometry if and only if it is a bilateral weighted shift operator with $\prod_{j=-\infty}^\infty \lambda_j>0$, or it is a weighted backward shift operator with $\prod_{j=-\infty}^0 \lambda_j>0$. 
Then it is similar to the simple bilateral or the simple backward shift operator, respectively.
\end{itemize}
\end{corollary}

\begin{proof}
(i): This is a simple consequence of \cite[Proposition 3.8]{Kubrusly}, Proposition \ref{aslim_prop} and \ref{isom_as_prop}.

(ii): The similarity to a co-isometry implies the similarity of $\Sl^*$ to an isometry.
Then by Proposition \ref{dual_aslim_prop} we have that $\irT$ has no root and $\#\Chi(u)\leq 1$ for every $u\in V$.

First, suppose that $\irT$ has no leaves. 
Then clearly $\Sl$ is a weighted bilateral shift operator. 
By \cite[Proposition 3.8]{Kubrusly} we have $\prod_{j=-\infty}^\infty \lambda_j>0$ and therefore $\Sl$ is similar to $S$.

Second, assume that $\irT$ has a leaf. 
Then it has a unique leaf and trivially $\Sl^*$ is a weighted unilateral shift operator. 
Again by \cite[Proposition 3.8]{Kubrusly} we have $\prod_{j=0}^\infty \lambda_{-j}>0$ and that $\Sl^*$ is similar to $S^+$.
\end{proof}

We notice that a contractive weighted shift operator on a directed tree is similar to a unitary operator if and only if it is a bilateral shift operator of class $C_{11}$ and its asymptotic limit is invertible. 
This is a simple consequence of the previous corollary.

In what follows we prove some preliminary results in order to verify (i)--(iii) of Theorem \ref{tree_cycl_thm}. First we provide the following similarity lemma.

\begin{lem} \label{br2_laef_similarity_lem}
Consider a rootless directed tree $\irT$ with the properties $\Br(\irT) = 1$ and $\Lea(\irT) \neq \emptyset$.
Let $\Sl$ be a bounded weighted shift operator on $\irT$ (see Figure \ref{tree_1-2Lea_fig}).
 Then
$\Sl$ is similar to an orthogonal sum $W\oplus N$ where
\begin{itemize}
\item[\textup{(i)}] $W$ is a weighted backward shift operator and $N$ is a cyclic nilpotent operator acting on a finite dimensional space whenever $\#\Lea(\irT) = 2$,
\item[\textup{(ii)}] $W$ is a weighted bilateral shift operator and $N$ is a cyclic nilpotent operator acting on a finite dimensional space if $\#\Lea(\irT) = 1$ holds.
\end{itemize}
\end{lem}

\begin{proof}
Let us define the following two subspaces: $\irE := \ell^2(\Z\cap V)$ and $\irE' := \irE^\perp = \ell^2(\{1',2',3'\dots\}\cap V)$. 
Clearly, the second subspace is finite dimensional. 
Set the vectors $g_k := \big(\prod_{j=1}^k \frac{1}{\lambda_j}\big) e_k - \big(\prod_{j=1}^k \frac{1}{\lambda_{j'}}\big) e_{k'}$ for every $2\leq k\leq k_0$. 
Now we define two operators on these subspaces as follows: 
\[ 
W\colon \irE\to\irE, \quad e_n\mapsto \left\{ \begin{matrix}
\lambda_{n+1}e_{n+1} & \text{ if } n \text{ is not a leaf} \\
0 & \text{ if } n \text{ is a leaf}
\end{matrix}\right., 
\] 
which is a weighted bilateral shift operator, and
\[ 
N\colon \irE'\to\irE', \quad e_{k'}\mapsto \left\{ \begin{matrix}
\frac{\|g_{k}\|}{\|g_{k+1}\|}e_{(k+1)'} & \text{ if } k < k_0\\
0 & \text{ if } k = k_0
\end{matrix}\right.. 
\]
These are clearly bounded operators. 
We state that the following operator
\[ 
X \colon \ell^2(V) \to \ell^2(V), \quad e_{k'} \mapsto \frac{1}{\|g_k\|}g_k \; (1\leq k\leq k_0), \; e_n\mapsto e_n \; (n\in \Z\cap V)
\]
is invertible and $X(W\oplus N)^* = \Sl^* X$ is satisfied.
The boundedness of $X$ is trivial since $\vee\{e_k, e_{k'}\}$ is invariant for all $1 \leq k \leq k_0$ and $e_n$ is an eigenvector for every $n\in\Z\cap V$, the invertiblity is also obvious. 
The following equations show that $X(W\oplus N)^* = \Sl^* X$ holds as well which ends the proof:
\[ 
X(W\oplus N)^* e_n = X \lambda_{n} e_{n-1} = \lambda_{n} e_{n-1} = \Sl^* e_n = \Sl^* X e_n \quad (n \in V\cap\Z)
\]
\[ 
X(W\oplus N)^* e_{1'} = 0 = \Sl^* \Big(\frac{1}{\|g_1\|}g_1\Big) = \Sl^* X e_{1'} 
\]
\[ 
X(W\oplus N)^* e_{k'} = X \frac{\|g_{k-1}\|}{\|g_{k}\|} e_{(k-1)'} = \frac{1}{\|g_{k}\|} g_{k-1} 
\]
\[ 
= \Sl^* \Big(\frac{1}{\|g_k\|}g_k\Big) = \Sl^* X e_{k'} \quad (2\leq k\leq k_0). 
\]
\end{proof}

Let $m$ be the normalized Lebesgue measure on $\T$ and $L^2$ the Lebesgue space $L^2 = L^2 (\T)$. 
The simple bilateral shift operator $S$ can be represented as the multiplication operator by the identity function $\chi(\zeta) = \zeta$ on $L^2$. 
It is a known fact that $g\in L^2$ is cyclic for $S$ if and only if $g(\zeta) \neq 0$ a.e. $\zeta\in\T$ and $\int_\T \log|g| dm =-\infty$. 
From Lemma \ref{cyclic_denserenage_nilp_lem} it is obvious that $g$ is cyclic if and only if $Sg = \chi g$ is cyclic, but this can be obtained from the previous characterization as well.

The simple unilateral shift operator $S^+$ can also be represented as a multiplication operator by $\chi$, but on the Hardy space $H^2 = H^2(\T)$. 
Before proving the first three points of Theorem \ref{tree_cycl_thm} we show that the orthogonal sum $S\oplus S^+$ has no cyclic vectors. 
The proof uses only elementary Hardy space techniques.

\begin{prop}
The operator $S\oplus S^+$ has no cyclic vectors.
\end{prop}

\begin{proof}
Suppose that $f\oplus g\in L^2\oplus H^2$ is a cyclic vector, and let us denote the orthogonal projection onto $L^2\oplus \{0\}$ and onto $\{0\}\oplus H^2$ by $P_1$ and $P_2$, respectively. 
Then $\vee\{\chi^nf\colon n\in\N_0\} = P_1(\vee\{\chi^nf\oplus\chi^ng\colon n\in\N_0\})$ is dense in $L^2$ i.e.: $f$ is cyclic for $S$, and similarly we get that $g$ is cyclic for $S^+$ as well. 
This implies that $f(\zeta)\neq 0$ is valid for a.e. $\zeta\in\T$ and $g$ is an outer function. 
We state that $0\oplus g\notin (L^2\oplus H^2)_{f\oplus g}$.
To see this consider an arbitrary complex polynomial $p$. 
Then
\[ 
\|(pf)\oplus (pg) - 0\oplus g\|^2 = \int_\T |pf|^2+|(p-1)g|^2 dm. 
\]

One of the sets $A=p^{-1}(\{z\in\C\colon\re{z}<1/2\})$ or $A^c=p^{-1}(\{z\in\C\colon\re{z}\geq 1/2\})$ has Lebesgue measure at least 1/2. 
If $m(A)\geq 1/2$, then we obtain
\[ 
\|(pf)\oplus (pg) - 0\oplus g\|^2 \geq \int_A |(p-1)g|^2 dm = \int_A |g|^2/4 dm 
\]
\[ 
\geq \frac{1}{4}\inf\left\{\int_E |g|^2 dm\colon E\in\irL, m(E)\geq 1/2 \right\} > 0.
\]
Similarly if $m(B)\geq 1/2$, then we calculate
\[ 
\|(pf)\oplus (pg) - 0\oplus g\|^2 \geq \frac{1}{4}\inf\left\{\int_E |f|^2 dm\colon E\in\irL, m(E)\geq 1/2 \right\} > 0.
\]
Thus $S\oplus S^+$ is not cyclic.
\end{proof}

Now we are in a position to prove (i)-(iii) of Theorem \ref{tree_cycl_thm}.

\begin{proof}[Proof of the points (i)-(iii) in Theorem \ref{tree_cycl_thm}]

(i) This is an easy consequence of Lemma \ref{br2_laef_similarity_lem} and Theorem \ref{backw_cycl_thm}.

(ii) By (ii) of Lemma \ref{br2_laef_similarity_lem}, $\Sl$ is similar to $W\oplus N$. 
If $W$ has no cyclic vectors then obviously neither has $\Sl$. 
If $W$ is cyclic, then by Lemma \ref{cyclic_denserenage_nilp_lem} we can obviously see that $\Sl$ has a cyclic vector.
Since $C_{\cdot 1}$ bilateral shifts are cyclic, the other statement follows immediately.

(iii) By Proposition \ref{isom_as_prop}, the isometric asymptote $U$ of $\Sl$ is unitarily equivalent to the orthogonal sum $S\oplus S^+$ which has no cyclic vectors. 
This implies - together with Lemma \ref{is_as_forcyclem} - that neither has $\Sl$.
\end{proof}

We proceed with the examination of the case when a weighted shift operator on $\widetilde{\irT}$ is similar to an orthogonal sum of a bi- and a unilateral shift operator. 
Then we will be able to verify the last point of Theorem \ref{tree_cycl_thm}.

Let us define another bounded operator $\widetilde{W}\in\ell^2(V)$ by the following equations:
\[ 
\widetilde{W}e_{n} = w_{n+1} e_{n+1}, \qquad \widetilde{W}e_{k}' = w_{(k+1)'} e_{(k+1)'} \quad (n\in\Z, k\geq 2) 
\]
where the weights $\{w_n\colon n\in\Z\}\cup\{w_{k'}\colon k\in\N\setminus\{1\}\}$ are bounded. 
Obviously $\widetilde{W}$ is an orthogonal sum of a bi- and a unilateral weighted shift operator.

Our aim is to find out whether there exists a $\widetilde{W}$ such that it is similar to $\Sl$. 
In order to do this, we will try to find a bounded, invertible $X\in\irB(\ell^2)$ which intertwines $\Sl$ with a $\widetilde{W}$: $X\Sl = \widetilde{W}X$. 
However, we found it easier to examine the adjoint equation: $\Sl^*X^* = X^*\widetilde{W}^*$. 
We shall use the following notations:
\[ 
g_k = \prod_{j=1}^k \frac{1}{\lambda_j}\cdot e_k - \prod_{j=1}^k \frac{1}{\lambda_{j'}}\cdot e_{k'} \qquad (k\in\N). 
\]
It is easy to see that 
\[
\Sl^* g_k 
= \left\{
\begin{matrix}
g_{k-1} & \text{if } k>1\\
0 & \text{if } k=1\\
\end{matrix}
\right.
\]
and that for $k\neq l$ the vectors $g_k$ and $g_l$ are orthogonal to each other. 
We also define the following subspaces
\[ 
\irE = \vee\{e_k\colon k\in\Z\}, \quad \irE' = \vee\{e_{k'}\colon k\in\N\}, \quad \irG = \vee\{g_k\colon k\in\N\}.
\]
In the next lemma, for technical reasons, we assume that $\Sl$ is contractive.

\begin{lem}
The following two conditions are equivalent for the contractive $\Sl$:
\begin{itemize}
\item[\textup{(i)}] the positive sequence $\big\{\prod_{j=1}^k \frac{\lambda_{j'}}{\lambda_j}\colon k\in\N\big\}$ is bounded,
\item[\textup{(ii)}] $\ell^2(V) = \irE\dotplus\irG$.
\end{itemize}
\end{lem}

\begin{figure}
\includegraphics[scale=0.5]{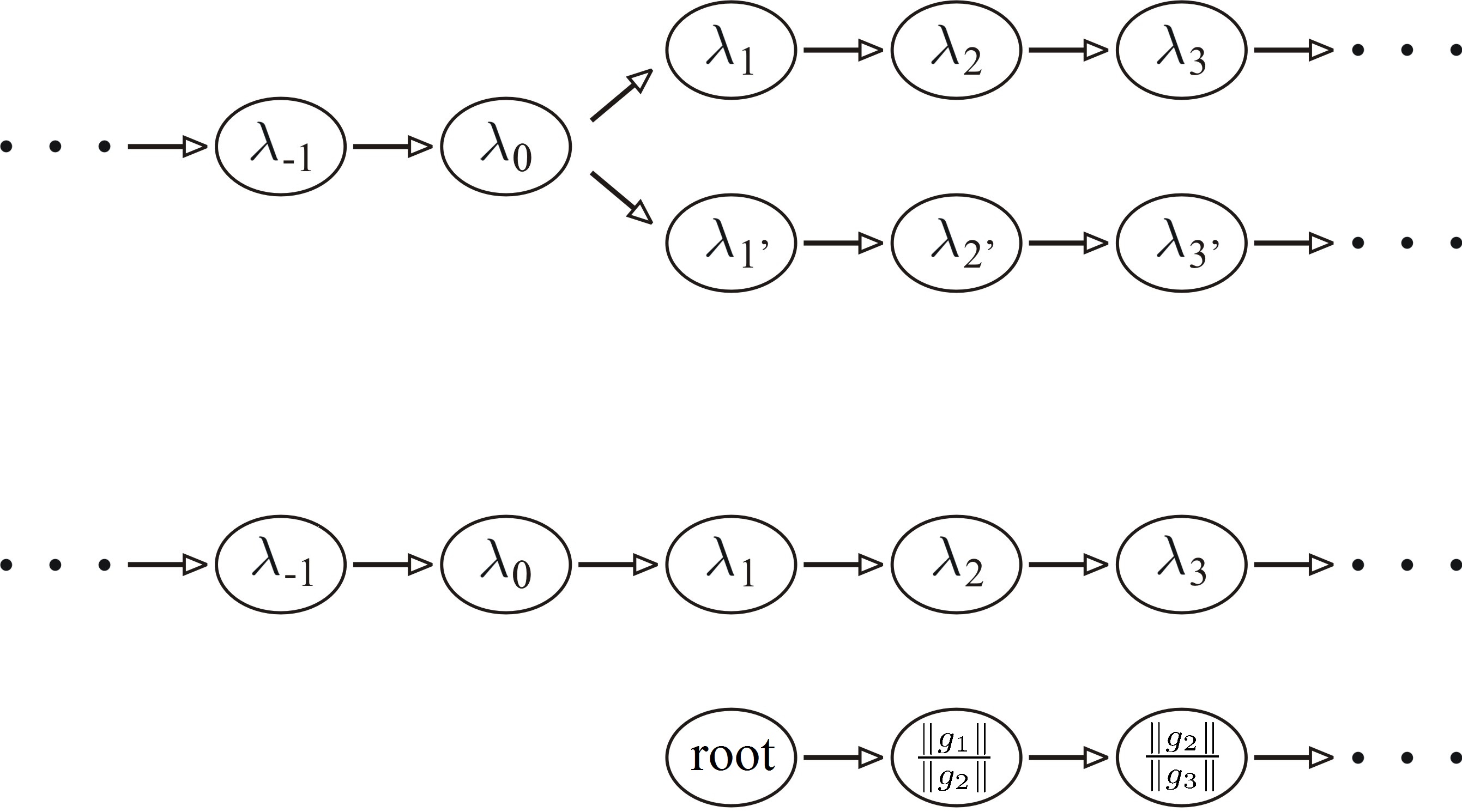}
\caption{Similarity of $\Sl$ on $\widetilde{\irT}$ to an orthogonal sum in some cases}
\end{figure}

\begin{proof}
(i)$\Longrightarrow$(ii): It is obvious that $\irE\cap\irG = \{0\}$, so we only have to prove the equation $\ell^2(V)=\irE +\irG$. 
To do this take an arbitrary vector $x\in\ell^2(V)$ and suppose that $x = e+g$ for some vectors $e\in\irE$ and $g\in\irG$. 
With the following notations
\[ 
\xi_{k'} = \langle x,e_{k'}\rangle , \quad \xi_n = \langle x,e_n\rangle , \quad \mu_k = \langle g,g_k\rangle , \quad \nu_n = \langle e,e_n\rangle , \qquad (k\in\N, n\in\Z) 
\]
we can write the following equations:
\[ 
\xi_{k'} = \langle g,e_{k'}\rangle = \Big\langle\frac{\mu_k}{\|g_k\|} g_k,e_{k'}\Big\rangle = \frac{-\mu_k}{\|g_k\|\cdot\lambda_{1'}\dots\lambda_{k'}}, 
\]
\[ 
\xi_k = \langle e,e_k\rangle + \langle g,e_k\rangle = \nu_k + \frac{\mu_k}{\|g_k\|\cdot\lambda_1\dots\lambda_k}, \quad (k\in\N) 
\]
and
\[
\xi_m=\nu_m \quad (m\in\N_0). 
\]
From them we infer that
\[ 
\mu_k = -\xi_{k'} \|g_k\| \cdot \lambda_{1'}\dots\lambda_{k'}, \qquad \nu_k = \xi_k + \frac{\lambda_{1'}\dots\lambda_{k'}}{\lambda_1\dots\lambda_k}\xi_{k'}, \qquad (k\in\N). 
\]
Therefore there exists an $e\in\irE$ and $g\in\irG$ such that $x = e+g$ holds if and only if $\sum_{k=1}^\infty |\mu_k|^2/\|g_k\|^2 < \infty$ and $\sum_{k=1}^\infty \nu_k^2 < \infty$ are fulfilled. 
But the first inequality always holds, since
\[ 
\sum_{k=1}^\infty \frac{|\mu_k|^2}{\|g_k\|^2} = \sum_{k=1}^\infty |\xi_{k'}|^2 (\lambda_{1'}\dots\lambda_{k'})^2 \leq \sum_{k=1}^\infty |\xi_{k'}|^2 \leq \|x\|^2 < \infty. 
\] 
The second one is equivalent to
\[ 
\sum_{k=1}^\infty \Big(\frac{\lambda_{1'}\dots\lambda_{k'}}{\lambda_1\dots\lambda_k}\Big)^2|\xi_{k'}|^2 < \infty, 
\]
and in particular this holds if the sequence $\big\{\prod_{j=1}^k \frac{\lambda_{j'}}{\lambda_j}\colon k\in\N\big\}$ is bounded.

(ii)$\Longrightarrow$(i): If the sequence is not bounded, then there exists a vector $x\in\ell^2(V)$ such that this sum is not finite. 
In fact, if 
\[ 
\prod_{j=1}^{k_m} \frac{\lambda_{j'}}{\lambda_j} > m \qquad (\forall \; m\in\N)
\]
hold with some increasing sequence $\{k_m\}_{m=1}^\infty$, then let $\xi_{k_m'} = m^{-3/2}$ and with this choice we have
\[ 
\sum_{k=1}^\infty \Big(\frac{\lambda_{1'}\dots\lambda_{k'}}{\lambda_1\dots\lambda_k}\Big)^2\xi_{k'}^2 > \sum_{m=1}^\infty \frac{1}{m} = \infty. 
\]
\end{proof}

Now, we are able to prove a similarity theorem. 
We say that the operator $T_1\in\irB(\irH)$ is a \textit{quasiaffine transform} of $T_2\in\irB(\irK)$ if there exists a quasiaffinity (i. e.: which is injective and has a dense range) $X\in\irI(T_1,T_2)$. 

\begin{prop}\label{sim_br_1_prop}
Let $\Sl$ be a weighted shift contraction on the directed tree $\widetilde{\irT}$ and set
\[ 
w_n = \lambda_n, (n\in\Z), \qquad w_{k'} = \frac{\|g_{k-1}\|}{\|g_k\|}, (k>1). 
\]
The following two points hold:
\begin{itemize}
\item[\textup{(i)}] $\Sl$ is always a quasiaffine transform of $\widetilde{W}$.
\item[\textup{(ii)}] If $\big\{\prod_{j=1}^k \frac{\lambda_{j'}}{\lambda_j}\colon k\in\N\big\}$ is bounded, then the weighted shift $\Sl$ is similar to $\widetilde{W}$.
\end{itemize}
\end{prop}

\begin{proof}
Since $\Sl$ is a contraction, $w_{k'} = \frac{\|g_{k-1}\|}{\|g_k\|}\leq 1$ and thus $\widetilde{W}$ is bounded. 
We define $X^*$ as follows:
\[ 
X^*e_{k'} = \frac{1}{\|g_k\|} g_k, \quad X^* e_n = e_n \qquad (k\in\N, n\in\Z). 
\]
The operator $X^*$ is bounded and quasiaffine because $\vee\{e_k,e_{k'}\}$ is invariant for $X$ for every $k\in\N$, and $\big\|X^*|\vee\{e_k,e_{k'}\}\big\|\leq 2$. 
The next equations show that $X^*$ intertwines $\widetilde{W}^*$ with $\Sl^*$:
\[ 
\Sl^* X^* e_n = \Sl^* e_n = \lambda_n e_{n-1} = \lambda_n X^* e_{n-1} = X^*\widetilde{W}^* e_n, \qquad (n\in\Z) 
\]
\[ 
\Sl^* X^* e_{k'} = \frac{1}{\|g_k\|} \Sl^* g_k = \left\{ \begin{matrix}
0 & k=1 \\
\frac{1}{\|g_k\|} g_{k-1} & k>1
\end{matrix} \right. = X^*\widetilde{W}^*e_{k'}. \qquad (k\in\N)
\]
This proves that $\Sl$ is a quasiaffine transform of $\widetilde{W}$.

If $\irE\dotplus\irG = \ell^2(V)$, then since $X^*|\irE'\in\irB(\irE',\irG)$ and $X^*|\irE\in\irB(\irE,\irE)$ are unitary transformations, therefore $X^*$ is an invertible bounded operator. 
This proves the similarity.
\end{proof}

We provide the following special case of the above proposition.

\begin{corollary}
If $\Sl \notin C_{0\cdot}(\ell^2(V))$ which is defined on $\widetilde{\irT}$, then it is similar to an orthogonal sum of a weighted bi- and a weighted unilateral shift operator.
\end{corollary}

\begin{proof}
The condition $\Sl \notin C_{0\cdot}(\ell^2(V))$ is equivalent to the positivity of $\prod_{j=1}^\infty \lambda_j$ or $\prod_{j=1}^\infty \lambda_{j'}$. 
By interchanging $\lambda_{j'}$ and $\lambda_j$ for every $j\in\N$, if necessary, we may assume that the first one is not zero. 
Then the sequence $\big\{\prod_{j=1}^k \frac{\lambda_{j'}}{\lambda_j}\colon k\in\N\big\}$ is obviously bounded. 
Proposition \ref{sim_br_1_prop} implies the similarity.
\end{proof}

We continue with the verification of (iv) of Theorem \ref{tree_cycl_thm}.

\begin{proof}[Proof of (iv) in Theorem \ref{tree_cycl_thm}]
Let us consider an $\Sl$ such that it is similar to a $\widetilde{W}$ and the bilateral summand of $\widetilde{W}$ is hypercyclic. 
By decreasing $\lambda_{k'}$-s, we may also assume that the unilateral summand is contractive. 
Take a hypercyclic vector $f\in\irE$ for the bilateral summand. 
We will show that $f\oplus e_{1'}$ is cyclic for $\widetilde{W}$ and therefore $\Sl$ is cyclic. 

First, let us take an arbitrary vector $e\in\irE$, then there is a subsequence such that $\frac{1}{k}\widetilde{W}^{j_k}f \to e \; (k\to\infty)$.
Therefore $\frac{1}{k}\widetilde{W}^{j_k}(f\oplus e_{1'}) \to e \; (k\to\infty)$ also holds, since the unilateral summand is a contraction.

Second, fix an $n\in\N_0$. 
Our aim is to prove that $\widetilde{W}^n e_{1'} \in \vee\{ \widetilde{W}^k(f\oplus e_{1'}) \colon k\in\N_0\}$. 
Since $\{\widetilde{W}^kf \colon k > n\}$ is dense in $\irE$, there is a subsequence $\{\widetilde{W}^{j_k}f\}_{k=1}^\infty$ with the property $\frac{1}{k}\widetilde{W}^{j_k}f \to \widetilde{W}^n f \; (k\to\infty)$. 
This implies that $\frac{1}{k}\widetilde{W}^{j_k}(f\oplus e_{1'}) \to \widetilde{W}^n f \; (k\to\infty)$ and hence $\widetilde{W}^n e_{1'} \in \vee\{ \widetilde{W}^k(f\oplus e_{1'}) \colon k\in\N_0\}$. 
Therefore $f\oplus e_{1'}$ is indeed a cyclic vector of $\widetilde{W}$.
\end{proof}

Let us denote the operator $\underbrace{S^+\oplus \dots \oplus S^+}_{k \text{ times}}$ by $S^+_k$ ($k\in\N$) and the orthogonal sum of $\aleph_0$ piece of $S^+$ by $S^+_{\aleph_0}$. 

We close this chapter by proving Theorem \ref{tree_adj_cycl_thm}. 
In order to do that we will apply the following theorem.

\begin{thm}\label{cyclic_ort_sum_thm}
The operator $S\oplus (S^+_k)^*$ is cyclic for every $k\in\N$.
\end{thm}

\begin{proof}
The method is the following: we intertwine $S\oplus S_k^+$ and $S$ with an injective operator $X\in\irB(L^2\oplus H^2,L^2)$: $SX = X(S\oplus S_k^+)$. 
Then taking the adjoint of both sides in the equation we get: $(S^*\oplus (S_k^+)^*)X^* = X^*S^*$.
Since $X^*$ has dense range and $S^*$ is cyclic, this implies the cyclicity of $S^*\oplus (S_k^+)^*$ for any $k \in \N$ by (i) of Lemma \ref{is_as_forcyclem}, which is unitarily equivalent to $S\oplus (S_k^+)^*$. 

For the $k=1$ case the definition of the operator $X$ is the following:
\[ 
X\colon L^2\oplus H^2 \to L^2, \quad f\oplus g\mapsto f\varphi + g, 
\]
where $\varphi\in L^\infty$, $\varphi(\zeta)\neq 0$ for a.e. $\zeta\in\T$ and $\int_\T \log|\varphi(\zeta)|d\zeta = -\infty$. 
An easy estimate shows that $X\in\irB(L^2\oplus H^2,L^2)$. 
Assume that $0 = f\varphi + g$.
If $f = 0$ ($g=0$, resp.), then $g = 0$ ($f=0$, resp.) follows immediately. 
On the other hand, taking logarithms of the absolute values and integrating over $\T$ we get 
\[
-\infty < \int_\T \log|g|dm = \int_\T \log|f| + \log|\varphi|dm \leq \int_\T |f|dm + \int_\T \log|\varphi|dm = -\infty, 
\]
which is a contradiction. 
Therefore $X$ is injective. 
The equation $SX = X(S\oplus S^+)$ is trivial, therefore $S\oplus (S^+)^*$ is indeed cyclic.

Now let us turn to the case when $k>1$. 
We will work with induction, so let us suppose that we have already proven the cyclicity of $S\oplus (S^+_{k-1})^*$ for a $k \geq 2$. 
Consider the following operator
\[ 
Y\colon L^2\oplus \underbrace{H^2 \oplus\dots\oplus H^2}_{k \text{ times}} \to L^2\oplus \underbrace{H^2 \oplus\dots\oplus H^2}_{k-1 \text{ times}}, 
\]
\[ 
f\oplus g_1\oplus \dots \oplus g_{k} \mapsto (f\varphi+g_1)\oplus g_2\oplus \dots \oplus g_{k},
\]
with the same $\varphi\in L^\infty$ as in the definition of $X$. 
Obviously $Y$ is bounded, linear and injective, and we have $Y(S\oplus S^+_{k}) = (S\oplus S^+_{k-1})Y$. 
This proves that $S\oplus (S^+_{k})^*$ is also cyclic.
\end{proof}

Of course, now a question arises naturally: Is the operator $S\oplus (S^+_{\aleph_0})^*$ cyclic? The previous thoughtline does not work for this case.
However, if the answer were positive, we could prove (ii) of Theorem \ref{tree_adj_cycl_thm} without the condition $\Br(\irT)<\infty$.

Now we are able to verify Theorem \ref{tree_adj_cycl_thm}.

\begin{proof}[Proof of Theorem \ref{tree_adj_cycl_thm}]
Obviously $\irT$ is leafless in both cases. 
The isometric asymptote $U$ of $\Sl$ is just taken. 
Since $U^*$ is cyclic, the operator $\Sl^*$ is also cyclic by Lemma \ref{is_as_forcyclem}, Proposition \ref{isom_as_prop} and Theorem \ref{backw_cycl_thm}.
\end{proof}

We would like to point out that if we take a weighted shift operator on a directed tree which is similar to a $\widetilde{W}$ and the adjoint of the bilateral summand has no cyclic vectors, then obviously we get an $\Sl$ on $\widetilde{\irT}$ such that $\Sl^*$ has no cyclic vectors. 
This shows that the adjoint of a weighted shift on a directed tree is usually not cyclic.

\newpage

\chapter*{Summary}

In this dissertation we presented results concerning asymptotic behaviour of Hilbert space power bounded operators and some applications. 
The dissertation consists of five papers. 
\cite{GeKe} has appeared in \emph{Periodica Mathematica Hungarica}, \cite{Ge_contr} has appeared in \emph{Acta Scientiarum Mathematicarum (Szeged)}, \cite{Ge_matrix} has accepted in \emph{Linear and Multilinear Algebra}, and \cite{Ge_SzN} has been accepted in \emph{Proceedings of the American Mathematical Society}.
The last one \cite{Ge_tree} is still under revision.
The papers \cite{Ge_contr,Ge_matrix,Ge_SzN,Ge_tree} were also uploaded to the arXiv.

First in Chapter \ref{A-s_chap}, we characterized all possible asymptotic limits of Hilbert space contractions. 
We proved that in a finite dimensional space the asymptotic limit of a contraction is always a projection. 
In the case when $\irH$ is separable and infinite dimensional we proved that a positive and contractive operator $A$ arises asymptotically from a contraction if and only if $A$ arises asymptotically from a contraction in uniform convergence. 
Moreover, these conditions are equivalent to the following: either $r_e(A) = 1$ or $A$ is a projection; which holds if and only if $\dim\irH(]0,1]) = \dim\irH(]\delta,1])$ is satisfied for every $0\leq\delta<1$ where $\irH(\omega)$ denotes the spectral subspace of $A$ associated with the Borel subset $\omega\subseteq \R$. 
If $\dim\irH>\aleph_0$, then we provided a similar result. 
We concluded Chapter \ref{A-s_chap} by investigating what conditions on two contractions have to be satisfied in order that their asymptotic limits coincide. 
We also examined the reverse question.

Then we turned to the investigation of $L$-asymptotic limits of power bounded matrices. 
We showed that in this case $A_{T,L}$ is independent of the particular choice of the Banach limit $L$, moreover, we have $A_{T,L} = A_{T,C} := \lim_{n\to\infty}\frac{1}{n}\sum_{j=1}^n T^{*j}T^j$. 
We called $A_{T,C}$ the Ces\`aro asymptotic limit of $T$.
In order to characterize the possible Ces\`aro asymptotic limits, we examined two separate cases. 
The first one is when $T$ is similar to a unitary operator, which is (in the matrix case) equivalent to the condition that $T$ is of class $C_{11}$. 
We proved that a positive definite matrix $A$ is the Ces\`aro asymptotic limit of a power bounded matrix $T$ if and only if the trace of $A^{-1}$ equals the dimension of the space which is equivalent to the condition that there is an invertible matrix $S$ with unit column vectors such that $A = S^{*-1}S^{-1} = (SS^*)^{-1}$ holds. 
Using the charaterization in the previous case we proved the description in the other case too, i.e. when $T$ is not similar to a unitary matrix. 
Namely, for a given singular positive semi-definite matrix $A\neq 0$ we can find a power bounded matrix $T$ such that $A = A_{T,C}$ is valid if and only if $\tr (A|(\ker A)^\perp)^{-1} \leq \dim(\ker A)^\perp$, which holds if and only if $A = S^{*-1}(I\oplus 0)S^{-1}$ is fulfilled with some invertible matrix $S$ with unit column vectors and an orthogonal projection of the form: $0\neq I\oplus 0 \neq I$.

Chapter \ref{SzN_chap} was devoted to a possible generalization of Sz.-Nagy's similarity theorem and a strengthening of it. 
We proved that similar power bounded operators $T$ and $S$ share the same asymptotic property, namely the powers of the operator $T$ and its adjoint $T^*$ are bounded from below on the orthogonal complement of the corresponding stable subspaces if and only if the same holds for $S$. 
In particular, this property is satisfied by any power bounded normal operator, and therefore by any power bounded $T$ that is similar to a normal operator.
This can be considered as a generalization of Sz.-Nagy's theorem, although only in the necessity part (the sufficiency part trivially cannot be generalized in this way).
Similarity to operators belonging to some other special classes was examined, too.
Theorem \ref{Sz-N_ref_thm} is one of the most important results concerning similarity to normal operators. 
Therefore we hope that our result will be useful when one wants to investigate similarity questions.
In the second part of Chapter \ref{SzN_chap} we described all possible asymptotic limits of those contractions that are similar to unitary operators and which act on a $\aleph_0$ dimensional space.
Namely, we proved that the positive contraction $A\in\irB(\irH)$ arises from such a contraction $T$ if and only if it is invertible, $r_e(A) = 1$ and $\dim\ker(A-\underline{r}(A)I)\in\{0,\aleph_0\}$.

Next, in Chapter \ref{comm_chap}, we examined the commutant mapping of contractions. This mapping belongs to those few links which relate the contraction to a well-understood operator (i.e. a unitary operator). 
The injectivity of this mapping was investigated. 
Of course when $T$ is of class $C_{1\cdot}$, the associated commutant mapping $\gamma_T$ is necessarily injective. 
We provided an example when the stable subspace of $T$ is non-trivial yet $\gamma_T$ is injective. 
Moreover, we provided four necessary conditions which has to be fulfilled whenever $\gamma_T$ is injective. 
One of these conditions is that $\sigma_p(T)\cap \overline{\sigma_p(T^*)} \cap\D = \emptyset$. We also proved that certain conditions on the stable component implies that the injectivity of $\gamma$ is equivalent to this previously mentioned condition. We verified that the commutant mapping of quasisimilar contractions are simultaneously injective or not. We examined the commutant mapping of orthogonal sums as well, and in particular it turned out that $\gamma_{T\oplus T}$ is injective if and only if $\gamma_T$ is injective.
Finally, we proved that the above mentioned four conditions together are still not enough to ensure that the commutant mapping is injective. Complete characterization was left open.

Finally in Chapter \ref{tree_chap}, we were interested in (contractive) weighted shifts on directed trees (with positive weights). 
This is a recently introduced class (\cite{JJS}) which is a very natural generalization of the classical weighted shift operators, and it turned out that they are very useful in testing certain problems. 
In our work first we proved that a countable orthogonal sum of backward shift operators with at most one zero weight is always cyclic. 
Our second aim was to provide cyclic properties of weighted shift operators on directed trees. 
In order to do this first, we examined their isometric asymptotes. 
Using cyclic properties of the isometric asymptote we were able to provide necessary and sufficient conditions for cyclicity in the case when the directed tree $\irT$ is rootless, its branching index is 1 and there is at least 1 leaf. 
Namely, when there are two leaves, the operator is always cyclic. 
If there is exactly one leaf, then the cyclicity of the weighted shift $\Sl$ is equivalent to the cyclicity of a bilateral shift operator which can be obtained from $\Sl$ very easily. 
There is one more case when the cyclicity of $\Sl$ is an interesting question, this is when $\Sl$ is defined on $\widetilde\irT$ which is a rootless and leafless directed tree such that the branching index is 1. 
On $\widetilde{\irT}$ we can find cyclic and non-cyclic shifts as well, but we could not provide necessary and sufficient conditions. 
In the special case when $\Sl$ is of class $C_{1\cdot}$, the weighted shift on $\widetilde\irT$ has no cyclic vectors. 
We were also interested in the cyclic properties of the adjoint of a weighted shift on a directed tree. 
We showed that not every $\Sl^*$ is cyclic. 
On the other hand, if $\Sl$ is of class $C_{1\cdot}$ and $\irT$ has a root, then $\Sl^*$ is cyclic. 
When $\irT$ is rootless and the branching index is finite, then the adjoint of every $C_{1\cdot}$ class contractive weighted shift on $\irT$ is cyclic.

\newpage

\chapter*{\"Osszefoglal\'as}

Disszert\'aci\'omban Hilbert t\'erbeli hatv\'anykorl\'atos oper\'atorok aszimptotikus viselked\'es\'et vizsg\'altam, \'es bemutattam ezen ter\"ulet n\'eh\'any lehets\'eges \'uj alkalmaz\'as\'at.
A dolgozat \"ot publik\'aci\'o eredm\'enyeit tartalmazza.
A \cite{GeKe} cikk m\'ar megjelent a \emph{Periodica Mathematica Hungarica} foly\'oiratban \'es a \cite{Ge_contr} pedig az \emph{Acta Scientiarum Mathematicarum (Szeged)} foly\'oiratban.
A \cite{Ge_matrix} publik\'aci\'om megjelen\'es alatt \'all a \emph{Linear and Multilinear Algebra} foly\'oiratban, a \cite{Ge_SzN} cikkem pedig a \emph{Proceedings of the American Mathematical Society} foly\'oiratba fogadt\'ak el.
\"Ot\"odik cikkem \cite{Ge_tree} egyel\H{o}re b\'ir\'alat alatt van.
A fenti \"ot cikk k\"oz\"ul n\'egyet \cite{Ge_contr,Ge_matrix,Ge_SzN,Ge_tree} felt\"olt\"ottem az arXiv-ra is.

A disszert\'aci\'om m\'asodik fejezet\'eben karakteriz\'altam azon pozit\'iv oper\'atorokat, melyek el\H{o}\'allnak kontrakci\'ok aszimptotikus hat\'ar\'ert\'ekeik\'ent.
Bel\'attam, hogy v\'eges dimenzi\'os t\'erben ez az aszimptotikus limesz csak ortogon\'alis projekci\'o lehet.
A szepar\'abilis, v\'egtelen dimenzi\'os esetben kider\"ult, hogy egy $A$ pozit\'iv kontrakci\'o pontosan akkor teljes\'iti a kiv\'ant felt\'etelt, ha l\'etezik olyan $T$ kontrakci\'o, melynek \"onadjung\'alt hatv\'anyai norm\'aban $A$-hoz konverg\'alnak.
Ezen felt\'etelek ekvivalensek azzal, hogy vagy $r_e(A) = 1$ vagy $A$ egy projekci\'o; \'es v\'eg\"ul ez akkor \'es csak akkor igaz, ha $\dim\irH(]0,1]) = \dim\irH(]\delta,1])$ teljes\"ul minden $0\leq\delta<1$ sz\'amra, ahol az $A$ oper\'ator $\omega\subseteq\R$ Borel halmazhoz tartoz\'o spektr\'al alter\'et $\irH(\omega)$ jel\"oli. 
Nem szepar\'abilis terek eset\'en is hasonl\'o ekvivalens felt\'etelek adhat\'oak.
A fejezet v\'eg\'en azt vizsg\'altam, hogy k\'et kontrakci\'ora milyen felt\'eteleket kell tenn\"unk ahhoz, hogy aszimptotikus hat\'ar\'ert\'ek\"uk megegyezzen. 
A ford\'itott k\'erd\'est is vizsg\'altam, azaz hogy az aszimptotikus limeszek egyenl\H{o}s\'ege mit implik\'al a kontrakci\'okra vonatkoz\'oan.

Ezut\'an hatv\'anykorl\'atos m\'atrixok $L$-aszimptotikus limesz\'et vizsg\'altam.
Az els\H{o} fontos l\'ep\'es az volt, hogy megmutattam, nem kell Banach limeszekkel foglalkoznunk, ugyanis az $A_{T,L}$ m\'atrix f\"uggetlen $L$ megv\'alaszt\'as\'at\'ol.
Kider\"ult, hogy: $A_{T,L} = A_{T,C} := \lim_{n\to\infty}\frac{1}{n}\sum_{j=1}^n T^{*j}T^j$ mindig tejes\"ul. 
Az $A_{T,C}$ oper\'atort a $T$ Ces\`aro aszimptotikus hat\'ar\'ert\'ek\'enek h\'ivjuk.
A lehets\'eges pozit\'iv szemidefinit m\'atrixok le\'ir\'as\'ahoz k\'et k\"ul\"on e\-set\-et vizsg\'altam.
Az els\H{o}, amikor a $T$ m\'atrix hasonl\'o egy unit\'er m\'atrixhoz, vagy m\'ask\'eppen mondva $C_{11}$ oszt\'aly\'u (ez az ekvivalencia term\'eszetesen csak m\'atrixok eset\'en \'all fenn).
Bebizony\'itottam, hogy egy $A$ pozit\'iv definit m\'atrix pontosan akkor Ces\`aro aszimptotikus limesze egy hatv\'anykorl\'atos $T$ m\'atrixnak, ha $\tr A^{-1}$ megyegyezik a t\'er dimenzi\'oj\'aval.
Ez pontosan akkor igaz, ha l\'etezik olyan invert\'alhat\'o $S$ m\'atrix, melynek oszlopvektorai egys\'egvektorok \'ugy, hogy $A = S^{*-1}S^{-1} = (SS^*)^{-1}$ teljes\"ul. 
Ebb\H{o}l a karakteriz\'aci\'ob\'ol vezettem le a m\'asik eset karakteriz\'aci\'oj\'at is, azaz amikor val\'odi a stabil alt\'er.
Nevezetesen egy $A\neq 0$ szingul\'aris, pozit\'iv szemidefinit m\'atrix pontosan akkor Ces\`aro aszimptotikus limesze egy hatv\'anykorl\'atos $T$-nek, ha $\tr (A|(\ker A)^\perp)^{-1} \leq \dim(\ker A)^\perp$.
Ez pedig pontosan akkor igaz, ha $A = S^{*-1}(I\oplus 0)S^{-1}$ teljes\"ul egy $0\neq I\oplus 0 \neq I$ alak\'u ortogon\'alis projekci\'oval \'es egy olyan invert\'alhat\'o $S$ m\'atrixszal, melynek oszlopvektorai mind egys\'egvektorok.

Az ezt k\"ovet\H{o} negyedik fejezetben az Sz.-Nagy f\'ele hasonl\'os\'agi t\'etel egy lehets\'eges \'altal\'anos\'it\'as\'at bizony\'itottam be, illetve bebizony\'itottam egy er\H{o}s\'it\'est is.
Bel\'attam, hogy egym\'ashoz hasonl\'o hatv\'anykorl\'atos oper\'atorok eset\'en az al\'abbi tulajdons\'ag ekvivalens: a stabil alt\'er ortogon\'alis komplementer\'en az adott oper\'ator hatv\'anyai alulr\'ol egyenletesen korl\'atosak. 
Speci\'alisan ez a tulajdons\'ag igaz norm\'alis hatv\'anykorl\'atos oper\'atorokra, s \'igy a hozz\'ajuk hasonl\'o oper\'atorokra is.
Ez tekinthet\H{o} az Sz.-Nagy t\'etel egy \'altal\'anos\'it\'as\'anak.
M\'as speci\'alis oper\'atorokhoz val\'o hasonl\'os\'agot is vizsg\'altam.
\'Ugy gondolom, ez az \'alal\'anos\'it\'as hasznos lehet k\"ul\"onf\'ele hasonl\'os\'agi probl\'em\'ak vizsg\'alata eset\'en.
A fejezet m\'asodik r\'esz\'eben le\'irtam azon pozit\'iv kontrakci\'okat, melyek egy olyan kontrakci\'ob\'ol \'allnak el\H{o} aszimptotikusan, mely hasonl\'o egy unit\'er oper\'atorhoz. 
Nevezetesen, egy $A\in\irB(\irH)$ pozit\'iv kontrakci\'o pontosan akkor \'all el\H{o} egy ilyen $T$ kontrakci\'ob\'ol, ha invert\'alhat\'o, $r_e(A) = 1$ \'es $\dim\ker(A-\underline{r}(A)I)\in\{0,\aleph_0\}$ teljes\"ul.

Az ezt k\"ovet\H{o} fejezetben kontrakci\'ok kommut\'ans lek\'epez\'es\'et vizsg\'altam. 
A $\gamma_T$ lek\'epez\'es kapcsolatot l\'etes\'it az adott $T$ kontrakci\'o kommut\'ansa \'es a $W_T$ unit\'er aszimptota kommut\'ansa k\"oz\"ott (mely j\'ol ismert). 
$\gamma_T$ injektivit\'as\'at vizsg\'altam. 
Term\'eszetesen ha $T$ a $C_{1\cdot}$ oszt\'alyhoz tartozik, akkor $\gamma_T$ sz\"uks\'egk\'eppen injekt\'iv. 
Megvizsg\'altam azt a k\'erd\'est, hogy a kommut\'ans lek\'epez\'es lehet-e injekt\'iv abban az e\-set\-ben, amikor a stabil alt\'er val\'odi.
El\H{o}sz\"or p\'eld\'at adtam ilyen kontrakci\'ora.
Ezen k\'iv\"ul n\'egy sz\"uks\'eges felt\'etelt adtam meg, melyet $T$-nek teljes\'iteni kell ahhoz, hogy $\gamma_T$ injekt\'iv legyen. 
E n\'egy felt\'etel k\"oz\"ul az egyik ekvivalens felt\'etelt ad abban az esetben, ha bizonyos egyszer\H{u} felt\'eteleket megk\"ovetel\"unk kontrakci\'onkr\'ol.
Ezen k\'iv\"ul megmutattam, hogy a kommut\'ans lek\'epez\'ese k\'et kv\'azihasonl\'o kontrakci\'onak egyszerre injekt\'iv vagy nem.
Kontrakci\'ok ortogon\'alis \"osszegeit is vizsg\'altam; p\'eld\'aul kider\"ult, hogy $\gamma_{T\oplus T}$ pontosan akkor injekt\'iv, ha $\gamma_T$ is az.
V\'eg\"ul konktr\'et p\'eld\'aval al\'at\'amasztva megmutattam, hogy a fent eml\'itett n\'egy felt\'etel egy\"utt sem ad elegend\H{o} felt\'etelt a kommut\'ans lek\'epez\'es injektivit\'as\'ara.
A teljes jellemz\'es nyitott probl\'ema maradt.

Az utols\'o fejezetben ir\'any\'itott f\'akon val\'o kontrakt\'iv eltol\'as oper\'atorokat viszg\'altam (ahol a s\'ulyok pozit\'ivak).
Ezt az oszt\'alyt 2012-ben vezett\'ek be (\cite{JJS}).
Tulajdonk\'eppen a szok\'asos s\'ulyozott eltol\'as oper\'atorok egy nagyon term\'eszetes \'altal\'anos\'it\'as\'at defini\'alt\'ak a szerz\H{o}k. 
Kider\"ult, hogy ezekkel az oper\'atorokkal bizonyos addig niytott k\'erd\'esekre egyszer\H{u} v\'alasz adhat\'o.
Legel\H{o}sz\"or bel\'attam, hogy egy szok\'asos csonk\'it\'o visszatol\'as oper\'ator pontosan akkor ciklikus, ha legfeljebb egy darab z\'er\'o s\'ulya van.
Ezt k\"ovet\H{o}en r\'at\'ertem az ir\'any\'itott f\'akon val\'o eltol\'as kontrakci\'ok izometrikus aszimptot\'ainak le\'ir\'as\'ara.
Ezt felhaszn\'alva pedig sz\"uks\'eges \'es elegend\H{o} felt\'eteleket adtam azon ir\'any\'itott f\'akon val\'o eltol\'as kontrakci\'ok ciklikuss\'ag\'ara, melyek eset\'en az adott f\'anak nincs gy\"okere, pontosan egy helyen \'agazik el, \'es ott k\'etfele, valamint van levele.
Nevezetesen, ha k\'et levele van a f\'anak, akkor azon minden eltol\'as ciklikus.
Ha pontosan egy lev\'el van, akkor a f\'an defini\'alt eltol\'as oper\'ator ciklikuss\'aga ekvivalens egy abb\'ol egyszer\H{u}en kaphat\'o k\'etir\'any\'u eltol\'as oper\'ator ciklikuss\'ag\'aval. 
M\'eg egy esetben \'erdekes a ciklikuss\'ag k\'erd\'ese, m\'egpedig amikor nincs gy\"ok\'er, nincs lev\'el, \'es az el\'agaz\'asi index pontosan 1.
Ez esetben tal\'alhat\'o olyan eltol\'as oper\'ator, mely ciklikus, \'es olyan is, mely nem ciklikus.
Azonban a ciklikuss\'ag teljes jellemz\'ese nyitott maradt.
Viszont siker\"ult bel\'atni, hogy ha ez az eltol\'as kontrakci\'o m\'eg $C_{1\cdot}$ oszt\'alybeli is, akkor nincs ciklikus vektora az adjung\'altnak.
Megvizsg\'altam az $\Sl^*$ oper\'atorok ciklikuss\'ag\'at is.
Kider\"ult, hogy nem minden $\Sl^*$ ciklikus.
Azonban ha $\Sl$ benne van a $C_{1\cdot}$ oszt\'alyban, \'es a f\'anak van gy\"okere, akkor $\Sl^*$ ciklikus. 
V\'eg\"ul bel\'attam, hogy amikor a f\'anak nincs gy\"okere, \'es az el\'agaz\'asi indexe v\'eges, akkor a rajta defini\'alt $C_{1\cdot}$-kontrakci\'ok adjung\'altjai mindig ciklikusak.

\newpage

\chapter*{Acknowledgements}

First, I would like to express my deep gratitude to my supervisor, professor L\'aszl\'o K\'erchy who gave me my first problem and during these years instructed me in developing my research.

I am also extremely grateful to professor Lajos Moln\'ar who introduced me to new research areas.

I am very grateful to my family for their everlasting support and trust in me and my work.
Without them this dissertation would not have been possible to write.
I would like to thank my fianc\'ee T\"unde Csilla Herczegh who have suggested many corrections concerning English grammar.
I owe my career choice to my father Dr. P\'al Geh\'er who started to lead me in the direction of Mathematics even in my elementary school years.
I also would like to express my deepest gratitude to my grandfather Dr. L\'aszl\'o Geh\'er who was my tutor during secondary school.
Our special relationship was formed by these lessons to a great extent, therefore his passing in 2011 meant a grievous loss to me.

Last, but not least, I would like to emphasize my thanks to my former Mathematics teacher in the secondary school P\'eter Zs\'iros and my former Biology teacher J\'ozsef Baranyai.

From July 2013, I have been supported by the "Lend\"ulet" Program (LP2012-46/2012) of the Hungarian Academy of Sciences. 

This research was also supported by the European Union and the State of Hungary, co-financed by the European Social Fund in the framework of T\'AMOP-4.2.4.A/ 2-11/1-2012-0001 'National Excellence Program'.

\newpage

\bibliographystyle{amsplain}

\begin{thebibliography}{11}

\bibitem{Ba} 
C. J. K. Batty, 
Asymptotic behaviour of semigroups of operators,
\emph{Banach Center Publications} {\bf 30}, Warszawa, 1994.

\bibitem{non-cyclic_bil_shift} 
B. Beauzamy, 
A Weighted Bilateral Shift with no Cyclic Vector,
\textit{J. Operator Theory}, \textbf{4} (1980), 287--288.

\bibitem{BK} 
H. Bercovici and L. K\'erchy, 
Spectral behaviour of $C_{10}$-contractions, 
\emph{Operator Theory Live}, Theta Ser. Adv. Math. 12, Theta, Bucharest, 2010, 17--33.

\bibitem{BJJS1} 
P. Budzy\'nski, Z. J. Jab\l onski, I. B. Jung and J. Stochel, 
Unbounded Subnormal Weighted Shifts on Directed Trees, 
\textit{J. Math. Anal. Appl.} \textbf{394} (2012), no. 2, 819--834. 

\bibitem{Ca} 
{\sc G. Cassier}, 
Generalized Toeplitz operators, restrictions to invariant subspaces and similarity problems, 
\emph{J. Operator Theory}, \textbf{53} (1) (2005) 101--140.

\bibitem{CassierFack} 
G. Cassier and T. Fack, 
Contractions in Von Neumann Algebras, 
\textit{J. Functional Analysis} \textbf{135} (1996), 297--338.

\bibitem{Co} 
J. B. Conway,
\emph{A Course in Functional Analysis}, 
New York, Springer Verlag, 1985.

\bibitem{CM} G. Corach and A. Maestripieri, 
Polar Decomposition of Oblique Projections, 
\emph{Linear Algebra Appl.}, \textbf{433} (2010), 511--519.

\bibitem{DaRo}
C. Davis and P. Rosenthal, 
Solving linear operator equations,
{\it Canad. J. Math.}, {\bf 26} (1974), 1384--1389.

\bibitem{Du} 
E. Durszt, 
Contractions as restricted shifts, 
\emph{Acta Sci. Math. (Szeged)}, \textbf{48} (1985), no. 1-4, 129--134. 

\bibitem{terElst} 
A. F. M. ter Elst, 
Antinormal operators, 
\emph{Acta Sci. Math. (Szeged)}, \textbf{54} (1990), 151--158.

\bibitem{Fi} 
L.A. Fialkow, 
A note on the range of the operator $X\mapsto AX-XB$, 
{\it Illinois J. Math.}, {\bf 25} (1981), 112--124.

\bibitem{FF} 
C. Foias and A. E. Frazho, 
{\it The Commutant Lifting Approach to Interpolation Problems}, 
Birkhauser Verlag, Basel, 1990.

\bibitem{GeKe} 
Gy.~P.~Geh\'er and L.~K\'erchy, 
On the commutant of asymptotically non-vanishing contractions, 
\emph{Period.~Math.~Hungar.}, \textbf{63} (2011), no.~2, 191--203.

\bibitem{Ge_contr} 
Gy.~P.~Geh\'er, 
Positive operators arising asymptotically from contractions, 
\emph{Acta Sci.~Math.~(Szeged)}, \textbf{79} (2013), 273--287.\\
ArXiv version: \url{http://arxiv.org/abs/1407.1278}

\bibitem{Ge_matrix} 
Gy.~P.~Geh\'er, 
Characterization of Ces\`aro- and $L$-asymptotic limits of matrices, 
\emph{Linear and Multilinear Algebra}, to appear (published online).\\
\url{http://www.tandfonline.com/doi/full/10.1080/03081087.2014.899359#.U7pVEEC8SHg}.\\
ArXiv version: \url{http://arxiv.org/abs/1407.1275}

\bibitem{Ge_SzN} 
Gy.~P.~Geh\'er, 
Asymptotic limits of operators similar to normal operators, 
\textit{Proc. Amer. Math. Soc.}, accepted.\\
ArXiv version: \url{http://arxiv.org/abs/1407.0525}

\bibitem{Ge_tree} 
Gy.~P.~Geh\'er, 
Asymptotic behaviour and cyclic properties of weighted shifts on directed trees,\\
\emph{submitted}. \\
ArXiv version: \url{http://arxiv.org/abs/1401.5927}

\bibitem{GeherL} 
L. Geh\'er, 
Cyclic vectors of a cyclic operator span the space, 
\textit{Proc. Amer. Math. Soc.}, \textbf{33}, (1972), 109--110.

\bibitem{Halmos} 
P. R. Halmos, 
\emph{A Hilbert Space Problem Book}, 
Second Edition, Springer Verlag, 1982.

\bibitem{Hu}
Z. Guang Hua, 
On cyclic vectors of backward weighted shifts. (Chinese), 
\textit{J. Math. Res. Exposition}, \textbf{4} (1984), no. 3, 1--6.

\bibitem{JJS} 
Z. J. Jab\l onski, I. B. Jung and J. Stochel, 
\emph{Weighted Shifts on Directed Trees},
Memoirs of the American Mathematical Society, Number 1017, 2012.

\bibitem{Ke_Hto} 
L. K\'erchy, 
\emph{Hilbert Terek Oper\'atorai}, 
POLYGON, Szeged, 2003.

\bibitem{Ke_gen_Toep} 
L. K\'erchy, 
Generalized Toeplitz operators, 
\emph{Acta Sci. Math. (Szeged)}, \textbf{68} (2002) 373--400.

\bibitem{Ke_isom_as} 
L. K\'erchy, 
Isometric asymptotes of power bounded operators, 
\emph{Indiana Univ. Math. J.}, \textbf{38} (1989), 173--188.

\bibitem{Ke_PAMS} 
L. K\'erchy, 
Hyperinvariant subspaces of operators with non-vanishing orbits,
\emph{Proc. Amer. Math. Soc.}, {\bf 127} (1999), 1363--1370.

\bibitem{Ke_cycl} 
L. K\'erchy, 
Cyclic properties and stability of commuting power bounded operators,
\emph{Acta Sci. Math. (Szeged)}, {\bf 71} (2005), 299--312.

\bibitem{KerchyDouglasAlg} 
L. K\'erchy, 
Quasianalytic contractions and function algebras, 
\textit{Indiana Univ. Math. J.}, \textbf{60} (2011), 21--40.

\bibitem{Ke_shifttype} 
L. K\'erchy,
Shift-type invariant subspaces of contractions,
\emph{J. Functional Analysis}, {\bf 246} (2007), 281--301.

\bibitem{KV} 
L. K\'erchy and Vu Quoc Phong, 
On invariant subspaces for power-bounded operators of class $C_{1\cdot}$,
\emph{Taiwanese J. Math.}, {\bf 7} (2003), 69--75.

\bibitem{Kubrusly} 
C. S. Kubrusly, 
\textit{An Introduction to Models and Decompositions in Operator Theory}, 
Birkh\"auser, 1997.

\bibitem{DuKu}
C. S. Kubrusly and B. P. Duggal,
Contractions with $C_{\cdot 0}$ direct summands,
Adv. Math. Sci. Appl., {\bf 11} (2001), 593--601. 

\bibitem{Ku_A_proj} 
C. S. Kubrusly, 
Contractions $T$ for which $A$ is a projection, 
\emph{Acta Sci. Math. (Szeged)}, to appear.

\bibitem{Lo} 
G. G. Lorentz, 
A contribution to the theory of divergent sequences, 
\emph{Acta Math.}, \textbf{80} (1948), 167--190.

\bibitem{Luft} 
E. Luft, 
The two-sided closed ideals of the algebra of bounded linear operators of a Hilbert space, 
\emph{Czechoslovak Math. J.}, \textbf{18} (1968), 595--605.

\bibitem{Me} 
S. Mecheri, 
A generalization of Fuglede-Putnam theorem,
\emph{J. Pure Math.}, \textbf{21} (2004), 31--38.

\bibitem{Ni}
N. K. Nikolʹskiˇi
\emph{Treatise on the Shift Operator} Spectral function theory. With an appendix by S. V. Hru\v s\v cev [S. V. Khrushch\"ev] and V. V. Peller. Translated from the Russian by Jaak Peetre. Grundlehren der Mathematischen Wissenschaften [Fundamental Principles of Mathematical Sciences], 273. Springer-Verlag, Berlin, 1986.

\bibitem{Ok} 
K. Okubo, 
The unitary part of paranormal operators, 
\emph{Hokkaido Math. J.}, \textbf{6} (1977), 273--275.

\bibitem{Pa_PF} 
P. Pagacz, 
The Putnam-Fuglede property for paranormal and *-paranormal operators,
\emph{Opuscula Math.}, \textbf{33} (2013), 565--574.

\bibitem{Pa_Wold} 
P. Pagacz,
On Wold-type decomposition, 
\emph{Linear Algebra Appl.}, \textbf{436} (2012), 3065--3071.

\bibitem{Pr} 
V. V. Prasolov, 
\emph{Problems and Theorems in Linear Algebra}, 
Translations of Mathematical Monographs, 134. AMS, Providence, 1994.

\bibitem{hypercyclic_bil} 
H. N. Salas, 
Hypercyclic weighted shifts, 
\textit{Trans. Amer. Math. Soc.}, \textbf{347} (1995), 993--1004.

\bibitem{supercyclic_bil} 
H. N. Salas, 
Supercyclicity and weighted shifts, 
\textit{Studia Math.}, \textbf{135} (1999), 55--74.

\bibitem{Shields} 
A. L. Shields, 
Weighted Shift Operators and Analytic Function Theory, 
\emph{Topics in Operator Theory}, Math. Surveys 13, Amer. Math. Soc., Providence, R. I., 1974, 49--128.

\bibitem{SzN_dil}
B. Sz.-Nagy,
Sur les contractions de l'espace de Hilbert, 
\emph{Acta Szeged.} {\bf 15} (1953), 87--92.

\bibitem{SzN_unit}
B. Sz.-Nagy,
On uniformly bounded linear transformations in Hilbert space, 
\emph{Acta Univ. Szeged. Sect. Sci. Math.}, \textbf{11} (1947), 152--157.

\bibitem{SzNF_C11}
B. Sz.-Nagy and C. Foias,
Sur les contractions de l'espace de Hilbert. VII. Triangulations canoniques. Fonction minimum. (French),
\emph{Acta Sci. Math. (Szeged)}, {\bf 25} (1964), 12--37. 

\bibitem{NFBK} 
B. Sz.-Nagy, C. Foias, H. Bercovici and L. K\'erchy, 
\emph{Harmonic Analysis of Operators on Hilbert Space}, 
Second edition, Springer, 2010.

\end{thebibliography}

\newpage

\chapter*{List of symbols}

$
\begin{matrix}

%numbers
\Z & \text{the set of integers}\\
\N & \text{the set of positive integers: } \{1,2,3,\dots\}\\
\N_0 & \text{the set of non-negative integers: } \{0,1,2,3,\dots\}\\
\R & \text{the set of real numbers}\\
[a,b] & \text{closed interval}\\
]a,b[ & \text{open interval}\\
[a,b[; ]a,b] & \text{half-open half closed intervals}\\
\C & \text{the set of complex numbers}\\
\D & \text{the unit disk of } \C\colon \; \{z\in\C\colon |z|<1\}\\
\T & \text{the unit circle of } \C\colon \; \{z\in\C\colon |z|=1\}\\
\#H & \text{cardinality of a set}\\
H^- & \text{closure of a set}\\
\overline{H} & \text{conjugate of a set } H\subseteq\C\\
H^c & \text{complement of a set}\\
\Llim_{n\to\infty} & \text{Banach limit (}L\text{-limit) of a sequence}\\ %def
\irL & \text{the set of all Banach limits}\\
\irP_\C & \text{the set of complex polynomials}\\
%operators
\|\cdot\| & \text{norm of a vector or the operator norm of an operator}\\
\langle\cdot,\cdot\rangle & \text{inner product of two vectors}\\
T^* & \text{adjoint of an operator } T\\
\sigma(\cdot) & \text{spectrum of an operator}\\
\sigma_p(\cdot) & \text{point spectrum of an operator (the set of eigenvalues)}\\
\sigma_{ap}(\cdot) & \text{approximate point spectrum of an operator}\\
r(\cdot) & \text{spectral radius of an operator}\\
\sigma_e(\cdot) & \text{essential spectrum of an operator}\\
r_e(\cdot) & \text{essential spectral radius of an operator}\\
\end{matrix}
$
\newpage
$
\begin{matrix}
\irB(\irH,\irK) & \text{the set of bounded linear operators } T\colon\irH\to\irK\\
\irB(\irH) & \irB(\irH,\irH)\\
\oplus & \text{orthogonal sum of subspaces or operators}\\
\dotplus & \text{direct sum of subspaces}\\
T|\irM & \text{restriciton of an operator } T \text{ to a subspace } \irM\\
\gamma(\cdot) & \text{reduced minimum modulus of a non-zero operator}\\ %def
\gamma=\gamma_T & \text{commutant mapping of a contraction } T\\ %def
H^\perp & \text{orthogonal complement of a set } H \text{ in a Hilbert space}\\
\vee H & \text{the generated (closed) subspace of a set } H \text{ in a Hilbert space}\\
I_\irM & \text{the identity operator on the subspace } \irM\\
\ker & \text{kernel of an operator}\\
\ran & \text{range of an operator}\\
\irI(A,B) & \text{Intertwining set of two operators } A \text{ and } B\\ %def
\{T\}' & \text{commutant of }T\in\irB(\irH)\\ %def
\irH_0 = \irH_0(T) & \text{stable subspace of a power bounded }T\\ %def
\rank & \text{rank of an operator}\\
\tr & \text{trace of a matrix}\\
A_{T,L} & \text{the }L\text{-asymptotic limit of a power bounded operator } T\\
A_T & \text{the asymptotic limit of a contraction } T\\
(X^+_{T,L},V_{T,L}) & \text{isometric asymptote of a power bounded operator } T\\
(X_{T,L},W_{T,L}) & \text{unitary asymptote of a power bounded operator } T\\
\Chi(W) & \text{The set of children of a set } W \text{ in a directed tree } \irT\\
\par(u) & \text{The parent of a vertex } u \text{ in a directed tree }\irT\\
\roo & \text{The root of a directed tree}\\
V^\circ & \text{The set of all vertices in a directed tree exept for the root}\\
\Lea(\irT) & \text{The set of leaves in a directed tree }\irT\\
\Gen_n(u) & \text{The } n \text{th generation of a vertex } u \text{ in a directed tree }\irT\\
\Gen(u) & \text{The level of a vertex } u \text{ in a directed tree }\irT\\
H^\infty & \text{The Banach space of bounded analytic functions on } \D\\
\end{matrix}
$

\end{document}